\documentclass[10pt,activeacut,amsfonts,reqno]{amsart}
\usepackage{a4wide}
\usepackage[english]{babel}
\usepackage{inputenc}
\usepackage{tikz}
\usetikzlibrary{patterns,calc} 
\usepackage{mathabx}
\usepackage{subfigure}
\usepackage{amsmath,amsthm,amsfonts,amssymb,amscd}
\usepackage{mathpazo}
\usepackage{amssymb,amsfonts}
\usepackage[all,arc]{xy}
\usepackage{enumerate}
\usepackage{mathrsfs}
\usepackage{tgtermes}
\usepackage{marginnote}
\usepackage{mathtools}
\usepackage{amsthm,thmtools,xcolor}

\DeclareMathOperator{\sign}{sign}
\usepackage[all]{xy}
\usepackage{tikz-cd}
\usetikzlibrary{cd}
\usepackage[square,sort,comma,numbers,nonamebreak]{natbib}
\usetikzlibrary{decorations.pathreplacing}

\theoremstyle{plain}
\newtheorem{thm}{Theorem}[section]
\newtheorem{assumption}[thm]{Assumption}

\newtheorem{lem}[thm]{Lemma}
\newtheorem{cor}[thm]{Corollary}
\newtheorem{prop}[thm]{Proposition}

\newtheorem{lemx}{Lemma}

\theoremstyle{definition}
\newtheorem{defn}[thm]{Definition}

\theoremstyle{remark}

\numberwithin{equation}{section}
\usepackage{color}
\usepackage{xcolor}
\usepackage[colorlinks = true,
linkcolor = blue,
urlcolor  = blue,
citecolor = blue,
anchorcolor = blue]{hyperref}
\usepackage{hyperref}
\numberwithin{equation}{section}

\newcommand{\vertiii}[1]{{\left\vert\kern-0.25ex\left\vert\kern-0.25ex\left\vert #1 
    \right\vert\kern-0.25ex\right\vert\kern-0.25ex\right\vert}}
\newcommand{\spn}{\operatorname{span}}

\title[Monotonicity formula]{A Monotonicity Formula for the Fractional Laplacian and Instability Results for the Shrira Equation}

\author[A. J. Méndez ]{A. J. Méndez}
\address{Universidade Estadual de Campinas, Rua Sérgio Buarque de Holanda, 651, cep 13083-859, Campinas, S\~{a}o Paulo, Brasil.}
\email{ajmendez@ime.unicamp.br}
\thanks{A. J. Méndez was supported by the National Council for Scientific and Technological Development (CNPq, Brazil) through the Research Productivity Fellowship, grant No.~304088/2024-2, and by the São Paulo Research Foundation (FAPESP-Fundação de Amparo à Pesquisa do Estado de São Paulo), grant No.~2024/20513-7.}
\author[O. Ria\~no]{Oscar Ria\~no}
\thanks{Oscar Ria\~no was supported by the São Paulo Research Foundation (FAPESP-Fundação de Amparo à Pesquisa do Estado de São Paulo), grant No.~2024/20513-7}
\address{Departamento de Matem\'aticas, Universidad Nacional de Colombia, Bogot\'a, Colombia}
\curraddr{}
\email{ogrianoc@unal.edu.co}
\date{\today}

\begin{document}
\begin{abstract}
By employing a new class of pseudo-differential operators introduced in a previous work, we establish a novel monotonicity formula for the fractional Laplacian $|\nabla_x|^\alpha$ in $\mathbb{R}^n$, with $n \geq 2$ and $\alpha \in [1,2)$,  This framework enables us to localize our analysis to specific regions of Euclidean space where monotonicity properties of the $L^2$-mass of solutions to dispersive equations with fractional dispersion are preserved.

As an application, we focus on the Shrira equation, proving conditional instability results for its traveling wave solutions in the critical regime. We deduce key virial-type estimates that govern the long-time behavior of the $L^2$-mass. As a consequence, we establish instability results without requiring pointwise decay estimates employed in previous works. Our approach provides a robust and flexible method for monotonicity techniques to higher dimensions, shedding light on the delicate interplay between nonlocal dispersion and spatial localization in nonlinear dispersive models.
\end{abstract}
	\maketitle
%	\tableofcontents
\section{Introduction}
We consider the higher-dimensional \emph{fractional Zakharov–Kuznetsov (fZK)} equation 
\begin{equation}\label{E:fZK}
	\partial_t u - \partial_{x_1} |\nabla_x|^{\alpha} u + \frac{1}{2} \partial_{x_1}(u^2) = 0, \quad (x_1,\dots,x_n)\in \mathbb{R}^n, \quad t\in \mathbb{R}, \quad \alpha\in [1,2),
\end{equation}
where $|\nabla_x|^{\alpha}$ denotes the \emph{fractional Laplacian}, whose symbol in Fourier space is given by
\begin{equation*}
	\widehat{(|\nabla_x|^{\alpha}f)}(\xi) := |\xi|^{\alpha} \widehat{f}(\xi), \quad \xi \in \mathbb{R}^n,\quad f \in \mathcal{S}(\mathbb{R}^n).
\end{equation*}
The model \eqref{E:fZK} arises in various physical and mathematical contexts, where the flexibility of the parameter $\alpha$ and the spatial dimension $n$ allows for the analysis of the interplay between dispersive and nonlinear effects in multiple variable regimes. Regarding well-posedness, we refer to \cite{HerrIoneKeniKoch2010,ABFS1989,KillpLaurensVisan2024,Schippa2020,HickmanLinRia2019,FonsLinaresPonce2013,KenigPonceVega1991,molinetPilodBO,KenigIonescu2007,IfrimTata2019}. For further developments, see \cite{Angulo2018,kenig2011local,RianoRoudenkoYang2022,KenigMartel2009,RianoRoudenko2022,Mendez2020,KenigPilodPonceVega2020,Eychenne2023,Mendez2023} and references therein. 

When $\alpha=2$, one has that $|\nabla_x|^{2}=-\Delta$, the standard Laplacian, then in one spatial dimension the equation \eqref{E:fZK} reduces to the classical \emph{Korteweg-de Vries (KdV)} equation, while for dimensions $n \geq 2$, it coincides with the \emph{Zakharov-Kuznetsov (ZK)} equation (see \cite{ZakharovKuznet1974}).

Referring to special solutions of \eqref{E:fZK}, the solitary-wave solutions for the equation \eqref{E:fZK}  are of the form 
\begin{equation}\label{travequ}
u(x_1,\dots,x_n,t)=Q_c(x_1-c\,t,\dots,x_n), 
\end{equation}
where $c>0$ denotes the \emph{speed of propagation}. One has the scaling factor $Q_c(x)=c Q(c^{1/\alpha} x)$, $x\in \mathbb{R}^n$, where $Q$ is a real-valued vanishing at infinity solution of the elliptic equation
\begin{equation}\label{EQ:groundState}
Q+|\nabla_x|^{\alpha}Q-\tfrac{1}{2}Q^{2}=0, \quad x\in \mathbb{R}^n.
\end{equation}
The results in \cite{RianoRoudenko2022,Maris2002,FrankLenzmann2013,FrankLenzmannSilvestre2016,EychenneValet2023} establish the existence of a unique (up to translation) positive radial solution $Q$ to \eqref{EQ:groundState} such that $
	Q \in H^{\infty}(\mathbb{R}^n) := \bigcap_{s \in \mathbb{N}} H^s(\mathbb{R}^n),$
and $Q$ satisfies the polynomial decay estimate
\begin{equation}\label{poldecayGS}
	\frac{C_1}{1 + |x|^{n+\alpha}} \leq Q(x) \leq \frac{C_2}{1 + |x|^{n+\alpha}},\quad x\in\mathbb{R}^{n}
\end{equation}
 and some constants $C_2 \geq C_1 > 0$. Moreover, for any multi-index $\beta$, the derivatives of $Q$ satisfy the decay estimate
\begin{equation}\label{Qdecay}
	|\partial_x^{\beta} Q(x)| \lesssim \langle x \rangle^{-n-\alpha - |\beta|},\quad \mbox{where}\quad \langle x \rangle := (1 + |x|^2)^{\frac{1}{2}}.
\end{equation}
 Our primary goal in this work is to develop a monotonicity formula specifically adapted to the fZK equation \eqref{E:fZK} and to apply this formula to investigate the instability of solitary wave solutions.
 
Real-valued solutions of \eqref{E:HBO} formally conserve the $L^2$-norm (\emph{mass}),
\begin{equation}\label{E:mass}
	M[u(t)] := \int_{\mathbb{R}^2} |u(x,t)|^2 \, \mathrm{d}x = M[u(0)],
\end{equation}
the \emph{energy},
\begin{equation}\label{E:energy}
	E[u(t)] := \frac{1}{2} \int_{\mathbb{R}^2} \big|\, |\nabla_x|^{\frac{\alpha}{2}} u(x,t)\big|^2 \, \mathrm{d}x - \frac{1}{6} \int_{\mathbb{R}^2} u^3(x,t) \, \mathrm{d}x = E[u(0)],
\end{equation}
and 
\begin{equation}\label{E:l1inte-g}
	\int_{\mathbb{R}^2} u(x,t) \, \mathrm{d}x = \int_{\mathbb{R}^2} u(x,0) \, \mathrm{d}x.
\end{equation}
To the best of our knowledge, no further conserved quantities are known for \eqref{E:HBO}.

The conservation of the $L^{2}$-mass in \eqref{E:mass} is of particular interest. Although it is preserved in time, its spatial distribution is not well understood; in particular, it is unknown whether there exist regions in space where the $L^{2}$-mass concentrates or tends to decay. This question is central when attempting to describe the long-time behavior of solutions.

In various applications, energy estimates for equation \eqref{E:fZK} involving weighted functions that mimic the decay of solitary waves play a fundamental role. Let us provide a detailed discussion to give a precise description of the problem at hand and to identify natural obstructions arising from the operator in \eqref{E:fZK}. Specifically, we consider the localized functional
\begin{equation}\label{funct}
	\mathcal{I}(t) = \int_{\mathbb{R}^n} u^{2}(x,t) \, \varphi(x,t) \, \mathrm{d}x,
\end{equation}
where the weight function $\varphi$ is smooth and designed to reflect the decay properties of the ground state solution $Q$ described in \eqref{travequ}--\eqref{Qdecay} (see \eqref{p2} below for its precise definition). Estimates for $\mathcal{I}(t)$ have led to the development of monotonicity formulas, which play a central role in addressing questions of stability and instability of solitary waves; see, for instance, \cite{KenigMartel2009, kenig2011local, MartelMerle2000, CoteMunoPilodSimps2016, FHRY2023, FarahHolmerRoudenko2019}.
	
The first obstacle we face when differentiating the functional \eqref{funct} is handling the term
\begin{equation}\label{problem}
\int_{\mathbb{R}^{n}} u\,\partial_{x_{1}}|\nabla_{x}|^{\alpha}u\,\varphi\,\mathrm{d}x.
\end{equation}
At this point, a significant analytical challenge in our setting arises from the nonlocal nature of the operator $|\nabla_x|^{\alpha}$, with $\alpha \in [1,2)$. Unlike the classical Laplacian, which permits standard integration by parts and provides local control in Sobolev spaces, the fractional Laplacian lacks a pointwise differential structure. Thus, many standard tools do not directly apply in this setting. In particular, this makes it especially challenging to control nonlinear interactions, which typically rely on localized energy techniques.
	
To overcome these difficulties, we adopt the strategy of rewriting \eqref{problem} as the commutator
\begin{equation*}
\int_{\mathbb{R}^{n}} u\,\partial_{x_{1}}|\nabla_{x}|^{\alpha}u\,\varphi\,\mathrm{d}x = -\frac{1}{2} \int_{\mathbb{R}^{n}} u\left[\partial_{x_{1}}|\nabla_{x}|^{\alpha};\,\varphi\right]u\,\mathrm{d}x.
\end{equation*}
This somehow natural approach has been used in various contexts, e.g., in the one-dimensional setting in \cite{kenig2011local}. Indeed, we extend the ideas developed in \cite{kenig2011local} to higher-dimensional settings.
	
In this context, we establish a monotonicity formula specifically tailored to the equation under consideration. This identity serves as a central analytical tool in the implementation of classical energy estimates and represents one of the main contributions of the present work. 
\begin{lem}\label{takamajaka}
Let $\alpha \in [1,2)$, $\gamma \in \left(\frac{1}{2}, \frac{\alpha + 1}{2}\right]$, and $\omega \in \mathbb{R}$. Consider 
$\sigma = (\sigma_1, \dots, \sigma_n) \in \mathbb{R}^n$ satisfying the condition
\begin{equation}\label{sigmacond}
\sigma_1 > \sqrt{\frac{\alpha}{2(1+\alpha)}} \, \sqrt{\sigma_2^2 + \cdots + \sigma_n^2}.
\end{equation}
Define the functions
\begin{equation}\label{p2}
\varphi(x) := \int_{-\infty}^{\sigma \cdot x} \frac{\mathrm{d}s}{\langle s + \omega \rangle^{2\gamma}},
\end{equation}
and
\begin{equation}\label{p1}
\phi(x) := \frac{1}{\langle \sigma \cdot x + \omega \rangle^{\gamma}},
\end{equation}
where $x = (x_1, \dots, x_n) \in \mathbb{R}^n.$ Then, there exist positive constants $c_1$, $c_2$, depending only on $\alpha$, $n$, $\sigma$, and $\gamma$, such that for every $u \in \mathcal{S}(\mathbb{R}^n)$ the following inequality holds:
\begin{equation}\label{main1}
\int_{\mathbb{R}^n} u\, \varphi \, \partial_{x_1} |\nabla_x|^\alpha u \, \mathrm{d}x
\leq - c_1 \int_{\mathbb{R}^n} \big| |\nabla_x|^{\frac{\alpha}{2}} (u \phi) \big|^2 \, \mathrm{d}x + c_2 \int_{\mathbb{R}^n} u^2 \phi^2 \, \mathrm{d}x.
\end{equation}
In particular,
\begin{itemize}
\item[(i)] If $n=1$, then $c_1 = \frac{(\alpha+1) }{2}\sigma_1$.
\item[(ii)] For $n \geq 2$, if $\sigma_j = 0$ for all $2 \leq j \leq n$, one can take $c_1 = \frac{\sigma_1}{2}$.
\end{itemize}
\end{lem}
Setting $n=1$, the above result agrees with those in \cite[Lemmas 6]{kenig2011local}, \cite[Lemma 3]{KenigMartel2009}. Thus, Lemma \ref{takamajaka} can be regarded as a higher-dimensional version of these results. For similar estimates for the fractional Laplacian in one dimension, we refer to \cite{EychenneValet2023II,EychenneValet2023}. 

Lemma \ref{takamajaka} is useful for analyzing the decay properties of solutions of the fZK in certain regions controlled by the soliton type function $\varphi(x)$. In particular, we believe that the results described in Lemma \ref{takamajaka} can be applied to studying the large time behavior of solutions, such as the stability and instability of solitary wave solutions of \eqref{E:fZK} in dimensions $n\geq 2$ or related models. As an application, we consider the case $n=2$, $\alpha=1$ in \eqref{E:fZK}, which corresponds to Shrira's equation (see equation \eqref{E:HBO} below), a $L^2$-critical model with quadratic nonlinearity.

We remark that condition \eqref{sigmacond} seems to be related to fast decay regions of solutions of fZK, see the numerical studies in \cite{RianoRoudenkoYang2022}. The identity \eqref{main1} suggests that decay of the $L^2$-mass of solutions to fZK should be concentrated primarily in the region $\left\{\sigma \cdot x - c t \geq \beta\right\}$, where $\sigma\in\mathbb{R}^n$ satisfies the geometric conditions in \eqref{sigmacond} for some $\beta \in \mathbb{R}$ and $c >0$. A similar phenomenon of concentration occurs for the ZK equation ($\alpha=2$ in \eqref{E:fZK}), we refer to \cite{FHRY2023,mendez2025decay,CoteMunoPilodSimps2016,MMPP2021}. 
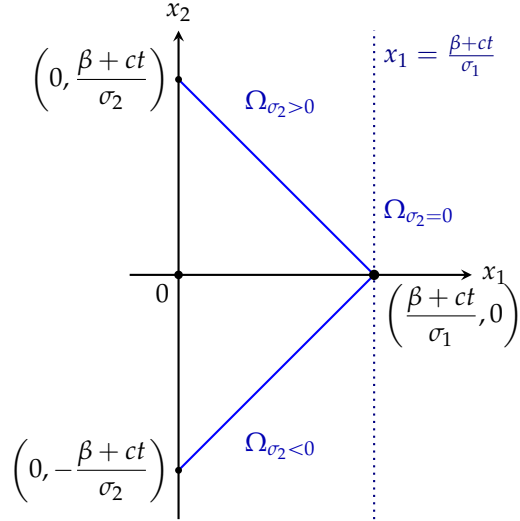
\begin{figure}[htbp]
\centering
    \begin{tikzpicture}[scale=0.65, >=stealth]
%Puntos originales
\coordinate (O) at (0,0);
\coordinate (P1) at (0,4);
\coordinate (P2) at (0,-4);
\coordinate (Q)  at (4,0);
% Ejes
\draw[->, thick, black] (-1,0) -- (6,0) node[right] {$x_1$};
\draw[->, thick, black] (0,-5) -- (0,5) node[above] {$x_2$};
\draw[thick, color=blue] (P1) -- (Q);
\draw[thick, color=blue] (P2) -- (Q);
% Línea vertical azul oscuro punteada
\draw[thick, dotted, color=blue!50!black] (4,-5) -- (4,5);
% Texto de la línea vertical
\node[color=blue!50!black, right] at (4,4.5) {$x_1 = \frac{\beta + ct}{\sigma_1}$};
% Puntos marcados
\fill[black] (P1) circle (2pt) node[left] {$\displaystyle \left(0, \frac{\beta + ct}{\sigma_2}\right)$};
\fill[black] (P2) circle (2pt) node[left] {$\displaystyle \left(0, -\frac{\beta + ct}{\sigma_2}\right)$};
\fill[black] (Q) circle (3pt) node[below right] {$\displaystyle \left(\frac{\beta + ct}{\sigma_1}, 0\right)$};
% Origen 
\fill[black] (O) circle (2.5pt);
\node[below left] at (O) {$0$};
% Etiquetas para las regiones
\node[blue!70!black, left] at (3,3.5) {$\Omega_{\sigma_{2}>0}$ };
\node[blue!70!black, left] at (3,-3.5) {$\Omega_{\sigma_{2}<0}$ };
\node[blue!70!black, above right] at (4,0.8) {$\Omega_{\sigma_{2}=0}$ };
\end{tikzpicture}
\caption{Diagram representing the regions of space associated with the $L^2$-mass behavior of solutions of \eqref{E:fZK} in $\Omega_1 = \left\{ x = (x_1, x_2) \in \mathbb{R}^2 : \sigma \cdot x - ct \geq \beta \right\}$. Depending on the sign of $\sigma_2$, we consider three different sub-regions of $\Omega_1$: $\Omega_{\sigma_2 < 0} = \left\{ x \in \mathbb{R}^2 : \sigma \cdot x - ct > \beta \right\},$ $\Omega_{\sigma_2 = 0} = \left\{ x \in \mathbb{R}^2 : \sigma \cdot x - ct = \beta \right\},$ $\Omega_{\sigma_2 > 0} = \left\{ x \in \mathbb{R}^2 : \sigma \cdot x - ct > \beta \right\}.$ Each of these regions captures distinct dynamics of the $L^2$-mass of solutions.}
\label{fig:L2mass}
\end{figure}			
In the particular case where the weight function takes the form $\varphi(\sigma \cdot x-ct)$ as in \eqref{p2}, with $\sigma$ satisfying condition \eqref{sigmacond} and parameters $\beta, c > 0$, we can extract meaningful information about the decay of the $L^2$-mass of solutions in different regions of space. For example, when $n = 2$, in Figure~\ref{fig:L2mass} we observe that the $L^2$-mass of the solution displays different behavior. This is a key observation to obtain our instability result. 

%%%%%%%%%%%%%%%%%%%%%%%%%%%%%%%%%%%%%%%%%%%%%%%%%%%%%%%%%%%%%%%%%%%%%%%%%%%%%%%%%%%%%%%%%%%%%%%%%%%%%%%%%%%%%%%%%%%%%%%%%%%%%%%%%%%

\subsection{Applications to instability for the Shrira equation}
Using the identity $\mathcal{R}_1 = -\partial_{x_1} |\nabla_x|^{-1}$, where $\mathcal{R}_1$ denotes the \emph{Riesz transform} in the first coordinate, we observe that the equation \eqref{E:fZK} with $n=2$, $\alpha=1$ agrees with the \emph{Shrira equation}
\begin{equation}\label{E:HBO}
\partial_t u- \mathcal{R}_1 \Delta u+  \frac{1}{2}\partial_{x_1}(u^2)=0, \qquad (x_1,x_2)\in \mathbb{R}^2, \, \, t\in \mathbb{R}.
\end{equation}
The model \eqref{E:HBO} was introduced by Shrira in \cite{S1989} to describe two-dimensional long-wave perturbations in a boundary-layer type shear flow. From a mathematical point of view, based on the fact that in one dimension the Fourier multiplier associated with the Riesz transform coincides with that of the Hilbert transform operator, \eqref{E:HBO} (or \eqref{E:fZK} with $n=2$, $\alpha=1$) can be regarded as a two-dimensional extension of the \emph{Benjamin-Ono} equation (BO)
\begin{equation}\label{E:BO}
\partial_t u- \mathcal{H}\partial_{x_1}^2 u+ \frac{1}{2}\partial_{x_1}(u^2)=0, \qquad x, t\in\mathbb{R},
\end{equation}
where $\mathcal{H}$ stands for  the \emph{Hilbert transform}, which is defined as $\widehat{\mathcal{H}f}(\xi)=-i\sign(\xi)\widehat{f}(\xi)$.  

Building on the monotonicity formula in Lemma~\ref{takamajaka}, and incorporating ideas from \cite{MartelMerle2001,FarahHolmerRoudenko2019}, we show that solitary waves of \eqref{E:HBO} are orbitally unstable. We begin by introducing additional notation and preliminaries.

Focusing on the Shrira equation, corresponding to the case $\alpha=1$ and $n=2$ in \eqref{EQ:groundState}, the solitary wave solutions to \eqref{E:HBO} take the form $Q_c(x_1,x_2) = c\, Q(c x_1, c x_2)$, with $c > 0$, where $Q$ satisfies
 the decay conditions \eqref{poldecayGS} and \eqref{Qdecay} (with $\alpha=1$, $n=2$).

We note thate quation \eqref{E:HBO} is invariant under the $L^2$-critical scaling
\begin{equation}\label{scaling1}
	u_\lambda(x,t) := \lambda\, u(\lambda x, \lambda^2 t), \quad \lambda > 0,
\end{equation}
for which
\begin{equation*}
	\|u_\lambda(\cdot,t)\|_{L^2} = \|u(\cdot,\lambda^2 t)\|_{L^2}.
\end{equation*}
Concerning well-posedness in $H^{s}-$spaces, we recall that the results in \cite[Theorem 4.1 and Corollary A.1]{HickmanLinRia2019} establish that the initial value problem associated to \eqref{E:HBO} can not be solved by a Picard iterative method implemented on its integral formulation for any initial data in the Sobolev space $H^s(\mathbb{R}^2)$, $s\in \mathbb{R}$. Thus, using compactness methods, in \cite{Schippa2020}, it was proved that \eqref{E:HBO} is locally well-posed in $H^s(\mathbb{R}^2)$ whenever $s>3/2$. As far as we know, there are no results concerning local well-posedness for lower regularities $s\leq \frac{3}{2}$, or global in-time solutions. We refer to \cite{CunhaRiano2025,Riano2020} for results concerning well-posedness in weighted Sobolev spaces. 

Next, we consider the definitions of orbital stability and instability. For fixed $c > 0$ and $\omega > 0$, it is defined a \emph{neighborhood (or tube) of radius $\omega$ around $Q_c$} modulo translations by
\begin{equation}\label{tubdef}
	U_{\omega} = \left\{ u \in H^{\frac{1}{2}}(\mathbb{R}^2) : \inf_{z \in \mathbb{R}^2} \| u(\cdot) - Q_c(\cdot + z) \|_{H^{\frac{1}{2}}} \leq \omega \right\}.
\end{equation}
We can now state the definitions of orbital stability and instability.
\begin{defn}\label{orbStabdef}
	Let $c > 0$. We say that $Q_c$ is \emph{orbitally stable} in $H^{\frac{1}{2}}(\mathbb{R}^2)$ if for every $\omega > 0$ there exists $\delta > 0$ such that whenever the initial data satisfies $u_0 \in U_{\delta}$, the solution $u(t)$ to \eqref{E:HBO} with initial condition $u(\cdot,0) = u_0$ remains in $U_{\omega}$ for all time $t \in \mathbb{R}$. In other words,
	\begin{equation*}
		\sup_{t \in \mathbb{R}} \inf_{z \in \mathbb{R}^2} \| u(\cdot, t) - Q_c(\cdot + z) \|_{H^{\frac{1}{2}}} < \omega.
	\end{equation*}
\end{defn}
\begin{defn}\label{orbInstabdef}
	Let $c > 0$. We say that $Q_c$ is \emph{unstable} if there exists some $\omega > 0$ such that for every $\delta > 0$, if $u_0 \in U_{\delta}$, then the corresponding solution $u(t)$ to \eqref{E:HBO} satisfies
	\begin{equation*}
	\exists t_0 = t_0(u_0) > 0 \quad \text{such that} \quad u(t_0) \notin U_{\omega}.
	\end{equation*}
\end{defn}
As mentioned earlier, the Cauchy problem for \eqref{E:HBO} is known to be locally well-posed in $H^s(\mathbb{R}^2)$ for $s > \frac{3}{2}$. \emph{Thus, to work in the energy space $H^{\frac{1}{2}}(\mathbb{R}^2)$, we need to assume existence, uniqueness, continuous dependence, and conservation laws hold in that space:}
\begin{assumption}\label{Assm1}
	Given initial data $u_0 \in H^{\frac{1}{2}}(\mathbb{R}^2)$, there exists a unique solution $u \in C([0,T]; H^{\frac{1}{2}}(\mathbb{R}^2))\cap \dots$ to \eqref{E:HBO}. Moreover, this solution can be approximated in the $C([0,T]; H^{\frac{1}{2}}(\mathbb{R}^2))$ topology by smooth solutions $u^{(n)} \in C([0,T]; H^\infty(\mathbb{R}^2))$. Additionally, the mass and energy are conserved, i.e.,
	\begin{equation*}
	M[u(t)] = M[u_0] \quad \text{and} \quad E[u(t)] = E[u_0].
	\end{equation*}
\end{assumption}

Our second main result establishes the orbital instability of solitary wave solutions to \eqref{E:HBO}:

\begin{thm}[$H^{\frac{1}{2}}$-conditional instability of $Q$]\label{MainInsta} Let $\psi_0$ denote the positive, radially symmetric eigenfunction corresponding to the unique negative eigenvalue of the linearized operator $L$ defined in \eqref{linearizedop} (see Theorem \ref{propL} below). Consider Assumption \ref{Assm1}, there exists $\omega_0 > 0$ such that for any $\delta > 0$, one can find initial data $u_0 \in H^{\frac{1}{2}}(\mathbb{R}^2)$ satisfying
	\begin{equation}\label{condthm}
		\| u_0 - Q \|_{H^{\frac{1}{2}}} \leq \delta \quad \text{and} \quad u_0 - Q \perp \{ \partial_{x_1} Q, \partial_{x_2} Q, \psi_0 \},
	\end{equation}
	and a finite time $t_0 = t_0(u_0)$ such that the corresponding solution $u(t)$ of \eqref{E:HBO} with initial data $u_0$ satisfies
	\begin{equation*}
	u(t_0) \notin U_{\omega_0},
	\end{equation*}
	or equivalently,
	\begin{equation*}
		\inf_{z \in \mathbb{R}^2} \| u(\cdot, t_0) - Q(\cdot - z) \|_{H^{\frac{1}{2}}} \geq \omega_0.
	\end{equation*}
\end{thm}
Theorem \ref{MainInsta} demonstrates that solitary waves are unstable under Assumption \ref{Assm1}.  
This is a significant result, despite being a conditional statement, because it sets techniques for two-dimensional dispersive equations with fractional dispersion, that are $L^2$-critical, and involve quadratic nonlinearity.

The proof of Theorem \ref{MainInsta} builds upon the work of Martel and Merle \cite{MartelMerle2001} on the critical generalized KdV equation, as well as the instability results for the 2D Zakharov-Kuznetsov equation by Farah, Holmer, and Roudenko \cite{FarahHolmerRoudenko2019}. Unlike these prior works, which rely heavily on pointwise decay formulas, our approach employs Lemma \ref{takamajaka} to derive monotonicity formulas (see Corollary \ref{propdecayIMPR} below), providing an alternative and potentially simpler method to prove instability. The novelty lies in the combination of two monotonicity formulas, where the first provides a base decay, while the second refines this estimate. More precisely, considering the decomposition $\epsilon$ of $u$ around the ground state $Q$ (see \eqref{eqlineareq1} below), we establish in Corollary \ref{propdecayIMPR} that for $1<m<\frac{3}{2}$,
\begin{equation}\label{DecayEpsilonIntr}
\int_{\mathbb{R}}\int_{0}^{\infty} \langle x_1\rangle^{m}\epsilon^2(x_1,x_2,s)\, \mathrm{d}x_1  \mathrm{d}x_2 \lesssim \kappa_{\epsilon}    
\end{equation}
for some constant $\kappa_{\epsilon}>0$ that is small in proportion to the proximity of the initial data $u_0$ to $Q$. This yields a simpler control than the pointwise decay estimates used in previous works.  We believe that this strategy may be adapted to other dispersive equations with fractional dispersion in different spatial variables, not only for instability but also for different studies on the large time behavior of solutions, further highlighting the potential use of our approach. We remark that from $Q(x)\sim \frac{1}{1+|x|^3}$, one may conjecture that \eqref{DecayEpsilonIntr} could be extended to the range $0<m<4$, but our techniques do not bring a conclusion on this note.

We conclude with some remarks and consequences related to Theorem \ref{MainInsta}:
\begin{itemize}
\item The initial data $u_0$ used in Theorem \ref{MainInsta} is constructed following \cite{FarahHolmerRoudenko2019,FarahHolmerRoudenkoSJM2019}. For simplicity, set $c = 1$. Define
\begin{equation*}
a = \frac{\int_{\mathbb{R}^2} \psi_0 Q \, \mathrm{d}x}{\|\psi_0\|_{L^2}^2}, \quad \delta > 0, \quad \text{and} \quad n_0 = \left(1 + \frac{\|\psi_0\|_{H^{\frac{1}{2}}}}{\|\psi_0\|_{L^2}}\right) \frac{\|Q\|_{H^{\frac{1}{2}}}}{\delta}.
\end{equation*}
Then, for any $n \in \mathbb{Z}^{+}$ be such that $n\geq n_0$ set
\begin{equation*}
\epsilon_0 = \frac{1}{n}(Q + a \psi_0),
\end{equation*}
and define $u_0 := Q + \epsilon_0$, which is an initial condition for Theorem \ref{MainInsta}. Since $u_0 \in H^\infty(\mathbb{R}^2)$, the current local well-posedness theory guarantees the existence of $T = T(\|u_0\|_{H^{\frac{3}{2}^+}}) > 0$ and a unique solution
\begin{equation*}
u \in C([0,T); H^\infty(\mathbb{R}^2)) \subset C([0,T); H^{\frac{1}{2}}(\mathbb{R}^2))
\end{equation*}
of \eqref{E:HBO} with $u(0) = u_0$, preserving the mass and energy, i.e., $M[u(t)] = M[u_0]$ and $E[u(t)] = E[u_0]$. Thus $u_0$ satisfies Assumption \ref{Assm1}. Hence, even though local well-posedness is not established directly in the energy space $H^{\frac{1}{2}}(\mathbb{R}^2)$, there are local solutions for which \eqref{condthm} holds for some initial data arbitrarily close to $Q$. 
\item The instability result aligns with numerical observations from \cite{RianoRoudenkoYang2022}, which conjecture finite or infinite time blow-up.
\end{itemize}

%%%%%%%%%%%%%%%%%%%%%%%%%%%%%%%%%%%%%%%%%%%%%%%%%%%%%%%%%%%%%%%%%%%%%%%%%%%%%%%%%%%%%%

\subsection{Organization of the document} In Section \ref{SectPrelimNot}, we introduce the notation and preliminary results, including pseudo-differential calculus results and some properties of the linearized operator $L$. Section \ref{SectionMainlemma} concerns the proof of Lemma \ref{takamajaka}, which provides monotonicity results for the fractional Laplacian on arbitrary dimensions. Finally, in Section \ref{SectInstaResults}, we deduce our main result on instability for the two-dimensional Shrira equation, that is, Theorem \ref{MainInsta}. In this section, we also derive key monotonicity formulas involving the polynomial decay of solutions. 

%%%%%%%%%%%%%%%%%%%%%%%%%%%%%%%%%%%%%%%%%%%%%%%%%%%%%%%%%%%%%%%%%%%%%%%%%%%%%%%%%%%%%%%%%%

\section{Notation and preliminaries}\label{SectPrelimNot}

We begin by fixing the notation used throughout the paper.
\begin{itemize}
    \item Given $a, b \in \mathbb{R}^{+}\cup\{0\}$,  $a\lesssim b$ means that there exists a positive constant $c>0$ such that $a\leq c b$. $a\gtrsim b$ means $b\lesssim a$. We write $a\sim b$ whenever $a\lesssim b$ and $b\lesssim a$. On some calculations, we use the notation $a\lesssim_r b$, $a\gtrsim_r b$, and $a\sim_r b$, where the subindex emphasizes the dependence of the implicit constant on $r>0$.
\item Given $a>0$, $a^{+}=a+\delta$, $a^{-}=a-\delta$ for $\delta>0$ small.  
    \item  The \emph{commutator} between operators $A$ and $B$ is denoted as $[A;B]=AB-BA$. 
    \item  We denote by $\mathbb{Z}^{+}_{0}=\mathbb{Z}^{+}\cup\{0\}$.
    \item The \emph{Fourier transform} is given by 
\begin{equation*}
\widehat{f}(\xi) = \int_{\mathbb{R}^{n}} e^{-ix \cdot \xi} f(x) \,\mathrm{d}x, \quad f \in \mathcal{S}(\mathbb{R}^{n}).
\end{equation*}
\item To facilitate calculations involving the Fourier transform, it is convenient to introduce the notation $D_{j}=-i\partial_{x_{j}},\,\, j\in \{1,2,\cdots,n\}$. This allows us to express higher-order derivatives compactly using either $\partial_{x_1}^{\beta}$ or, when appropriate, the alternative notation $D^{\beta}=D_{1}^{\beta_{1}}D_{2}^{\beta_{2}}\cdots D_{n}^{\beta_{n}}$. 
\item $C^\infty_0(\mathbb{R}^n)$ denotes the space of compactly supported smooth functions. $\mathcal{S}(\mathbb{R}^n)$ denotes the \emph{Schwartz class} of functions. 
\item Given a function $\chi\in \mathcal{S}(\mathbb{R}^n)$, we denote by $\chi(D)$ the Fourier multiplier by the function $\chi(\xi)$.
\item $(\cdot,\cdot)$ denotes the standard inner product in $L^2(\mathbb{R}^n)$.
\end{itemize}
In this manuscript, we consider $\chi \in C^{\infty}_{0}(\mathbb{R}^{n})$ be a smooth cutoff function satisfying the following conditions:
\begin{itemize}
    \item[(i)] $0 \leq \chi \leq 1$, $\chi$ radial,
    \item[(ii)] $\chi$ acts as a smooth indicator of the unit ball, that is, $\chi(\xi)=1$, for $|\xi|\leq 1$, and $\chi(\xi)=0$, whenever $|\xi|\geq 2$.
\end{itemize}

\subsection{Pseudo-differential calculus and commutator estimates}

Adapting $L^2$-monotonicity for the fZK model in the framework of Martel and Merle \cite{MartelMerle2001} presents significant challenges. The main difficulty lies in the nonlocal nature of the operators appearing in energy estimates. This prevents the usual application of integration by parts, which is an essential tool in many classical approaches. To overcome this obstacle, pseudo-differential operators have proven an effective alternative, as explored in \cite{kenig2011local}.  

The class of symbols that combines regularity and decay properties introduced in \cite{Hormander2007,kenig2011local,mendez2025decay} is fundamental for our work. This specific class emerges in different contexts, e.g., it was developed to capture the \emph{propagation of regularity and decay phenomena} for solutions of the Zakharov-Kuznetsov equation, see \cite{mendez2025decay}.  
\begin{defn}
Let $\sigma \in \mathbb{R}^{n}$ and $m, \omega, q \in \mathbb{R}$. We define the class of symbols $S^{m,q}_{\sigma,\omega}(\mathbb{R}^{n} \times \mathbb{R}^{n})$ as the set of functions $a \in C^{\infty}(\mathbb{R}^{n} \times \mathbb{R}^{n})$ such that for all multi-indices $\alpha, \beta \in (\mathbb{Z}_{0}^{+})^{n}$, there exists a positive constant $c$ satisfying
\begin{equation*}
\left|\partial_{x_1}^{\alpha}\partial_{\xi}^{\beta}a(x,\xi)\right| \leq c \langle \sigma \cdot x + \omega \rangle^{q-|\alpha|} \langle \xi \rangle^{m-|\beta|}, \quad \text{for all } x, \xi \in \mathbb{R}^{n}.
\end{equation*}
We refer to an element in $S^{m,q}_{\sigma,\omega}(\mathbb{R}^{n} \times \mathbb{R}^{n})$ as a symbol of order $(m,q)$.
\end{defn}
The previous symbols are fundamental because their structure helps explain how regularity propagates in solutions. Additionally, as we will show, they are critical for studying the monotonicity of the \(L^2\)-mass. This approach opens new possibilities for exploring the dynamics of non-local operators, and the usefulness of these tools will become increasingly clear as the analysis progresses.

Given a symbol $a$ in the class $S^{m,q}_{\sigma,\omega}$, we define the \emph{associated operator} $a(x, D)$ by
\begin{equation*}
a(x,D)f(x) = \frac{1}{(2\pi)^{n}} \int_{\mathbb{R}^{n}} e^{ix \cdot \xi} a(x,\xi) \widehat{f}(\xi) \,\mathrm{d}\xi,
\end{equation*}

A fundamental property in studying pseudo-differential operators is their continuity in Lebesgue spaces. 

\begin{thm}\label{continuity}
Let $\sigma \in \mathbb{R}^n$ and $m, \omega, q \in \mathbb{R}$ be fixed, and $1<p<\infty$. The operator $a(x, D)$, associated with the symbol $a$, is continuous on $L^p(\mathbb{R}^n)$ in the sense
\begin{equation}
\|a(x, D)f\|_{L^p} \lesssim \|\langle \sigma \cdot x + \omega \rangle^q \langle \nabla_x \rangle^m f\|_{L^p}, \quad f \in \mathcal{S}(\mathbb{R}^n),
\end{equation}
where the implicit constant in the inequality is independent of $f$.
\end{thm}
\begin{proof}
See  \cite{mendez2025decay}.
\end{proof}
\begin{defn}  
Let $\Sigma$ be a class of symbols, and suppose that $a(x, \xi) \in \Sigma$. We define the operator $\Psi_a$ to belong to the class $\mathrm{OP}\Sigma$.  
\end{defn}  
To obtain certain estimates for commutators, we will utilize the following asymptotic analysis.  
\begin{prop}\label{PO1}  
Let $\sigma \in \mathbb{R}^n$ and $m_1, m_2, \omega, q_1, q_2 \in \mathbb{R}$. Consider the symbols $a \in \mathbb{S}_{\sigma,\omega}^{m_1,q_1}$ and $b \in \mathbb{S}_{\sigma,\omega}^{m_2,q_2}$. Then, there exists a symbol $c \in \mathbb{S}^{m_1+m_2-1,q_1+q_2-1}_{\sigma,\omega}$ such that  
\begin{equation*}  
     [\Psi_a ; \Psi_b] = \Psi_a \Psi_b - \Psi_b \Psi_a =: \Psi_c.  
\end{equation*}  
Moreover, the symbol $c$ is given by  
\begin{equation*}  
    c \simeq \sum_{|\beta|>0} \frac{(-i)^{|\beta|}}{\beta !} \Big( (\partial^{\beta}_{\xi} a) (\partial^{\beta}_{x} b) - (\partial^{\beta}_{\xi} b) (\partial^{\beta}_{x} a) \Big),  
\end{equation*}  
in the sense that for all $N \geq 2$,  
\begin{equation*}  
    c - \sum_{0<|\beta| < N} \frac{(-i)^{|\beta|}}{ \beta !} \Big( (\partial^{\beta}_{\xi} a) (\partial^{\beta}_{x} b) - (\partial^{\beta}_{\xi} b) (\partial^{\beta}_{x} a) \Big) \in \mathbb{S}^{m_1+m_2-N, q_1+q_2-N}_{\sigma, \omega}.  
\end{equation*}  
\end{prop}  
\begin{proof}
We refer to \cite{mendez2025decay}.
\end{proof}
Additionally, we require the following commutator estimate for the operator $|\nabla_x|$,
which plays a crucial role in handling nonlinear interactions and establishing weighted energy estimates.
\begin{prop}\label{commutatorestim1}
Let $1<p<\infty$, then
\begin{equation}
   \left\|\, |\nabla_x|(fg)-f |\nabla_x|g-\sum_{j=1}^n\partial_{x_j}f\mathcal{R}_j g\right\|_{L^p}\lesssim \||\nabla_x| f\|_{L^{\infty}}\|g\|_{L^p}, 
\end{equation}
where $\mathcal{R}_j$ denotes the Riesz transform operator in the $j$-direction.
\end{prop}

\begin{proof}
This is a particular result of \cite[Theorem 1.2 (3)]{DonLi2019}.    
\end{proof}
%%%%%%%%%%%%%%%%%%%%%%%%%%%%%%%%%%%%%%%%%%%%%%%%%%%%%%%%%
\subsection{Linearized operator and linearization}

In this subsection, our results are restricted to dimension $n=2$, which is where we will study the Shrira equation \eqref{E:HBO}. We first present some key results for the \emph{linearized operator} associated with the Shrira equation \eqref{E:HBO} in $\mathbb{R}^2$, that is
\begin{equation}\label{linearizedop}
L=|\nabla_x|+1-Q,    
\end{equation}
which is obtained by linearization around the ground state solution $Q$ of \eqref{E:HBO}. 

We work with the generator $\Lambda$ of the scaling symmetry given by 
\begin{equation}\label{scalingOp}
\Lambda f=f+x\cdot \nabla f, \quad x\in \mathbb{R}^2.
\end{equation}
The following results are deduced in \cite{FrankLenzmannSilvestre2016}, see also \cite{FrankLenzmann2013}.  

\begin{thm}[Properties of $L$] \label{propL} The following holds for an operator $L$ in \eqref{linearizedop}:
\begin{itemize}
    \item $\ker L=\spn \{\partial_{x_1} Q,\partial_{x_2} Q\}$.
    \item $L$ has a unique single negative eigenvalue $-\mu_0$ (with $\mu_0>0$) associated to a positive radial symmetric eigenfunction $\psi_0\in H^{\infty}(\mathbb{R}^2)$.  Moreover, for any multi-index $\beta$,
    \begin{equation}
        |\partial^{\beta}\psi_0(x)|\lesssim \langle x \rangle^{-3-|\beta|}, \, \, \text{ for all }\,\, x\in \mathbb{R}^2.
    \end{equation}
    \item (Orthogonality Condition). There exists some $c_0>0$ such that for all $f\in H^{\frac{1}{2}}(\mathbb{R}^2)$ with $f$ orthogonal to the set $\spn\{\partial_{x_1}Q,\partial_{x_2}Q,\psi_0\}$, that is,
    \begin{equation}
    (f,\psi_0)=(f,\partial_{x_1}Q)=(f,\partial_{x_2}Q)=0,
    \end{equation}
one has
\begin{equation}\label{Coercond}
(Lf,f)\geq c_0 \|f\|_{H^{\frac{1}{2}}}^2.
\end{equation}
\end{itemize}
\end{thm}
\begin{prop}\label{scalingprop}
The following identities hold
\begin{itemize}
    \item[(i)] $L(\Lambda Q)=-Q$,
    \item[(ii)] $\int_{\mathbb{R}^2} Q\Lambda Q\, \mathrm{d}x=0$.
\end{itemize}
\end{prop}
\begin{proof}
The proof is straightforward after applying integration by parts.
\end{proof}

%%%%%%%%%%%%%%%%%%%%%%%%%%%%%%%%%%%%%%%%%%%%%%%%%%%%%%%%%%%%%%%%%%%%%%%%%%%%%%%%%%%%%%%%%%%%%%%%%%%%%%
\section{Proof of Lemma \ref{takamajaka}}\label{SectionMainlemma}

It is worth mentioning that in the deduction of Lemma \ref{takamajaka}, we work on an arbitrary spatial dimension $n\geq 1$. In the instability results, we focus on the two-dimensional setting $n=2$, which corresponds to the Shrira equation \eqref{E:HBO}.

We begin by recalling some decay properties of the kernel determined by the operator $\partial_{x_1}|\nabla_x|^{\alpha}$, when it is concentrated at low frequencies.

\begin{prop}\label{kernel_decay}
Let $\alpha>0$ and $\chi\in C^{\infty}_0(\mathbb{R}^n)$ be the function introduced above. Define also
\begin{equation*}
    \Omega(x) := \frac{1}{(2\pi)^{n}} \int_{\mathbb{R}^{n}} e^{ix\cdot \xi} |\xi|^{\alpha} (i\xi_{1})\,  \chi(\xi) \,\mathrm{d}\xi.
\end{equation*}
Then, for all $x \in \mathbb{R}^{n}$, we have the decay estimate
\begin{equation}\label{goal1}
    \left| \Omega(x) \right| \lesssim_{\alpha,n} \frac{1}{\langle x \rangle^{n+\alpha+1}}.
\end{equation}
\end{prop}
\begin{proof}
The proof is similar to that of \cite[Claim 6]{kenig2011local}.
\end{proof}
We are now in a position to establish Lemma \ref{takamajaka}.
\begin{proof}[ Proof of the Lemma \ref{takamajaka}]
Since $\partial_{x_1}|\nabla_x|$ is a skew-symmetric operator, it is straightforward to verify that  
\begin{equation*}
\int_{\mathbb{R}^n} ( \partial_{x_1} |\nabla_x|^\alpha u ) \, u \varphi \, \mathrm{d}x = -\frac{1}{2} \int_{\mathbb{R}^n} u \, \left[\partial_{x_1} |\nabla_x|^\alpha; \varphi\right] u \, \mathrm{d}x.
\end{equation*}  
We use a specific construction to decouple the commutator to handle the term on the right-hand side, employing the standard \emph{high-low frequency} decomposition using the function $\chi \in C^\infty_0(\mathbb{R}^n)$ introduced above. More precisely, we introduce the operators  
\begin{equation*}
\left(p_1(x,D)f\right)(x) := \partial_{x_1}|\nabla_x|^{\alpha}\chi(D)f=\frac{1}{(2\pi)^n} \int_{\mathbb{R}^n} e^{i x \cdot \xi} (i \xi_1)|\xi|^\alpha  \chi(\xi) \widehat{f}(\xi) \, \mathrm{d}\xi
\end{equation*}  
and  
\begin{equation*}
\left(p_2(x,D)f\right)(x) := \partial_{x_1}|\nabla_x|^{\alpha}(1-\chi(D))= \frac{1}{(2\pi)^n} \int_{\mathbb{R}^n} e^{i x \cdot \xi} (i \xi_1) |\xi|^\alpha (1 - \chi(\xi)) \widehat{f}(\xi) \, \mathrm{d}\xi,
\end{equation*}  
whenever \( f \in \mathcal{S}(\mathbb{R}^n) \).  Thus, $\partial_{x_{1}}|\nabla_{x}|^{\alpha} =p_{1}(x,D)+p_{2}(x,D)$ yields
\begin{equation}\label{firstdecom}
\begin{split}
\int_{\mathbb{R}^{n}}u\partial_{x_{1}}|\nabla_{x}|^{\alpha}u \varphi \,\mathrm{d}x&=-\frac{1}{2}\int_{\mathbb{R}^{n}}u \left[p_{1}(D); \varphi\right]u \,\mathrm{d}x
 -\frac{1}{2}\int_{\mathbb{R}^{n}}u \left[p_{2}(D); \varphi\right]u \,\mathrm{d}x\\
 &=:\Pi_{1}+\Pi_{2}.
\end{split}
\end{equation}
Firstly, recalling the operator $\Omega(x)$ in Proposition \ref{kernel_decay}, we have
\begin{equation}\label{express}
\begin{split}
\left[P_{1}(D); \varphi \right]u(x)=\int_{\mathbb{R}^{n}}\Omega(x-y)\left(\varphi(y)-\varphi(x)\right) u(y)\,\mathrm{d}y.
\end{split}
\end{equation}
We first study the difference $\varphi(y)-\varphi(x)$.  We write
\begin{equation*}
\begin{aligned}
\varphi(y)-\varphi(x)=&\frac{\sigma\cdot(y-x)}{\langle \sigma\cdot x+\omega\rangle^{\gamma}\langle \sigma\cdot y+\omega\rangle^{\gamma}}+R(x,y)\\
=&\phi(x)\phi(y)\big(\sigma\cdot (y-x)\big)+R(x,y),\quad x,y\in\mathbb{R}^{n},    
\end{aligned}
\end{equation*}
where the remainder term $R(x,y)$ is given by 
\begin{equation*}
R(x,y):=\int_{\sigma\cdot x}^{\sigma\cdot y}\left(\frac{1}{\langle s+\omega\rangle^{2\gamma}}-\frac{1}{\langle \sigma\cdot x+\omega\rangle^{\gamma}\langle \sigma\cdot y+\omega\rangle^{\gamma}}\right)\,\mathrm{d}s.
\end{equation*}
Following similar arguments in \cite[Claim 6]{kenig2011local}, we deduce
\begin{itemize}
\item  If  $ |\sigma\cdot (x- y)|\leq \frac{1}{2}\left(\langle \sigma\cdot x+\omega\rangle+  \langle \sigma\cdot y+\omega\rangle\right),$ then 
\begin{equation}\label{decay1}
|R(x,y)|\lesssim   
 \frac{|\sigma\cdot (x-y)|^{2}}{\left(\langle \sigma\cdot x+\omega\rangle+\langle \sigma\cdot y+\omega\rangle\right)^{2\gamma+1}}.
 \end{equation}
\item If  $|\sigma\cdot ( x- y)|>\frac{1}{2}\left(\langle \sigma\cdot x+\omega\rangle +\langle \sigma\cdot y+\omega\rangle \right),$ then
  \begin{equation}\label{decay2}
 |R(x,y)|\lesssim   \left(1+\frac{|\sigma \cdot (x-y)|}{\langle \sigma\cdot x+\omega\rangle^{\gamma}\langle \sigma\cdot y+\omega \rangle^{\gamma}}\right).
 \end{equation}
\end{itemize}
We now return to $\Pi_1$ in \eqref{firstdecom}. With the above estimates at hand, we proceed as follows
\begin{equation*}
\begin{split}
\Pi_{1}&=-\frac{1}{2}\int_{\mathbb{R}^{n}}u(x)\left(\int_{\mathbb{R}^{n}}\Omega(x-y)\left(\varphi(y)-\varphi(x)\right) u(y)\,\mathrm{d}y\right)\,\mathrm{d}x\\
&=-\frac{1}{2}\int_{\mathbb{R}^{n}}\int_{\mathbb{R}^{n}}u(x)\phi(x)\mathcal{K}(x,y)\phi(y)u(y)\,\mathrm{d}y\mathrm{d}x\\
&\quad -\frac{1}{2}\int_{\mathbb{R}^{n}}\int_{\mathbb{R}^{n}}u(x) \Omega(x-y)R(x,y)u(y)\,\mathrm{d}y\mathrm{d}x\\
&=:\Pi_{1,1}+\Pi_{1,2},
\end{split}
\end{equation*}
where $\mathcal{K}(x,y):= \sigma\cdot(y-x)\Omega(x-y),\,\, x,y\in\mathbb{R}^{n}$. By using the definition of $\Omega$ (see Proposition \ref{kernel_decay}), it follows from integration by parts
 \begin{equation*}
 \begin{split}
 \mathcal{K}(x,y)=&-\frac{1}{(2\pi)^{n}}\sum_{j=1}^{n}\sigma_{j}\int_{\mathbb{R}^{n}}\partial_{\xi_{j}}\left(e^{i(x-y)\cdot \xi}\right)\xi_{1}|\xi|^{\alpha}\chi(\xi)\,\mathrm{d}\xi\\
 =&\frac{\sigma_{1}}{(2\pi)^{n}}\int_{\mathbb{R}^{n}}e^{i(x-y)\xi}|\xi|^{\alpha}\chi(\xi)\,\mathrm{d}\xi+ \alpha \sum_{j=1}^{n}\frac{\sigma_{j}}{(2\pi)^{n}}\int_{\mathbb{R}^{n}}e^{i(x-y)\cdot \xi}\xi_{1}\xi_{j}|\xi|^{\alpha-2}\chi(\xi)\,\mathrm{d}\xi\\
 &+\sum_{j=1}^{n}\frac{\sigma_{j}}{(2\pi)^{n}}\int_{\mathbb{R}^{n}}e^{i(x-y)\cdot \xi}\xi_{1}|\xi|^{\alpha}(\partial_{\xi_{j}}\chi)(\xi)\,\mathrm{d}\xi\\
 =:& \mathcal{K}_1(x-y)+\mathcal{K}_2(x-y)+\mathcal{K}_3(x-y).
 \end{split}
 \end{equation*}
Consequently, we write
 \begin{equation*}
 \begin{split}
 \Pi_{1,1} &=-\frac{1}{2}\int_{\mathbb{R}^{n}}u(x)\phi(x) (\mathcal{K}_{1} *(u\phi))(x)\,\mathrm{d}x-\frac{1}{2}\int_{\mathbb{R}^{n}}u(x)\phi(x)(\mathcal{K}_{2} *(u\phi))(x)\,\mathrm{d}x\\
 &\quad -\frac{1}{2}\int_{\mathbb{R}^{n}}u(x)\phi(x) (\mathcal{K}_{3} *(u\phi))(x)\,\mathrm{d}x\\
&=:\Pi_{1,1,1}+\Pi_{1,1,2}+\Pi_{1,1,3}.
 \end{split}
 \end{equation*}
We rewrite the first term on the left-hand side above as follows  
\begin{equation}\label{part1smoothing}
\begin{split}
\Pi_{1,1,1}&=-\frac{\sigma_{1}}{2}\int_{\mathbb{R}^{n}}u\phi |\nabla_{x}|^{\alpha}\chi(D)(u\phi)\,\mathrm{d}x=-\frac{\sigma_{1}}{2}\int_{\mathbb{R}^{n}}|\nabla_{x}|^{\frac{\alpha}{2}}(u\phi) |\nabla_{x}|^{\frac{\alpha}{2}}\chi(D)(u\phi)\,\mathrm{d}x.
\end{split}
\end{equation}  
This expression will be estimated later along with $\Pi_{1,1,2}$, which we write as
\begin{equation}\label{lfre}
\begin{split}
\Pi_{1,1,2}&=-\sum_{j=1}^{n}\frac{\alpha\sigma_{j}}{2}\int_{\mathbb{R}^{n}}u\phi |\nabla_{x}|^{\alpha-2}\chi(D)D_{1}D_{j}(u\phi)\,\mathrm{d}x.
\end{split}
\end{equation}  
 Now, for $\Pi_{1,1,3}$, we observe
 \begin{equation*}
 \begin{split}
 \Pi_{1,1,3}&=\sum_{j=1}^{n}\frac{\sigma_{j}}{2}\int_{\mathbb{R}^{n}}u\phi \partial_{x_1}|\nabla_{x}|^{\alpha}\chi_{j}(D)(u\phi)\,\mathrm{d}x,
 \end{split}
 \end{equation*}
 where, for simplicity in the notation, we have defined $\chi_{j}(\xi)=(i\partial_{\xi_{j}}\chi)(\xi),\,\, \xi \in\mathbb{R}^{n}$. Since $\partial_{x_1}|\nabla_{x}|^{\alpha}\chi_{j}(D)\in OPS^{0,0}_{\sigma,\omega}$, using the continuity of pseudo-differential operators in Theorem \ref{continuity}, we obtain 
 \begin{equation}\label{ine1.1}
 \begin{split}
 |\Pi_{1,1,3}|&\lesssim \left\|\frac{u}{\langle \sigma \cdot x+\omega\rangle^{\gamma}}\right\|_{L^{2}}^{2}\sim \int_{\mathbb{R}^{n}}u^{2}\phi^{2}\,\mathrm{d}x.
 \end{split}
 \end{equation}
  Next, we focus on the term $\Pi_{1,2}$. To this end, we define 
 \begin{equation}\label{oper1}
 (\mathcal{T}u)(x):=\int_{\mathbb{R}^{n}}\Omega(x-y)R(x,y)u(y)\,\mathrm{d}y,\quad u\in\mathcal{S}(\mathbb{R}^{n}).
 \end{equation}
We express $\Pi_{1,2}$ using the operator $\mathcal{T}$ as follows
 \begin{equation*}
 \begin{split}
 \Pi_{1,2}
 &=-\frac{1}{2}\int_{\mathbb{R}^{n}}\left(\frac{u}{\langle \sigma \cdot x+\omega\rangle^{\gamma}}\right)\langle \sigma \cdot x+\omega\rangle^{\gamma}(\mathcal{T}u)(x)\,\mathrm{d}x\\
 &=-\frac{1}{2}\int_{\mathbb{R}^{n}}\widetilde{u}\langle \sigma \cdot x+\omega\rangle^{\gamma}\left(\mathcal{T}\left(\langle \sigma\cdot y +\omega\rangle^{\gamma}\, \widetilde{u} \right)\right)(x)\,\mathrm{d}x,
 \end{split}
 \end{equation*}
 where $\widetilde{u}(x):=\frac{u(x)}{\langle \sigma \cdot x+\omega\rangle^{\gamma}},\quad x\in\mathbb{R}^{n}$. Thus, the problem of estimating $\Pi_{2,1}$ reduces to show that the operator defined as
\begin{equation*}
 (\mathcal{T}_{1}f)(x) := \langle \sigma \cdot x + \omega \rangle^{\gamma} \mathcal{T}\left(\langle \sigma \cdot y + \omega \rangle^{\gamma} f\right)(x), \quad f \in \mathcal{S}(\mathbb{R}^{n}),
\end{equation*}
 is $L^{2}-L^{2}$ continuous. Notice that for $f\in\mathcal{S}(\mathbb{R}^{n}),$
 \begin{equation*}
 (\mathcal{T}_{1}f)(x)=\int_{\mathbb{R}^{n}}\langle \sigma \cdot x+\omega\rangle^{\gamma}\Omega(x-y)R(x,y)\langle \sigma \cdot y+\omega\rangle^{\gamma}f(y)\,\mathrm{d}y,
 \end{equation*}
 that is, 
$\mathcal{T}_{1}$  has kernel $\mathcal{Q}$, which is given by  
 \begin{equation*}
 \begin{split}
 \mathcal{Q}(x,y)&:= \langle \sigma \cdot x+\omega\rangle^{\gamma}\Omega(x-y)R(x,y)\langle \sigma \cdot y+\omega\rangle^{\gamma}=\mathcal{Q}_{1}(x,y)+\mathcal{Q}_{2}(x,y),
 \end{split}
 \end{equation*}
 where 
 \begin{equation*}
 \mathcal{Q}_{1}:=\mathcal{Q}|_{\Sigma_{1}}\quad \mbox{and}\quad \mathcal{Q}_{2}:=\mathcal{Q}|_{\Sigma_{2}},
 \end{equation*}
 being $\Sigma_{1}$ and $\Sigma_{2}$ the  following regions:
 \begin{equation*}
 \Sigma_{1}:=\left\{(x,y)\in\mathbb{R}^{n}\times \mathbb{R}^n\,:\,  |\sigma\cdot (x- y)|\leq \frac{1}{2}\left(\langle \sigma\cdot x+\omega\rangle+  \langle \sigma\cdot y+\omega\rangle\right)\right\},
 \end{equation*}
 and 
  \begin{equation*}
 \Sigma_{2}:=\left\{(x,y)\in\mathbb{R}^{n}\times \mathbb{R}^n\,:\,  |\sigma\cdot (x- y)|>\frac{1}{2}\left(\langle \sigma\cdot x+\omega\rangle+  \langle \sigma\cdot y+\omega\rangle\right)\right\}.
 \end{equation*}
 In virtue of \eqref{decay1}, it is clear that   
 \begin{equation}\label{one1}
 \begin{split}
 \int_{\mathbb{R}^{n}}|\mathcal{Q}_{1}(x,y)|\,\mathrm{d}x &=\int_{\Sigma_{1}}\langle \sigma \cdot x+\omega\rangle^{\gamma}|\Omega(x-y)R(x,y)|\langle \sigma \cdot y+\omega\rangle^{\gamma}\,\mathrm{d}x\\
 &\leq  \int_{\Sigma_{1}}\frac{|\Omega(x-y)||\sigma\cdot(x-y)|^{2}\langle \sigma\cdot x+\omega\rangle^{\gamma} \langle \sigma\cdot y+\omega\rangle^{\gamma}}{\left(\langle \sigma\cdot x+\omega\rangle+\langle \sigma\cdot y+\omega\rangle\right)^{2\gamma+1}}\,\mathrm{d}x\\
 &\lesssim_{\sigma} \int_{\Sigma_{1}}\frac{|\Omega(x-y)||\sigma\cdot(x-y)|^{2}}{\langle \sigma\cdot x+\omega\rangle+\langle \sigma\cdot y+\omega\rangle}\,\mathrm{d}x.
 \end{split}
 \end{equation}
 By Proposition \ref{kernel_decay}, it follows that
  \begin{equation}\label{one2}
  \int_{\Sigma_{1}}\frac{|\Omega(x-y)||\sigma\cdot (x-y)|^{2}}{\langle \sigma\cdot x+\omega\rangle+\langle \sigma\cdot y+\omega\rangle}\,\mathrm{d}x
  \lesssim_{\alpha,\sigma,n}  \int_{\mathbb{R}^{n}}\frac{\mathrm{d}x}{\langle x-y\rangle^{n+\alpha}}
\sim_{\alpha,n} 1.
 \end{equation}
We combine \eqref{one1} and \eqref{one2} to obtain
 \begin{equation*}
\sup_{y\in\mathbb{R}^{n}}  \int_{\mathbb{R}^{n}}|\mathcal{Q}_{1}(x,y)|\,\mathrm{d}x\lesssim_{\alpha,n}1.
 \end{equation*}
On the other hand, since $\sigma\neq 0$ on $\Sigma_2$, we have that
\begin{equation*}
\begin{aligned}
\langle x-y \rangle \gtrsim \langle \sigma\cdot x+\omega \rangle+\langle \sigma\cdot y+\omega \rangle.  
\end{aligned}    
\end{equation*}
This inequality, \eqref{decay2} and the decay of $\Omega$ establish
\begin{equation*}
\begin{split}
 \int_{\mathbb{R}^n}|\mathcal{Q}_{2}(x,y)|\,\mathrm{d}x
&\lesssim\int_{\Sigma_{2}}\frac{\langle \sigma \cdot x+\omega\rangle^{\gamma}\langle \sigma \cdot y+\omega\rangle^{\gamma}}{\langle x-y \rangle^{n+\alpha+1}}\,\mathrm{d}x +\int_{\Sigma_{2}}\frac{|\sigma \cdot (x-y)|}{\langle x-y\rangle^{n+\alpha+1}}\,\mathrm{d}x\\
&\lesssim \langle \sigma \cdot y+\omega\rangle^{\gamma}\int_{|x-y|\gtrsim \langle \sigma\cdot y+\omega\rangle}\frac{1}{\langle x-y \rangle^{n+\alpha+1-\gamma}}\,\mathrm{d}x +\int_{\mathbb{R}^n}\frac{1}{\langle x-y\rangle^{n+\alpha}}\,\mathrm{d}x\\
&\lesssim  1+ \frac{\langle \sigma\cdot y+\omega\rangle^{\gamma}}{\langle \sigma\cdot y+\omega\rangle^{\alpha+1-\gamma}}\lesssim  1,
\end{split}
\end{equation*}
which is bounded provided that $\gamma \leq \frac{\alpha+1}{2}$. 

We conclude that $\sup_{y\in\mathbb{R}^{n}}  \int_{\mathbb{R}^{n}}|\mathcal{Q}_{2}(x,y)|\,\mathrm{d}x\lesssim_{\alpha,n}1$, which using the previous estimation of of the integral of $Q_1(x,y)$, and that $\Sigma_1$ and $\Sigma_2$ are disjoint yield
\begin{equation*}  
	\sup_{y\in\mathbb{R}^{n}}\int_{\mathbb{R}^{n}}|\mathcal{Q}(x,y)|\,\mathrm{d}x\lesssim 1.  
\end{equation*}  
By symmetry of the preceding argument, it  follows that 
\begin{equation*}  
	\sup_{x\in\mathbb{R}^{n}}\int_{\mathbb{R}^{n}}|\mathcal{Q}(x,y)|\,\mathrm{d}y\lesssim 1.  
\end{equation*}  
By Schur's test (see \cite{MR517709}, $\S 5.$), we establish the desired continuity of the operator $\mathcal{T}_{1}$ in $L^2$, which in turn, implies the required bound for $\Pi_{1,2}$. More precisely, we have deduced
\begin{equation*}
	|\Pi_{1,2}|\lesssim \int_{\mathbb{R}^{n}}u^{2}\phi^{2}\,\mathrm{d}x.
	\end{equation*} 

Next, we focus on the term $\Pi_{2}$ in \eqref{firstdecom}.  
Observing that $p_{2}(D) \in OPS^{\alpha+1,0}_{\sigma,\omega},$ we immediately obtain from Proposition \ref{PO1} that
\begin{equation*}
\begin{split}
\left[p_{2}(D); \varphi\right]=c(x,D)\in OPS^{\alpha,0}_{\sigma,\omega} ,
\end{split}
\end{equation*}
where the symbol $c(x,\xi)$ satisfies
\begin{equation*}
\begin{split}
c(x,\xi)&=-i\sum_{j=1}^n(\partial_{\xi_j}p_{2})(\xi)(\partial_{x_j} \varphi)(x) - \frac{1}{2}\sum_{j,m=1}^n(\partial_{\xi_j}\partial_{\xi_m}p_{2})(\xi)(\partial_{x_j}\partial_{x_m}\varphi)(x) + c_{\alpha-2}(x,\xi)\\
&=:c_{\alpha}(x,\xi)+c_{\alpha-1}(x,\xi)+c_{\alpha-2}(x,\xi),
\end{split}
\end{equation*}
recall that $p_2(\xi)=i\xi_1|\xi|^{\alpha}(1-\chi(\xi))$. We have that
\begin{equation*}
\begin{aligned}
\Pi_{2}=&-\frac{1}{2}\int_{\mathbb{R}^n} u \big(c_{\alpha}(x,D)+c_{\alpha-1}(x,D) \big) u \,\mathrm{d}x-\frac{1}{2}\int_{\mathbb{R}^n} u\, c_{\alpha-2}(x,D) u \, \mathrm{d}x.
\end{aligned}    
\end{equation*}
Since $c_{\alpha-2}(x,\xi)\in S^{\alpha-2,-2\gamma-2}_{\sigma,\omega}\subset S^{0,-2\gamma}_{\sigma,\omega}$, provided that $\alpha\in [1,2)$, Theorem \ref{continuity} implies
\begin{equation*}
\begin{split}
\left|\int_{\mathbb{R}^{n}}uc_{\alpha-2}(x,D)u\,\mathrm{d}x\right|&=\left|\int_{\mathbb{R}^{n}}\langle \sigma\cdot x+\omega\rangle^{-\gamma}u \langle \sigma\cdot x+\omega\rangle^{\gamma}c_{\alpha-2}(x,D)u\,\mathrm{d}x\right|\\
&\lesssim \left\|\frac{u}{\langle \sigma\cdot x+\omega\rangle^{\gamma}}\right\|_{L^{2}}^{2}\lesssim \int_{\mathbb{R}^{n}}u^{2}\phi^{2}\,\mathrm{d}x.
\end{split}
\end{equation*}
On the other hand, let us further decompose the operators $c_{\alpha}(x,D)$ and $c_{\alpha-1}(x,D)$. We have
\begin{equation}\label{fc1}
	\begin{split}
		c_{\alpha}(x,\xi) =\ & |\xi|^{\alpha} (1 - \chi(\xi))\, \partial_{x_{1}}\varphi(x) \\
		& + \alpha \sum_{j=1}^n \xi_1 \xi_j |\xi|^{\alpha - 2} (1 - \chi(\xi))\, \partial_{x_j}\varphi(x) \\
		& - \sum_{j=1}^n \xi_1 |\xi|^{\alpha} \left(\partial_{\xi_j} \chi(\xi)\right)\, \partial_{x_j}\varphi(x).
	\end{split}
\end{equation}
and 
\begin{equation*}
\begin{aligned}
c_{\alpha-1}(x,\xi)=&-i\alpha \sum_{j=1}^n \xi_j |\xi|^{\alpha-2}(1-\chi(\xi))(\partial_{x_j}\partial_{x_1}\varphi)(x)+\frac{1}{2}\xi_1 |\xi|^{\alpha-2}(1-\chi(\xi))(\partial_{x_j}^2\varphi)(x)\\
&-\frac{i}{2}\alpha(\alpha-2)\sum_{j,m=1}^n \xi_1 \xi_j \xi_m |\xi|^{\alpha-4} (1-\chi(\xi))(\partial_{x_j}\partial_{x_m}\varphi)(x)\\
&+\frac{i}{2}\sum_{j=1}^n|\xi|^{\alpha} (\partial_{\xi_j}\chi)(\xi)(\partial_{x_j}\partial_{x_1}\varphi)(x)+\frac{i \alpha}{2}\sum_{j,m=1}^n \xi_1 \xi_m |\xi|^{\alpha-2}(\partial_{\xi_j}\chi)(\xi)(\partial_{x_j}\partial_{x_m}\varphi)(x)\\
&+\frac{i \alpha}{2}\sum_{j,m=1}^n \partial_{\xi_j}\big(\xi_1|\xi|^{\alpha}\partial_{\xi_m}\chi)(\xi)\big)(\partial_{x_j}\partial_{x_m}\varphi)(x).
\end{aligned}    
\end{equation*}
Consequently, using that $\partial_{x_j}\varphi=\sigma_j\phi^2
$, and that   $\partial_{x_j}\partial_{x_m}\varphi(x)=-2\gamma\sigma_j \sigma_m\phi^2(x)\, \left(\frac{\sigma\cdot x+\omega}{\langle \sigma\cdot x+\omega \rangle^2}\right)$, we deduce
\begin{equation*}
\begin{aligned}
c_{\alpha}(x,D)+&c_{\alpha-1}(x,D)\\
=&\sigma_1 \phi^2 |\nabla_x|^{\alpha}(1-\chi(D))  +\alpha\sum_{j=1}^n \sigma_j \phi^2|\nabla_x|^{\alpha-2}(1-\chi(D))D_{1}D_j \\
&+2\gamma \alpha\sum_{j=1}^n \sigma_1\sigma_j \phi^2  \left(\frac{\sigma\cdot x+\omega}{\langle \sigma\cdot x+\omega \rangle^2}\right)\partial_{x_j}|\nabla_x|^{\alpha-2}(1-\chi(D))\\
&+\gamma \alpha\sum_{j=1}^n \sigma_j^2 \phi^2  \left(\frac{\sigma\cdot x+\omega}{\langle \sigma\cdot x+\omega \rangle^2}\right)\partial_{x_1}|\nabla_x|^{\alpha-2}(1-\chi(D))\\
&-\gamma\alpha(\alpha-2)\sum_{j,m=1}^n \sigma_j \sigma_m  \phi^2  \left(\frac{\sigma\cdot x+\omega}{\langle \sigma\cdot x+\omega \rangle^2}\right) \partial_{x_1}\partial_{x_j}\partial_{x_m}|\nabla_x|^{\alpha-4}(1-\chi(D))\\
&+\phi^2 \left(\frac{\sigma\cdot x+\omega}{\langle \sigma\cdot x+\omega \rangle^2}\right)\widetilde{c}_{\alpha}(D)\\
=:&\sum_{j=1}^5 c_{\alpha,j}(x,D)+\phi^2 \left(\frac{\sigma\cdot x+\omega}{\langle \sigma\cdot x+\omega \rangle^2}\right)\widetilde{c}_{\alpha}(D),
\end{aligned}
\end{equation*}
where the operator $\widetilde{c}_{\alpha}(D)$ has associated symbol
\begin{equation*}
\begin{aligned}
\widetilde{c}_{\alpha}(\xi)=&\frac{i}{2}\sum_{j=1}^n \sigma_1 \sigma_j|\xi|^{\alpha} (\partial_{\xi_j}\chi)(\xi)+\frac{i \alpha}{2}\sum_{j,m=1}^n \sigma_j \sigma_m \xi_1 \xi_m |\xi|^{\alpha-2}(\partial_{\xi_j}\chi)(\xi)\\
&+\frac{i \alpha}{2}\sum_{j,m=1}^n \sigma_j \sigma_m \partial_{\xi_j}\big(\xi_1|\xi|^{\alpha}\partial_{\xi_m}\chi)(\xi)\big). 
\end{aligned}    
\end{equation*}
We remark that since the derivatives of $\chi(\xi)$
 are supported in the region $1\leq |\xi|\leq 2$, we have that $\widetilde{c}_{\alpha}(D)\in OPS_{\sigma,\omega}^{0,0}$. It follows
\begin{equation*}
\begin{aligned}
-\frac{1}{2}\int_{\mathbb{R}^n} & u \big(c_{\alpha}(x,D)+c_{\alpha-1}(x,D) \big) u \,\mathrm{d}x\\
=&-\frac{1}{2}\sum_{j=1}^5\int_{\mathbb{R}^n} u c_{\alpha,j}(x,D) u\, \mathrm{d}x-\frac{1}{2} \int_{\mathbb{R}^n} u \phi^2 \left(\frac{\sigma\cdot x+\omega}{\langle \sigma\cdot x+\omega \rangle^2}\right)\widetilde{c}_{\alpha}(D) u \, \mathrm{d}x\\
=:&\sum_{j=1}^5\Pi_{2,j}-\frac{1}{2} \int_{\mathbb{R}^n} u \phi^2 \left(\frac{\sigma\cdot x+\omega}{\langle \sigma\cdot x+\omega \rangle^2}\right)\widetilde{c}_{\alpha}(D) u \, \mathrm{d}x.
\end{aligned}    
\end{equation*}
Using that
\begin{equation*}
	\langle \sigma \cdot x + \omega \rangle^{\gamma - 1} \phi^2(x) \, \widetilde{c}_{\alpha}(D) \big( \langle \sigma \cdot x + \omega \rangle^{\gamma} \, \cdot \big) \in OPS_{\sigma,\omega}^{0,0},
\end{equation*}
the continuity of the pseudo-differential operators implies
\begin{equation}\label{ctildeestimate}
	\begin{aligned}
		& \left| \int_{\mathbb{R}^n} u \, \phi^2 \left( \frac{\sigma \cdot x + \omega}{\langle \sigma \cdot x + \omega \rangle^2} \right) \widetilde{c}_{\alpha}(D) u \, \mathrm{d}x \right| \\
		& \quad \lesssim \left\| \frac{u}{\langle \sigma \cdot x + \omega \rangle^{\gamma}} \right\|_{L^2}
		\left\| \langle \sigma \cdot x + \omega \rangle^{\gamma - 1} \phi^2(x) \, \widetilde{c}_{\alpha}(D) \left( \langle \sigma \cdot x + \omega \rangle^{\gamma} \frac{u}{\langle \sigma \cdot x + \omega \rangle^{\gamma}} \right) \right\|_{L^2} \\
		& \quad \lesssim \int_{\mathbb{R}^n} u^2 \phi^2 \, \mathrm{d}x,
	\end{aligned}
\end{equation}
where the above implicit constant depends on
$\sum_{j=1}^n |\sigma_1 \sigma_j| + \sum_{j,m=1}^n |\sigma_j \sigma_m|$.

The estimates for the terms $\Pi_{2,j}$, $j=3,4,5$, all follow from similar ideas. To show the main considerations, let us study $\Pi_{2,5}$. Using that $\partial_{x_1}\partial_{x_j}\partial_{x_m}|\nabla_x|^{\alpha-4}(1-\chi(D))$ is a skew-symmetric operator, we find
\begin{equation*}
 \begin{aligned}
\Pi_{2,5}=\frac{\gamma \alpha(\alpha-2)}{4}\sum_{j,m=1}^n \sigma_j \sigma_m \int_{\mathbb{R}^n} u \, \left[\phi^2\left(\frac{\sigma\cdot x+\omega}{\langle \sigma\cdot x+\omega\rangle^2 }\right);\partial_{x_1}\partial_{x_j}\partial_{x_m}|\nabla_x|^{\alpha-4}(1-\chi(D))\right] u.      
 \end{aligned}   
\end{equation*}
The properties of the commutator of pseudo-differential operators establish
\begin{equation*}
\begin{aligned}
\langle \sigma\cdot x+\omega\rangle^{\gamma}\left[\phi^2\left(\frac{\sigma\cdot x+\omega}{\langle \sigma\cdot x+\omega\rangle^2 }\right);\partial_{x_1}\partial_{x_j}\partial_{x_m}|\nabla_x|^{\alpha-4}(1-\chi(D))\right] \langle \sigma\cdot x+\omega\rangle^{\gamma} \in OPS_{\sigma,\omega}^{\alpha-2,-2}.  
\end{aligned}
\end{equation*}
Applying that $S_{\sigma,\omega}^{\alpha-2,-2}\subset S_{\sigma,\omega}^{0,0}$ provided that $\alpha<2$, we deduce from the continuity of pseudo-differential operators that
\begin{equation*}
 \begin{aligned}
\left|\Pi_{2,5}\right|\lesssim |\gamma||\alpha(\alpha-2)|\sum_{j,m=1}^n |\sigma_j| |\sigma_m| \int_{\mathbb{R}^n} u^2\phi^2 \, \mathrm{d}x. 
 \end{aligned}   
\end{equation*}
As mentioned above, the same argument applies to the terms involving the operators $c_{\alpha,3}(x, D)$ and $c_{\alpha,4}(x, D)$. To summarize, we conclude
\begin{equation*}
 \begin{aligned}
\left|\Pi_{2,3}\right|+\left|\Pi_{2,4}\right|\lesssim |\gamma||\alpha|\sum_{j=1}^n \big(|\sigma_1||\sigma_j|+|\sigma_j|^2\big)\int_{\mathbb{R}^n} u^2\phi^2 \, \mathrm{d}x. 
 \end{aligned}   
\end{equation*}
We now proceed to estimate $\Pi_{2,1}$, which we write
 \begin{equation}\label{highfredecom}
 \begin{split}
 \Pi_{2,1}&=\frac{\sigma_{1}}{2}\int_{\mathbb{R}^{n}}u \phi \left[|\nabla_{x}|^{\alpha}(1-\chi(D)); \phi\right] u\,\mathrm{d}x-\frac{\sigma_{1}}{2}\int_{\mathbb{R}^{n}}|\nabla_{x}|^{\frac{\alpha}{2}}(u\phi) |\nabla_{x}|^{\frac{\alpha}{2}}(1-\chi(D))(u\phi)\,\mathrm{d}x\\
 &=:\Pi_{2,1,1}+\Pi_{2,1,2}.
 \end{split}
 \end{equation}
The term $\Pi_{2,1,2}$ will be combined with $\Pi_{1,1,1}$ in \eqref{part1smoothing} later. We address now the term $\Pi_{2,1,1}$. By applying pseudo-differential calculus we have
\begin{equation}\label{decompoperA}
\begin{split}
\left[|\nabla_{x}|^{\alpha}(1-\chi(D)); \phi\right] = q_{\alpha-1}(x,D) + q_{\alpha-2}(x,D),
\end{split}
\end{equation}
where
\begin{equation*}
\begin{split}
q_{\alpha-1}(x,\xi) &= -\alpha \gamma \phi \sum_{j=1}^{n} \sigma_{j} \left( \frac{\sigma \cdot x + \omega}{\langle \sigma \cdot x + \omega \rangle^2} \right) |\xi|^{\alpha-2} (i\xi_{j}) (1-\chi(\xi)) \\
&\quad + \gamma \phi \sum_{j=1}^{n} \sigma_{j} \left( \frac{\sigma \cdot x + \omega}{\langle \sigma \cdot x + \omega \rangle^2} \right) |\xi|^{\alpha}(i\partial_{\xi_j}\chi )(\xi),
\end{split}
\end{equation*}
and $q_{\alpha-2}(x,D) \in OPS^{\alpha-2,-2\gamma-2}_{\sigma,\omega} \subset OPS^{0,-2\gamma}_{\sigma,\omega}$. Thus, after replacing the identity above into $\Pi_{2,1,1}$, we get 
 \begin{equation}\label{decompoperAC}
 \begin{split}
 \Pi_{2,1,1}&=-
\frac{\alpha \sigma_{1} \gamma}{2}\sum_{j=1}^{n}\sigma_{j}\int_{\mathbb{R}^{n}}u\phi^{2} \left(\frac{\sigma\cdot x+\omega}{\langle \sigma \cdot x+\omega\rangle^2}\right)\partial_{x_{j}}|\nabla_{x}|^{\alpha-2}(1-\chi(D))u\,\mathrm{d}x\\
&\quad +\frac{\sigma_{1}\gamma}{2}\sum_{j=1}^{n}\sigma_{j}\int_{\mathbb{R}^{n}}u\phi^{2} \left(\frac{\sigma\cdot x+\omega}{\langle \sigma \cdot x+\omega\rangle^2}\right)|\nabla_{x}|^{\alpha}\chi_{j}(D)u\,\mathrm{d}x+\int_{\mathbb{R}^n} u \phi q_{\alpha-2}(x,D) u\,  \mathrm{d}x \\
&=:\Pi_{2,1,1,1}+\Pi_{2,1,1,2}+\Pi_{2,1,1,3},
 \end{split}
 \end{equation}
 where we recall the notation $\chi_{j}(\xi)=(i\partial_{\xi_{j}}\chi)(\xi)$, $\xi \in\mathbb{R}^{n}$. Given that $|\nabla_{x}|^{\alpha}\chi_{j}(D)\in OPS_{\sigma,\omega}^{0,0}$, we can argue as in the estimate for \eqref{ctildeestimate} above to deduce
 \begin{equation*}
  \begin{aligned}
|\Pi_{2,1,1,2}|+|\Pi_{2,1,1,3}|\lesssim \int_{\mathbb{R}^n} u^2 \phi^2 \,  \mathrm{d}x,
\end{aligned}   
\end{equation*}
where the implicit constant depends on $|\alpha|$, $\gamma$ and $\sigma$. Since $\partial_{x_{j}} |\nabla_{x}|^{\alpha-2}(1-\chi(D))$  is a skew-symmetric operator for $1 \leq j \leq n$, it is straightforward to verify that
\begin{equation*}
\begin{split}
\Pi_{2,1,1,1} &= -\frac{\gamma \sigma_{1} \alpha}{4} \sum_{j=1}^{n} \sigma_{j} \int_{\mathbb{R}^{n}} u \left[ \partial_{x_{j}}|\nabla_{x}|^{\alpha-2}(1-\chi(D)) ; \left( \frac{\sigma \cdot x + \omega}{\langle \sigma \cdot x + \omega \rangle^2} \right) \phi^{2} \right] u \, \mathrm{d}x.
\end{split}
\end{equation*}
Nevertheless, we have
\begin{equation*}
\begin{split}
\left[ \partial_{x_{j}}|\nabla_{x}|^{\alpha-2}(1-\chi(D)) ; \left( \frac{\sigma \cdot x + \omega}{\langle \sigma \cdot x + \omega \rangle^2} \right) \phi^{2} \right] \in OPS^{0,-2\gamma}_{\sigma,\omega},
\end{split}
\end{equation*}
which, by continuity (see Theorem \ref{continuity}), implies that
\begin{equation*}
|\Pi_{2,1,1,1}| \lesssim \left( \frac{\alpha |\sigma_{1}| \gamma}{2} \sum_{j=1}^{n} |\sigma_{j}| \right) \int_{\mathbb{R}^{n}} u^{2} \phi^{2} \, \mathrm{d}x.
\end{equation*}
For $\Pi_{2,2} $, we write it as follows:
\begin{equation}\label{finalconmE}
\begin{split}
\Pi_{2,2} &= -\sum_{j=1}^{n} \frac{\alpha \sigma_{j}}{2} \int_{\mathbb{R}^{n}} u \phi |\nabla_{x}|^{\alpha-2}(1-\chi(D)) D_{1} D_{j} (u \phi) \, \mathrm{d}x \\
&\quad + \sum_{j=1}^{n} \frac{\alpha \sigma_{j}}{2} \int_{\mathbb{R}^{n}} u \phi \left[ |\nabla_{x}|^{\alpha-2}(1-\chi(D)) D_{1} D_{j}; \phi \right] u \, \mathrm{d}x \\
&= \Pi_{2,2,1} + \Pi_{2,2,2}.
\end{split}
\end{equation}
The term $\Pi_{2,2,1}$ is the high-frequency part of the term in \eqref{lfre}, which, together with $\Pi_{2,1,2}$ will be estimated later.

Next, we need to handle $\Pi_{2,2,2} $, from which we observe that Poisson's Bracket yields the main symbol of the commutator expansion:
\begin{equation*}
\begin{split}
&\left[ |\nabla_{x}|^{\alpha-2}(1-\chi(D)) D_{1} D_{j}; \phi \right] \\
&= \gamma(\alpha-2) \sum_{m=1}^{n} \sigma_{m} \left( \frac{\sigma \cdot x + \omega}{\langle \sigma \cdot x + \omega \rangle^2} \right) \phi(x) |\nabla_{x}|^{\alpha-4} \partial_{x_m}D_{1} D_{j} (1-\chi(D)) \\
&\quad - \gamma \sum_{m=1}^{n} \sigma_{m}  \phi(x) \left( \frac{\sigma \cdot x + \omega}{\langle \sigma \cdot x + \omega \rangle^2} \right) |\nabla_{x}|^{\alpha-2} \chi_{m}(D) \partial_{x_1} D_{j} \\
&\quad + \gamma \sigma_{1} \phi(x) \left( \frac{\sigma \cdot x + \omega}{\langle \sigma \cdot x + \omega \rangle^2} \right)  |\nabla_{x}|^{\alpha-2} (1-\chi(D)) \partial_{x_{j}} \\
&\quad + \gamma \sigma_{j} \phi(x) \left( \frac{\sigma \cdot x + \omega}{\langle \sigma \cdot x + \omega \rangle^2} \right)  |\nabla_{x}|^{\alpha-2} (1-\chi(D)) \partial_{x_{1}} + s_{\alpha-2,j}(x,D),
\end{split}
\end{equation*}
where $s_{\alpha-2,}(x,D)\in OPS_{\sigma,\omega}^{0,-2\gamma}$. Thus, we write
\begin{equation*}
\begin{split}
\Pi_{2,2,2} =& \frac{\gamma \alpha(\alpha-2)}{2}\sum_{j,m=1}^{n} \sigma_{j} \sigma_{m} \int_{\mathbb{R}^{n}} u \phi^2 \left( \frac{\sigma \cdot x + \omega}{\langle \sigma \cdot x + \omega \rangle^2} \right) |\nabla_{x}|^{\alpha-4} \partial_{x_{m}}D_{1} D_{j}  (1-\chi(D)) u \, \mathrm{d}x \\
&- \frac{\gamma \alpha}{2} \sum_{j,m=1}^{n} \sigma_{j} \sigma_{m} \int_{\mathbb{R}^{n}} u \phi^2 \left( \frac{\sigma \cdot x + \omega}{\langle \sigma \cdot x + \omega \rangle^2} \right) |\nabla_{x}|^{\alpha-2} \chi_{m}(D) \partial_{x_1} D_{j} u \, \mathrm{d}x \\
&+ \frac{\gamma \sigma_{1} \alpha}{2} \sum_{j=1}^{n} \sigma_{j} \int_{\mathbb{R}^{n}} u \phi^2 \left( \frac{\sigma \cdot x + \omega}{\langle \sigma \cdot x + \omega \rangle^2} \right) |\nabla_{x}|^{\alpha-2} (1-\chi(D)) \partial_{x_{j}} u \, \mathrm{d}x \\
& + \frac{\gamma \alpha}{2} \sum_{j=1}^{n} \sigma_{j}^{2} \int_{\mathbb{R}^{n}} u \phi^2 \left( \frac{\sigma \cdot x + \omega}{\langle \sigma \cdot x + \omega \rangle^2} \right)  |\nabla_{x}|^{\alpha-2} (1-\chi(D)) \partial_{x_{1}} u \, \mathrm{d}x\\
& +\frac{\alpha}{2}\sum_{j=1}^n \sigma_j \int_{\mathbb{R}^n} u\phi\, s_{\alpha-2.j}(x,D) u\, \mathrm{d}x.
\end{split}
\end{equation*}
At this point, the estimation of $\Pi_{2,2,2}$ follows similarly to the arguments in \eqref{decompoperA} and \eqref{decompoperAC}, where the antisymmetry of certain operators is used to produce commutators that reduce the regularity and decay requirements on the symbols involved. To avoid repetition, we summarize the result as follows:
\begin{equation*}
	|\Pi_{2,2,2}| \lesssim \int_{\mathbb{R}^n} u^2 \phi^2 \, \mathrm{d}x,
\end{equation*}
where the implicit constant depends on $\sigma$ and $\gamma$.

Combining $\Pi_{1,1,1}$ in \eqref{part1smoothing} with $\Pi_{2,1,2}$ in \eqref{highfredecom}, and $\Pi_{2,2,1}$ in \eqref{finalconmE} with $\Pi_{1,1,2}$ in \eqref{lfre}, we get  
 \begin{equation*}
 \begin{split}
 -\frac{1}{2}&\int_{\mathbb{R}^{n}}u \left[\partial_{x_{1}}|\nabla_{x}|^{\alpha}; \varphi\right]u \,\mathrm{d}x\\
 =&-\frac{\sigma_{1}}{2}\int_{\mathbb{R}^{n}} u\phi |\nabla_x|^{\alpha}(u\phi)\,\mathrm{d}x-\sum_{j=1}^n \frac{\alpha \sigma_j}{2}\int_{\mathbb{R}^n} u\phi |\nabla_x|^{\alpha-2}D_1D_j(u\phi)\, \mathrm{d}x+O\left(\left\|u\phi \right\|_{L^{2}_{x}}^{2}\right)\\
 =&-\frac{1}{2}\sum_{j=1}^n\sigma_1 \int_{\mathbb{R}^{n}} |\nabla_x|^{\frac{\alpha-2}{2}}D_j(u\phi) |\nabla_x|^{\frac{\alpha-2}{2}} D_j(u\phi)\,\mathrm{d}x\\
 &-\frac{1}{2}\sum_{j=1}^n \alpha \sigma_j\int_{\mathbb{R}^n} |\nabla_x|^{\frac{\alpha-2}{2}}D_1(u\phi) |\nabla_x|^{\frac{\alpha-2}{2}}D_j(u\phi)\, \mathrm{d}x+O\left(\left\|u\phi \right\|_{L^{2}_{x}}^{2}\right),
 \end{split}
 \end{equation*}
where we have used that $|\nabla_x|^2=\sum_{j=1}^n D_j^2$. 

Thus, setting 
\begin{equation*}
		M=\begin{pmatrix}
			(1+\alpha)\sigma_{1} & \frac{\alpha \sigma_{2}}{2}& \frac{\alpha \sigma_{3}}{2} &\cdots &\frac{\alpha \sigma_{n}}{2}\\
			 \frac{\alpha \sigma_{2}}{2} & \sigma_{1}&0&\cdots &0\\
		 \frac{\alpha \sigma_{3}}{2}&0 &\sigma_{1}&\cdots&0\\
			\vdots&\vdots&\vdots&\ddots&\vdots \\
		 \frac{\alpha \sigma_{2}}{2} & 0&\vdots&\cdots &\sigma_{1},
		\end{pmatrix}
	\end{equation*} 
and denoting by $\langle x,y \rangle_M=x^{T}My$ with $x,y \in \mathbb{R}^n$, we conclude
\begin{equation}\label{eqconcls1}
 \begin{split}
 -\frac{1}{2}\int_{\mathbb{R}^{n}}u \left[\partial_{x_{1}}|\nabla_{x}|^{\alpha}; \varphi\right]u \,\mathrm{d}x &=-\frac{1}{2} \int_{\mathbb{R}^n}\langle \vec{u},\vec{u} \rangle_{M}\, \mathrm{d}x+O\left(\left\|u\phi \right\|_{L^{2}_{x}}^{2}\right),
 \end{split}
 \end{equation}
 where $\vec{u}=(|\nabla_x|^{\frac{\alpha-2}{2}} D_1(u\phi),\dots,|\nabla_x|^{\frac{\alpha-2}{2}} D_n(u\phi))^{T}$. The condition \eqref{sigmacond} implies that $M$ is a symmetric positive definite matrix, in whose case shows us
 \begin{equation*}
 \langle \vec{u},\vec{u} \rangle_{M}\sim |\vec{u}|^2=\sum_{j=1}^n \big(|\nabla_x|^{\frac{\alpha-2}{2}}D_j(u\phi)\big)^2.  
 \end{equation*}
Inserting the above expression into \eqref{eqconcls1}, we obtain \eqref{main1}. Hence, the proof of the lemma is complete.
\end{proof}

We are required to work with a rescaled version of the weighted functions $\varphi$ and $\phi$ defined in \eqref{p2} and \eqref{p1}, respectively. More precisely, for $M>0$, we define 
$$f_M(x)=f\left(\frac{x}{M}\right),$$ 
and consider $\varphi_{M}(x)$ and $\phi_{M}(x)$. As a consequence of Lemma \ref{takamajaka}, we get:
 \begin{cor}\label{Maincorollary}
 	Let $\alpha \in [1,2)$, $\gamma \in \left(\frac{1}{2},\frac{ \alpha+1}{2}\right]$, and $\omega \in \mathbb{R}$. Consider $\sigma = (\sigma_1, \dots, \sigma_n) \in \mathbb{R}^n$ satisfying condition \eqref{sigmacond}. Then, there exist positive constants $c_1$ and $c_2$, depending only on $\alpha$, $n$, $\sigma$, and $\gamma$, such that for any $u \in \mathcal{S}(\mathbb{R}^n)$, the following inequality holds:
 	\begin{equation}\label{maineq2}
 		\int_{\mathbb{R}^n} \left( \partial_{x_1} |\nabla_x|^\alpha u \right) u \, \varphi_M \, \mathrm{d}x 
 		\leq -c_1 \int_{\mathbb{R}^n} \left| |\nabla_x|^{\frac{\alpha}{2}} \left( u \sqrt{ \partial_{x_1} (\varphi_M) } \right) \right|^2 \mathrm{d}x 
 		+ \frac{c_2}{M^\alpha} \int_{\mathbb{R}^n} u^2 \, \partial_{x_1} (\varphi_M) \, \mathrm{d}x.
 	\end{equation}
 	In particular, in dimension $n=1$, one may take $c_1 = \frac{\alpha+1}{2}$, while in higher dimensions $n \geq 2$, if $\sigma_j = 0$ for all $2 \leq j \leq n$, the constant $c_1$ can be taken as $c_1 = \frac{1}{2}$.
 \end{cor}
 \begin{proof}
Applying Lemma \ref{takamajaka} to $v_{M}(x)=u(Mx)$, using that $\partial_{x_1}(\varphi_M)(x)=\frac{\sigma_1}{M}(\phi_M)^2(x)$, i.e., $\phi_{M}(x)=\sigma_1^{-\frac{1}{2}}M^{\frac{1}{2}}\sqrt{\partial_{x_1}(\varphi_M)(x)}$, and changing variables, one obtains \eqref{maineq2} from \eqref{main1}.
 \end{proof}

%%%%%%%%%%%%%%%%%%%%%%%%%%%%%%%%%%%%%%%%%%%%%%%%%%%%%%%%%%%%%%%%%%%%%%%%%%%%%%%%%%%%%%%%%%%%%%%%%
 
\section{Instability results}\label{SectInstaResults}

We now focus on the deduction of the instability result in Theorem \ref{MainInsta}. For the proof, we adapt the strategy for $L^2$-critical equations established by Martel and Merle \cite{MartelMerle2001}, and the work of Farah, Holmer and Roudenko \cite{FarahHolmerRoudenko2019KdV,FarahHolmerRoudenko2019}. We recall that all our instability results are applied on spatial dimension $n=2$, which corresponds to model \eqref{E:HBO}. 
\subsection{The linearized equation around \texorpdfstring{$Q$}{}}\label{LinearizedSection}
We use the following canonical decomposition of $u$ around $Q$
\begin{equation}\label{eqlineareq1}
    \epsilon(x_1,x_2,t)=\lambda(t)\,  u(\lambda(t)x_1+z_1(t),\lambda(t)x_2+z_2(t),t)-Q(x_1,x_2), \quad (x_1,x_2)\in \mathbb{R}^2.
\end{equation}
We perform the time rescaling $t \mapsto s$ via the relation
\begin{equation}
	\frac{ds}{dt} = \frac{1}{\lambda^2}.
\end{equation}
\begin{lem}[Equation for $\epsilon$]\label{EquationEpsi} Let $L$ be the operator  \eqref{linearizedop}, and $\Lambda$ be given by \eqref{scalingOp}. For all $s\geq 0$, we have
\begin{equation}\label{Epsiloneq}
\begin{aligned}
  \partial_s\epsilon=\partial_{x_1}(L \epsilon)&+\frac{1}{\lambda}\frac{d \lambda}{ds}\Lambda Q+\Big(\frac{1}{\lambda}\frac{d z_1}{ds}-1\Big)\partial_{x_1}Q+\frac{1}{\lambda}\frac{d z_2}{ds}\partial_{x_2}Q\\
    &+\frac{1}{\lambda}\frac{d \lambda}{ds}\Lambda \epsilon+\Big(\frac{1}{\lambda}\frac{d z_1}{ds}-1\Big)\partial_{x_1}\epsilon+\frac{1}{\lambda}\frac{d z_2}{ds}\partial_{x_2}\epsilon\\
    &-\frac{1}{2}\partial_{x_1}(\epsilon^2).
\end{aligned}
\end{equation}
\end{lem}
\begin{proof}
The proof follows from differentiating \eqref{eqlineareq1} and the fact that $u$ solves \eqref{E:HBO}. Note that the equation above is justified in the distributional sense.
\end{proof}
Next, we deduce the mass and energy conservation laws for $\epsilon$ in \eqref{eqlineareq1}.
\begin{lem}\label{Energyconsrlemma} For any $s\geq 0$, we have the following mass and energy conservations for $\epsilon$
\begin{equation}
    \mathcal{M}[\epsilon(s)]:=2\int_{\mathbb{R}^2} Q(x)\epsilon(x,s)\, \mathrm{d}x +\int_{\mathbb{R}^2} (\epsilon(x,s))^2\, \mathrm{d}x=\mathcal{M}_0
\end{equation}
where 
\begin{equation}\label{Masszero}
    \mathcal{M}_0=2\int_{\mathbb{R}^2} Q(x)\epsilon(x,0)\, \mathrm{d}x+\int_{\mathbb{R}^2} (\epsilon(x,0))^2\, \mathrm{d} x.
\end{equation}
It also follows
\begin{equation}\label{energyscalingrel}
    E[Q+\epsilon(s)]=\lambda(s)\, E[u_0],
\end{equation}
where $E[\cdot]$ is defined in \eqref{E:energy}. Moreover, the energy linearization is
\begin{equation}\label{EnergyLine1}
   \begin{aligned}
     E[Q+\epsilon]=&  \frac{1}{2}(L\epsilon,\epsilon)-\Big(\int_{\mathbb{R}^2} Q\epsilon \, \mathrm{d}x+\frac{1}{2}\int_{\mathbb{R}^2} \epsilon^2\, \mathrm{d}x\Big)-\frac{1}{6}\int_{\mathbb{R}^2} \epsilon^3 \, \mathrm{d}x,
     \end{aligned}
\end{equation}
where the last term of the identity above satisfies
\begin{equation}\label{EnergyLine2}
  \|\epsilon\|_{L^3}^3\lesssim \||\nabla_x|^{\frac{1}{2}}\epsilon\|_{L^2}^{2}\|\epsilon\|_{L^2}.  
\end{equation}
\end{lem}

\begin{proof} The proof follows from similar arguments in the deduction of the energy and mass conservation for $\epsilon$ in \cite{MartelMerle2001,FarahHolmerRoudenko2019KdV,FarahHolmerRoudenko2019}. We emphasize that estimate \eqref{EnergyLine2} is a consequence of the \emph{fractional Gagliardo-Nirenberg inequality}, which  states that given $f\in H^{\frac{1}{2}}(\mathbb{R}^2)$,  then $f\in L^{p}(\mathbb{R}^2)$ for all $2<p<4$, and it follows 
\begin{equation}\label{G_N}
  \Vert f\Vert^{p}_{L^{p}(\mathbb{R}^2)}\lesssim \| |\nabla_x|^{\frac{1}{2}}f\|^{2(p-2)}_{L^2(\mathbb{R}^2)}\| f\|^{p-2(p-2)}_{L^2(\mathbb{R}^2)}.
\end{equation}
We also remark, that the deduction of \eqref{EnergyLine1} requires the following identity
\begin{equation*}
    \frac{1}{2}\int_{\mathbb{R}^2} \big||\nabla_x|^{\frac{1}{2}}Q\big|^2\, \mathrm{d}x=\frac{1}{6}\int_{\mathbb{R}^2} Q^3\, \mathrm{d}x,
\end{equation*}
which is obtained via Pohozaev identities, its deduction can be consulted in \cite[Remark 2]{RianoRoudenkoYang2022}. 
\end{proof}

%%%%%%%%%%%%%%%%%%%%%%%%%%%%%%%%%%%%%%%%%%%%%%%%%%%%%%%%%%%%%%%%%%%%%%%%%%%%%%%%%%%%%%%%%%%%%%%%%%%%%%%%%%%%%%%%%%%%%%%%%%%%%%%%%%%%%%%%%%%%%%%%%%%%%%%%%%%%%%%%%%%%%%%%%%%%%%%%%%%%%
\subsection{Modulation theory and parameter estimates}
This part, concerns the construction of the scaling parameter $\lambda(s)$ and the translation $z(s)=(z_1(s),z_2(s))$ in \eqref{eqlineareq1} such that $\epsilon(s)\perp \psi_0$, $\epsilon(s) \perp \partial_{x_1}Q$, $\epsilon(s) \perp \partial_{x_2}Q$.
\begin{prop} \label{propmodpar}
(i) There exists $\omega,\lambda_0>0$ and a unique $C^1$ map
    \begin{equation*}
        (\lambda,z)=(\lambda,z_1,z_2): U_{\omega}\rightarrow (1-\lambda_0,1+\lambda_0)\times \mathbb{R}^2
    \end{equation*}
such that if $u\in U_{\omega}$ and $\epsilon_{\lambda,z}$ is given by
\begin{equation}\label{defiepsilonpPara}
    \epsilon_{\lambda,z}(x)=\lambda u (\lambda x+z)-Q(x),
\end{equation}
then
\begin{equation}\label{orthocond1}
    \epsilon_{\lambda,z}\perp \psi_0, \qquad \epsilon_{\lambda,z}\perp \partial_{x_1}Q, \quad \text{ and }\quad \epsilon_{\lambda,z}\perp \partial_{x_2}Q. 
\end{equation}
Moreover, there exist a constant $C_1>0$, such that if $u\in U_{\omega_1}$ with $0<\omega_1<\omega$, then 
\begin{equation}\label{paradecay}
\|\epsilon_{\lambda,z}\|_{H^{\frac{1}{2}}}\leq C_1\omega_1\, \, \text{ and } \, |\lambda-1|+|z|\leq C_1\omega_1.    
\end{equation}
\\
(ii) If in addition, $u$ is cylindrically symmetric (i.e., $u(x_1,x_2)=u(x_1,|x_2|)$), then, taking $\omega>0$ smaller, if necessary, it follows $z_2=0$.
\end{prop}
\begin{proof}
The proof follows from the implicit function theorem, we refer to 
\cite{FarahHolmerRoudenko2019,MartelMerle2001,Bouard1996}.
\end{proof}
Assuming that $u(t)\in U_{\omega}$ for all $t\geq 0$. We defined the functions $\lambda(t)$ and $z(t)$ as follows.
\begin{defn}\label{modparamdef}
For all $t\geq 0$, let $\lambda(t)$ and $z(t)$ be such that $\epsilon_{\lambda(t),z(t)}$, defined according to \eqref{defiepsilonpPara}, satisfy
\begin{equation*}
\epsilon_{\lambda(t),z(t)}\perp \psi_0, \, \, \epsilon_{\lambda(t),z(t)}\perp \partial_{x_1}Q=0 \, \, \text{ and } \, \, \epsilon_{\lambda(t),z(t)}\perp \partial_{x_2}Q.
\end{equation*}
In this case, we also define
\begin{equation*}
\epsilon(t)=\epsilon_{\lambda(t),z(t)}(x)=\lambda(t)\, u(\lambda(t) x+z(t))-Q(x), \quad x\in \mathbb{R}^2.
\end{equation*}
\end{defn}
Working with the rescaling  $t\mapsto s$ given by $\frac{ds}{dt}=\frac{1}{\lambda^2}$, we can deduce the equation for $\frac{d\lambda}{ds}$ and $\frac{dz_1}{ds}$.

\begin{lem}[Modulation Parameters]\label{modparaL}
There exists $0<\omega_1<\omega$ ($\omega>0$ as in Proposition \ref{propmodpar}) such that if $u\in C([0,\infty),H^{\frac{1}{2}}(\mathbb{R}^2))$ solves \eqref{E:HBO} is cylindrical symmetric, and for all $t\geq 0$, $u(t)\in U_{\omega_1}$, then $\lambda$ and $z_1$ are $C^1$ functions and they satisfy the following equations:
\begin{equation}\label{modulationEq}
\begin{aligned}
-\frac{1}{\lambda}\frac{d\lambda}{ds}\int_{\mathbb{R}^2} \big( x\cdot \nabla \partial_{x_1}Q\big)\epsilon\, \mathrm{d}x &+\Big(\frac{1}{\lambda}\frac{d z_1}{ds}-1\Big)\Big(\int_{\mathbb{R}^2} |\partial_{x_1}Q|^2\, \mathrm{d}x-\int_{\mathbb{R}^2} \partial_{x_1}^2Q \epsilon \, \mathrm{d}x\Big)\\
&=\int_{\mathbb{R}^2} (\partial_{x_1}Q)^2\epsilon \, \mathrm{d}x-\frac{1}{2}\int_{\mathbb{R}^2} \epsilon^2 \partial_{x_1} Q\, \mathrm{d}x,
\end{aligned}    
\end{equation}
and 
\begin{equation}\label{modulationEq1}
\begin{aligned}
\frac{1}{\lambda}\frac{d\lambda}{ds}\Big(\frac{1}{\mu_0}\int_{\mathbb{R}^2} \psi_0 Q\, \mathrm{d}x-&\int_{\mathbb{R}^2} \big(x\cdot \nabla \psi_0\big)\epsilon\, \mathrm{d}x\Big)-\Big(\frac{1}{\lambda}\frac{dz_1}{ds}-1\Big)\int_{\mathbb{R}^2} \partial_{x_1}\psi_0 \epsilon\, \mathrm{d}x\\
=&\int L(\partial_{x_1}\psi_0 )\epsilon\, \mathrm{d}x-\frac{1}{2}\int_{\mathbb{R}^2} \partial_{x_1} \psi_0 \epsilon^2 \mathrm{d}x.
\end{aligned}    
\end{equation}
Moreover, there exists a universal constant $C_2>0$ such that if $\|\epsilon(s)\|_{L^2}\leq \widetilde{\omega}$, for all $s\geq 0$, where $\widetilde{\omega}<\omega_1$, then
\begin{equation}\label{boundparams}
  \Big|\frac{1}{\lambda}\frac{d \lambda}{ds}\Big|+ \Big|\frac{1}{\lambda}\frac{d z_1}{ds}-1\Big|\leq C_2\big(\|\epsilon\|_{L^2}+\|\epsilon\|_{L^2}^2\big).   
\end{equation}
\end{lem}
\begin{proof} The proof follows from deriving $\int_{\mathbb{R}^2} \partial_{x_1}Q\, \epsilon(s)\,\mathrm{d}x$ and $\int_{\mathbb{R}^2} \psi_0\, \epsilon(s)\,\mathrm{d}x$ with respect to $s$, and then using the equation in Lemma \ref{EquationEpsi}  with $z_2=0$. Here, Assumption \ref{Assm1} is important to justify such a procedure via smooth functions, and then take the limit. Thus, Lemma \ref{modparaL} follows closely the arguments in \cite{MartelMerle2001,FarahHolmerRoudenko2019}, so we omit its derivation. It is worth noting that if $\|\epsilon\|_{H^{\frac{1}{2}}}$ is small (granted taking $\omega_1$ small), the estimates \eqref{modulationEq} and \eqref{modulationEq1} lead to a system of equations that establishes \eqref{boundparams}. 
\end{proof}
\subsection{Virial-Type Estimates}\label{VirialEstSec}
Let us now deduce some key virial-type estimates. But first, we recall the function $\chi\in C^{\infty}_c(\mathbb{R})$ being such that $\chi(x)=1$, if $|x|\leq 1$, $\chi(x)=0$, whenever $|x|\geq 2$. Given $A\geq 1$, we also define
\begin{equation*}
   \chi_A(x)=\chi\Big(\frac{x}{A}\Big), \quad x\in \mathbb{R}.
\end{equation*}
Consider the function
\begin{equation}\label{Fdef}
   F(x_1,x_2)=\int_{-\infty}^{x_1}\Lambda Q(z,x_2)\, dz, \quad x=(x_1,x_2)\in \mathbb{R}^2.
\end{equation}
We obtain the following result from the decay properties of $Q$ in \eqref{Qdecay}. 
\begin{prop}\label{PropdecayF}
Let $F$ be given as in \eqref{Fdef}, $l,k\in \mathbb{Z}$ with $l,k\geq 0$. Assume that $k\geq 1$, then
\begin{equation*}
\begin{aligned}
|\partial_{x_1}^k\partial_{x_2}^{l}F(x_1,x_2)| \lesssim \frac{1}{\langle x\rangle^{3+(k-1)+l}}, \quad x=(x_1,x_2)\in \mathbb{R}^2.  
\end{aligned}
\end{equation*}
In the case $k=0$, let $\alpha_1+\alpha_2=3+l$ with $\alpha_2>1$, it follows
\begin{equation*}
\begin{aligned}
|\partial_{x_2}^{l}F(x_1,x_2)| \lesssim \frac{1}{\langle x_2 \rangle^{\alpha_1}}\Big|\int_{-\infty}^{x} \frac{1}{\langle z \rangle^{\alpha_2}}\, dz \Big|, 
\end{aligned}
\end{equation*}
with implicit constant depending on $\alpha_1,\alpha_2$ and independent of $x$, which implies
\begin{equation}\label{decayFinyvar0}
\begin{aligned}
\sup_{x\in \mathbb{R}}|\partial_{x_2}^{l}F(x_1,x_2)| \lesssim \frac{1}{\langle x_2 \rangle^{\alpha_1}}. 
\end{aligned}
\end{equation}
Additionally, if $x_1<0$
\begin{equation}\label{decayFinyvar}
\begin{aligned}
|\partial_{x_2}^{l}F(x_1,x_2)| \lesssim \frac{1}{\langle x_2 \rangle^{\alpha_1}}\frac{1}{\langle x_1 \rangle^{\alpha_2}}\Big|\int_{-\infty}^{x_1} \frac{1}{\langle z \rangle^{\alpha_3}}\, dz \Big|,  
\end{aligned}
\end{equation}
where $\alpha_1+\alpha_2+\alpha_3=3+l$ with $\alpha_3>1$, and the implicit constant is independent of $x_1<0$.
\end{prop}
As a consequence of the previous result, we have $F\chi_A\in H^s(\mathbb{R}^2)$ for all $s\geq 0$, that is,
\begin{equation}\label{regcond}
    F\chi_A\in \bigcap_{s\geq 0}H^s(\mathbb{R}^2)=H^{\infty}(\mathbb{R}^2).
\end{equation}
Next, we define the functional
\begin{equation}\label{Jfunctional}
\begin{aligned}
J_A(s)=\int_{\mathbb{R}^2}\epsilon(x_1,x_2)F(x_1,x_2)\chi_A(x_1)\,\mathrm{d}x.  
\end{aligned}   
\end{equation}
By Proposition \ref{PropdecayF}, Cauchy-Schwarz inequality, and support considerations of the function $\chi_A$, we deduce
\begin{equation}\label{Jesteq}
\begin{aligned}
|J_A(s)|\leq & \|\epsilon(s)\|_{L^2}\|F\chi_A\|_{L^2}\\
\lesssim &\|\epsilon(s)\|_{L^2}\Big(\|\chi_{A}\|_{L^{\infty}}\|F\|_{L^2((-\infty,0)\times \mathbb{R}))}+\|\sup_{x_1}|F(x_1,\cdot)|\|_{L^2(\mathbb{R})}\|\chi_A\|_{L^2([0,\infty))}\Big)\\
\lesssim &(1+A^{\frac{1}{2}})\|\epsilon(s)\|_{L^2},
\end{aligned}   
\end{equation}
where the implicit constant above is independent of $A\geq 1$. 
\begin{lem}\label{eqJA}
 Suppose that $\epsilon(s)\in H^{\frac{1}{2}}(\mathbb{R}^2)$ for all $s\geq 0$. Let $\gamma_1\in (0,1)$ be fixed. Then the function $s\mapsto J_A(s)$ is $C^1$ and
\begin{equation*}
\begin{aligned}
\frac{d}{ds}J_A=&-\frac{1}{\lambda}\frac{d \lambda}{ds}\Big(J_A-\frac{1}{2}\int_{\mathbb{R}} \Big( \int_{\mathbb{R}} \Lambda Q(x_1,x_2)\, \mathrm{d}x_1\Big)^2\, \mathrm{d}x_2\Big)+\Big(1-\Big(\frac{1}{\lambda}\frac{dz_1}{ds}-1\Big)\Big)\int_{\mathbb{R}^2} \epsilon Q\, \mathrm{d}x \\
&+R(\epsilon,A),
\end{aligned} 
\end{equation*}
where
\begin{equation}\label{remaindereq}
\begin{aligned}
|R(\epsilon,A)|\lesssim & \Big(1+\frac{1}{A}\Big)\|\epsilon\|^2_{L^2}+\Big(\frac{1}{A}+\frac{1}{A^{\frac{1}{2}}}\Big)\|\epsilon\|_{L^2} \\
&+ \Big|\frac{1}{\lambda}\frac{d \lambda}{ds}\Big|\bigg(\frac{1}{A^{\gamma_1}}+\|\epsilon\|_{L^2} +A^{\frac{1}{2}}\|\epsilon\|_{L^2([A,\infty)\times \mathbb{R})}+\Big|\int_{\mathbb{R}^2} \epsilon\,  x_2 \partial_{x_2}F \chi_A\, \mathrm{d}x\Big|\bigg)   \\
&+\Big|\frac{1}{\lambda}\frac{d z_1}{ds}-1\Big|\bigg(\frac{1}{A}+\Big(1+\frac{1}{A^{\frac{1}{2}}}+\frac{1}{A}\Big)\|\epsilon\|_{L^2}\bigg),
\end{aligned}   
\end{equation}
where the implicit constant above is independent of $A\geq 1$.
\end{lem}
\begin{proof}
The proof is similar to that given in \cite[Lemma 6.1]{FarahHolmerRoudenko2019}. Let us highlight the differences that arise from working with the polynomial decay of $Q$ and the fact that $L$ has a non-local operator. Since the equation in Lemma \ref{EquationEpsi} with $z_2=0$ is defined in the distributional sense, $F\chi_A\in H^{\infty}(\mathbb{R}^2)$ with $|(x_1,x_2)| F\chi_A\in L^2(\mathbb{R}^2)$ (due to Proposition \ref{PropdecayF}), we can multiply \eqref{Epsiloneq} by $F\chi_A$ and integrate in space to deduce
 \begin{equation}
 \begin{aligned}
 \frac{d}{ds}J_A(s)=&\int_{\mathbb{R}^2} \Big(\partial_{x_1}(L\epsilon)+\frac{1}{\lambda}\frac{d\lambda}{ds}\Lambda \epsilon+\Big(\frac{1}{\lambda}\frac{d z_1}{ds}-1\Big)\partial_{x_1}\epsilon\Big)F\chi_A\, \mathrm{d}x\\
 &+\int_{\mathbb{R}^2} \Big(\frac{1}{\lambda}\frac{d\lambda}{ds}\Lambda Q+\Big(\frac{1}{\lambda}\frac{d z_1}{ds}-1\Big)\partial_{x_1}Q\Big)F\chi_A\, \mathrm{d}x-\frac{1}{2}\int_{\mathbb{R}^2} \partial_{x_1}(\epsilon^2) F\chi_A\, \mathrm{d}x\\
 =:& \mathcal{I}+\mathcal{II}+\mathcal{III}.
 \end{aligned}    
 \end{equation}
 Integration by parts yields
\begin{equation*}
\begin{aligned}
|\mathcal{III}|\leq  \frac{1}{2}\int_{\mathbb{R}^2} \epsilon^2 |\partial_{x_1} F||\chi_A|\, \mathrm{d}x+\frac{1}{2A}\int_{\mathbb{R}^2} \epsilon^2 |F||\chi'_A|\, \mathrm{d}x\lesssim &  \Big(\|\partial_{x_1} F\|_{L^{\infty}}\|\chi\|_{L^{\infty}}+\frac{\|F\|_{L^{\infty}}\|\chi'\|_{L^{\infty}}}{A}\Big)\|\epsilon\|_{L^2}^2\\
\lesssim & \Big(1+\frac{1}{A}\Big)\|\epsilon\|_{L^2}^2.
\end{aligned}   
\end{equation*}
We divide the estimate of $\mathcal{II}$ as follows
\begin{equation}
\begin{aligned}
\mathcal{II}=&\frac{1}{\lambda}\frac{d\lambda}{ds} \int_{\mathbb{R}^2} \Lambda Q F\chi_A\, \mathrm{d}x +\Big(\frac{1}{\lambda}\frac{d z_1}{ds}-1\Big)\int_{\mathbb{R}^2} \partial_{x_1}Q F\chi_A\, \mathrm{d}x \\
=:&\frac{1}{\lambda}\frac{d\lambda}{ds}\mathcal{II}_1+\Big(\frac{1}{\lambda}\frac{d z_1}{ds}-1\Big)\mathcal{II}_2.
\end{aligned}    
\end{equation}
Using that $\partial_{x_1}F=\Lambda Q$, we write
\begin{equation*}
\begin{aligned}
\mathcal{II}_1=\frac{1}{2}\int_{\mathbb{R}^2} \partial_{x_1}(F^2)\chi_A\, \mathrm{d}x =&\frac{1}{2}\int_{\mathbb{R}^2}  \partial_{x_1}(F^2)\, \mathrm{d}x+\frac{1}{2}\int_{\mathbb{R}^2}  \partial_{x_1}(F^2)(\chi_A-1)\, \mathrm{d}x\\
   =&\frac{1}{2}\int_{\mathbb{R}}  \Big(\int_{\mathbb{R}}  \Lambda Q(x_1,x_2)\, \mathrm{d}x_1\Big)^2\, \mathrm{d}x_2 +\frac{1}{2}\int_{\mathbb{R}^2}  \partial_{x_1}(F^2)(\chi_A-1)\, \mathrm{d}x.
\end{aligned}
\end{equation*}
Given that $\partial_{x_1}(F^2)=2(\Lambda Q)F$, and that $\chi_A-1$ is supported in the region $\{(x_1,x_2): |x_1|\geq A\}$, we chose $\gamma_1\in (0,1)$ fixed to obtain 
\begin{equation*}
\begin{aligned}
\Big|\int_{\mathbb{R}^2}  \partial_{x_1}(F^2)(\chi_A-1)\, \mathrm{d}x\Big| \lesssim & \sum_{k=0}^{\infty} \frac{2^{-\gamma_1 k}}{A^{\gamma_1}}\, \, \int_{2^{k}A\leq |x_1|\leq 2^{k+1}A} |x_1|^{\gamma_1}|\Lambda Q||F|\, \mathrm{d}x  \\
\lesssim& \sum_{k=0}^{\infty} \frac{2^{-\gamma_1 k}}{A^{\gamma_1}}\, \, \|F\|_{L^{\infty}}\||x|^{\gamma_1}\Lambda Q\|_{L^1}\lesssim  \frac{1}{A^{\gamma_1}},
\end{aligned}    
\end{equation*}
where we have used that $|\Lambda Q(x)|\lesssim \langle x\rangle^{-3}$, $x\in \mathbb{R}^n$, and thus $|x_1|^{\gamma_1}\Lambda Q\in L^1(\mathbb{R}^2)$, whenever $\gamma_1\in (0,1)$. This completes the study of $\mathcal{II}_1$. Next, we deal with $\mathcal{II}_2$. Combining integration by parts and the fact that Proposition \ref{scalingprop} establishes that $Q\perp \Lambda Q$,  we get
\begin{equation*}
\begin{aligned}
   \mathcal{II}_2=-\int_{\mathbb{R}^2}  Q \Lambda Q \chi_A\, \mathrm{d}x -\frac{1}{A}\int_{\mathbb{R}^2}  Q  F \chi'_A\, \mathrm{d}x= -\int_{\mathbb{R}^2}  Q \Lambda Q (\chi_A-1) \, \mathrm{d}x -\frac{1}{A}\int_{\mathbb{R}^2} Q  F \chi'_A\, \mathrm{d}x.
\end{aligned}    
\end{equation*}
The mean value inequality and the fact that $\chi(0)=1$ yield
\begin{equation}\label{phidiffere}
\begin{aligned}
 |\chi_A(x_1)-1|\leq \frac{\|\chi'\|_{L^{\infty}}}{A}|x_1|.   
\end{aligned}    
\end{equation}
It follows from H\"older's inequality that
\begin{equation*}
 \begin{aligned}
|\mathcal{II}_2|\lesssim \frac{1}{A}\big(\||x_1|Q\|_{L^2}\|\Lambda Q\|_{L^2} +\|\chi'\|_{L^{\infty}}\|F\|_{L^{\infty}}\|Q\|_{L^1} \big)\lesssim \frac{1}{A}.
 \end{aligned}   
\end{equation*}
The study of $\mathcal{II}$ is complete. Let us now focus on $\mathcal{I}$. By integration by parts, we  write
\begin{equation*}
 \begin{aligned}
\mathcal{I}=&-\Big(\int_{\mathbb{R}^2}  (L\epsilon ) \Lambda Q \chi_A \, \mathrm{d}x+\frac{1}{A}\int_{\mathbb{R}^2}  (L \epsilon)F\chi'_A\, \mathrm{d}x\Big)+\frac{1}{\lambda}\frac{d\lambda}{ds}\int_{\mathbb{R}^2}  \Lambda \epsilon F \chi_A\, \mathrm{d}x\\
&-\Big(\frac{1}{\lambda}\frac{dz_1}{ds}-1\Big)\Big(\int_{\mathbb{R}^2}  \epsilon \Lambda Q \chi_A\, \mathrm{d}x+\frac{1}{A}\int_{\mathbb{R}^2}  \epsilon F\chi'_A\, \mathrm{d}x\Big)\\
=:&\mathcal{I}_1+\frac{1}{\lambda}\frac{d\lambda}{ds}\mathcal{I}_2-\Big(\frac{1}{\lambda}\frac{dz_1}{ds}-1\Big)\mathcal{I}_3.
 \end{aligned}   
\end{equation*}
At this point, we will comment on the estimation of $\mathcal{I}_2$ and $\mathcal{I}_3$ due to their similarity to \cite{FarahHolmerRoudenko2019} and previous arguments. From the definition of $\Lambda$, it is seen that
\begin{equation*}
 \begin{aligned}
 \mathcal{I}_3=&\int_{\mathbb{R}^2}  \epsilon Q\, \mathrm{d}x+\int_{\mathbb{R}^2}  \epsilon Q(\chi_A-1)\, \mathrm{d}x+\int_{\mathbb{R}^2}  \epsilon \big(x\cdot \nabla Q\big)\chi_A \, \mathrm{d}x+\frac{1}{A}\int_{\mathbb{R}^2}  \epsilon F\chi'_A\, \mathrm{d}x\\
 =&\int_{\mathbb{R}^2}  \epsilon Q\, \mathrm{d}x+O\Big(\big(1+\frac{1}{A^{\frac{1}{2}}}+\frac{1}{A}\big)\|\epsilon\|_{L^2}\Big)
 \end{aligned}   
\end{equation*}
On the other hand, we have
\begin{equation*}
\begin{aligned}
 \mathcal{I}_2=\int_{\mathbb{R}^2}  \Lambda \epsilon F\chi_A\, \mathrm{d}x=&-\int_{\mathbb{R}^2}  \epsilon \Lambda (F\chi_A)\, \mathrm{d}x\\
 =&-J_A-\int_{\mathbb{R}^2}  \epsilon \, x_1  \Lambda Q \chi_A\, \mathrm{d}x-\int_{\mathbb{R}^2}  \epsilon\,  x_2 \partial_{x_2}F\chi_A\, \mathrm{d}x-\frac{1}{A}\int_{\mathbb{R}^2}   \epsilon \, x_1 F \chi_A'\, \mathrm{d}x\\
 =&-J_A-\int_{\mathbb{R}^2}  \epsilon\,  x_2 \partial_{x_2}F\chi_A\, \mathrm{d}x+O(\|\epsilon\|_{L^2})+O(A^{\frac{1}{2}}\|\epsilon\|_{L^2([A,\infty)\times \mathbb{R})}).
\end{aligned}    
\end{equation*}
Finally, we consider $\mathcal{I}_1$. Since $L$ is self-adjoint and $L(\Lambda Q)=-Q$, we deduce
\begin{equation*}
\begin{aligned}
\mathcal{I}_1=&-\int_{\mathbb{R}^2}  \epsilon L(\Lambda Q\chi_A)\, \mathrm{d}x-\frac{1}{A}\int_{\mathbb{R}^2}  \epsilon L(F\chi_A')\, \mathrm{d}x\\
=&\int_{\mathbb{R}^2}  \epsilon  Q\, \mathrm{d}x+\int_{\mathbb{R}^2}  \epsilon  Q(\chi_A-1)\, \mathrm{d}x-\int_{\mathbb{R}^2}  \epsilon \left[L;\chi_A\right]\Lambda Q\, \mathrm{d}x-\frac{1}{A}\int_{\mathbb{R}^2}  \epsilon |\nabla_x|(F\chi_A')\, \mathrm{d}x\\
&-\frac{1}{A}\int_{\mathbb{R}^2}  \epsilon F\chi_A'\, \mathrm{d}x+\frac{1}{A}\int_{\mathbb{R}^2}  \epsilon Q F \chi'_A\, \mathrm{d}x\\
=:&\int_{\mathbb{R}^2}  \epsilon  Q\, \mathrm{d}x+\sum_{j=1}^5\mathcal{I}_{1,j}.
\end{aligned}    
\end{equation*}
Before we proceed with the study of $\mathcal{I}_1$, we observe that
\begin{equation*}
\begin{aligned}
\left[L;f\right]g=\left[|\nabla_x|;f\right]g.     
\end{aligned}    
\end{equation*}
Additionally, given $s\in (0,1)$, by using the integral representation of the fractional Laplacian (see, \cite[Lemma 3.2]{DiPalatucciValdinoci2012}), one has $|\nabla_x|^{2s}(\chi_A)=\frac{1}{A^{2s}}\big(|\nabla|^{2s}\chi\big)_A$, and it follows
\begin{equation*}
 \begin{aligned}
|\big(|\nabla|^{2s}\chi\big)(x_1,x_2)|= &c\Big| \lim_{r\to 0^{+}}\int_{|(z_1,z_2)|\geq r}\frac{\chi(x_1+z_1)+\chi(x_1-z_1)-2\chi(x_1)}{|(z_1,z_2)|^{2+2s}}\, dz_1 dz_2 \Big|\\
\lesssim & \|\chi''\|_{L^{\infty}}\int_{|(z_1,z_2)|\leq 1} \frac{|z_1|^2}{|(z_1,z_2)|^{2+2s}}\, dz_1 dz_2 \\
&+\|\chi\|_{L^{\infty}}\int_{|(z_1,z_2)|\geq 1} \frac{1}{|(z_1,z_2)|^{2+2s}}\, dz_1 dz_2 \\
\lesssim &\|\chi''\|_{L^{\infty}}+\|\chi\|_{L^{\infty}},
 \end{aligned}   
\end{equation*}
where we used a second order Taylor expansion to deduce $|\chi(x_1+z_1)+\chi(x_1-z_1)-2\chi(x_1)|\lesssim \|\chi''\|_{L^{\infty}}|z_1|^2$. Thus, we arrive at
\begin{equation}\label{Fderweight}
 \|\,|\nabla_x|^{2s}(\chi_A)\|_{L^{\infty}}\lesssim \frac{1}{A^{2s}}.  
\end{equation}
Now, we can control the factors $\mathcal{I}_{1,j}$, $j=1,\dots,5$. Using Proposition \ref{commutatorestim1}, the fact that $\mathcal{R}_1$ is a bounded operator in $L^2$ and \eqref{Fderweight}, we deduce
\begin{equation*}
\begin{aligned}
|\mathcal{I}_{1,2}|\leq & \|\left[|\nabla_x|;\chi_A\right]\Lambda Q\|_{L^2}\|\epsilon\|_{L^2}\\
\lesssim &\big(\|\left[|\nabla_x|;\chi_A\right]\Lambda Q-\frac{1}{A}\chi_A'\mathcal{R}_1 \Lambda Q \|_{L^2}+\frac{1}{A}\|\chi_A'\mathcal{R}_1 \Lambda Q \|_{L^2}\big)\|\epsilon\|_{L^2}\\
\lesssim &\big( \|\,|\nabla_x|(\chi_A)\|_{L^{\infty}}\|\Lambda Q\|_{L^2}+\frac{1}{A}\|\Lambda Q\|_{L^2}\big)\|\epsilon\|_{L^2}\\
\lesssim & \frac{1}{A}\|\epsilon\|_{L^2}.
\end{aligned}    
\end{equation*}
Applying a similar estimate to \eqref{Jesteq} and \eqref{phidiffere}, together with H"older's inequality, it follows that 
\begin{equation*}
\begin{aligned}
|\mathcal{I}_{1,1}|&+|\mathcal{I}_{1,3}|+|\mathcal{I}_{1,4}|+|\mathcal{I}_{1,5}|\\
&\lesssim \frac{\|\epsilon\|_{L^2}\|x_1 Q\|_{L^2}}{A}+\frac{\|\epsilon\|_{L^2}\|F\chi_A'\|_{L^2}}{A}+\frac{\|\epsilon\|_{L^2}\|Q\|_{L^2}\|F\chi_A'\|_{L^{\infty}}}{A}\\
&\lesssim \Big(\frac{1}{A^{\frac{1}{2}}}+\frac{1}{A}\Big)\|\epsilon\|_{L^2}.
\end{aligned}    
\end{equation*}
This completes the study of $\mathcal{I}_1$ and the proof of the lemma.
\end{proof}
\subsection{Control of Parameters} 
In this part, we obtain estimates for $\epsilon$ in terms of the initial condition $\epsilon(x,0)$. 

Using our mass and energy relations in Lemma \ref{Energyconsrlemma}, we can follow the arguments in \cite[Lemmas 10 and 11]{MartelMerle2001} and \cite[Lemma 6.3 and 6.4]{FarahHolmerRoudenko2019} to deduce the following two lemmas:
\begin{lem}\label{lemmacontrolparam1}
Let $\epsilon_0(x)=\epsilon(x,0)$, $E_0=E[Q+\epsilon_0]$. Recall the definition  $\mathcal{M}_0$ in \eqref{Masszero}. Then it follows
\begin{equation*}
 \Big|\mathcal{M}_0-2\int_{\mathbb{R}^2} \epsilon_0 Q \mathrm{d}x\Big|+\Big|E_0+\int_{\mathbb{R}^2} \epsilon_0 Q\, \mathrm{d}x\Big|+\Big|E_0+\frac{1}{2}\mathcal{M}_0\Big|\lesssim \|\epsilon_0\|_{H^{\frac{1}{2}}}^2+\|\epsilon_0\|_{H^{\frac{1}{2}}}^3,
\end{equation*}
where the implicit constant is independent of $\int_{\mathbb{R}^2} \epsilon_0 Q\, \mathrm{d}x$, $E_0$, $\epsilon_0$.
\end{lem}

\begin{lem}\label{boundenessL}
There exists $\omega_1>0$ such that if $\|\epsilon(s)\|_{H^{\frac{1}{2}}}\leq \omega$, $|\lambda(s)-1|<\omega$, and $\epsilon(s)\perp\{\partial_{x_1}Q,\partial_{x_2}Q,\psi_0\}$ for all $s\geq 0$, where $0<\omega<\omega_1$, then
\begin{equation}\label{eqH1bound}
 \begin{aligned}
  (L\epsilon(s),\epsilon(s))\leq \|\epsilon(s)\|_{H^{\frac{1}{2}}}^2\lesssim    \|\epsilon_0\|_{H^{\frac{1}{2}}}^2+\|\epsilon_0\|_{H^{\frac{1}{2}}}^3+\omega \Big|\int_{\mathbb{R}^2} \epsilon_0 Q\, \mathrm{d}x\Big|,   
 \end{aligned}   
\end{equation}
where the implicit constant is independent of $\epsilon$, $\omega$, and $Q$.
\end{lem}
\subsection{Proof of instability} \label{sectionInstab3d}
We follow the strategy used in \cite{MartelMerle2001,FarahHolmerRoudenko2019,FarahHolmerRoudenko2019KdV}. The novelty of this manuscript is the deduction of the monotonicity estimates in  Corollary \ref{propdecayIMPR}, with which we do not need to use point decay estimates as in previous works.

Let us first mention some preliminary consequences of our results in the previous subsections. For simplicity, we set $c=1$. For $n\in \mathbb{Z}^{+}$ to be chosen later,  we define
\begin{equation*}
    u_0^n=Q+\epsilon_0^n,
\end{equation*}
where
\begin{equation}\label{inicondtins}
    \epsilon_0^n=\frac{1}{n}(Q+a\psi_0),
\end{equation}
with $a\in \mathbb{R}$ be given by
\begin{equation*}
   a=-\frac{\int_{\mathbb{R}^2} \psi_0 Q \, \mathrm{d}x}{\|\psi_0\|_{L^2}^2}.
\end{equation*}
Then by Theorem \ref{propL}, we have that for every $n\in \mathbb{Z}^{+}$,
\begin{equation*}
    \epsilon_0^n\perp \{\partial_{x_1}Q,\partial_{x_2}Q,\psi_0\},
\end{equation*}
with $u_0^n$ being radial, and thus cylindrical symmetric. 

Let us assume \emph{by contradiction that $Q$ is stable}. From this assumption, the following holds:
\begin{itemize}
    \item[(i)] Let $0<\omega<\omega_1$, where $\omega_1$ is such that Lemma \ref{modparaL} holds. Taking $n$ large, the stability assumption of $Q$ implies that the corresponding solution $u^n(t)\in C([0,\infty);H^{\frac{1}{2}}(\mathbb{R}^2))$ of \eqref{E:HBO} with initial condition $u_0^n$ satisfies $u^n(t)\in U_{\omega}$, and there exists $C^1$ functions $\lambda^n(t)$, $z_1^n(t)$ such that if $\epsilon^n(t)$ is defined as in Definition \ref{modparamdef} with $z(t)=(z_1(t),0,0)$, it follows
\begin{equation*}
    \epsilon^n(t)\perp \{\partial_{x_1}Q,\partial_{x_2}Q,\psi_0\},
\end{equation*}
with $\lambda^n(0)=1$, and $z_1^n(0)=0$.  To simplify notation, \underline{we omit the index $n$ in what follows}. 

Moreover, rescaling time $t\mapsto s$ by $\frac{d s}{dt}=\frac{1}{\lambda^2}$, Lemma \ref{modparaL} implies that $\lambda(s)$, $z_1(s)$ are $C^1$ functions, and $\epsilon(s)$ satisfies the equation in Lemma \ref{EquationEpsi}
\item[(ii)] From Proposition \ref{propmodpar} and the fact that $u(t)\in U_{\omega}$, we have
\begin{equation}\label{smalleepsicond}
    \|\epsilon(s)\|_{H^{\frac{1}{2}}}\leq C_1\omega, \quad \text{ and }\quad |\lambda(s)-1|\leq C_1 \omega.
\end{equation}
Thus, taking $0<\omega<1$ small, we can assume
\begin{equation}\label{smalleepsicond2}
     \frac{1}{2}\leq  \lambda(s)\leq \frac{3}{2},
\end{equation}
for all $s\geq 0$. Moreover, by $\eqref{smalleepsicond}$, up to taking $\omega>0$ small, we can assume that Lemma \ref{boundenessL} holds. That is,
\begin{equation}\label{H1/2estimate}
 \|\epsilon(s)\|_{H^{\frac{1}{2}}}^2\lesssim    \|\epsilon_0\|_{H^{\frac{1}{2}}}^2+\|\epsilon_0\|_{H^{\frac{1}{2}}}^3+\omega \Big|\int_{\mathbb{R}^2} \epsilon_0 Q\, \mathrm{d}x\Big|,      
\end{equation}
for all $s\geq 0$.

\item[(iii)] By \eqref{boundparams}, if $\omega>0$ is sufficiently small
\begin{equation}\label{smalleepsicond2.1}
\begin{aligned}
  \Big|\frac{1}{\lambda}\frac{d \lambda}{ds}\Big|+ \Big|\frac{1}{\lambda}\frac{d z_1}{ds}-1\Big|\leq  C_2(C_1\omega+C_1^2\omega^2)\leq 2C_2 C_1 \omega.    
\end{aligned}
\end{equation}
Since $\frac{d z_1}{dt}=\frac{1}{\lambda^2}\frac{d z_1}{d s}$, we have
\begin{equation*}
    \frac{1-C_2C_1\omega}{\lambda}\leq \frac{d z_1}{dt}\leq \frac{1+C_2C_1\omega}{\lambda}.
\end{equation*}
Thus, taking $\omega>0$ small and using that $|\lambda(s)-1|\leq C_1 \omega$, we can assume
\begin{equation}\label{Smallnessdiff}
 \Big|\frac{1}{\lambda}\frac{d z_1}{ds}-1\Big|\leq \frac{1}{2},  
\end{equation}
and
\begin{equation}\label{smalleepsicond3}
    \frac{3}{4}\leq \frac{d z_1}{dt} \leq \frac{5}{4}.
\end{equation}
Then $z_1(t)$ is an increasing function and by mean value inequality
\begin{equation*}
    z_1(t_0)-z_1(t)\geq \frac{3}{4}(t_0-t)
\end{equation*}
for every $t_0,t\geq 0$ with $t\in [0,t_0]$. Since $z_1(0)=0$, it follows
\begin{equation*}
    z_1(t)\geq \frac{1}{2}t,
\end{equation*}
for all $t\geq 0$. 

\item[(iv)] Given that $\frac{ds}{dt}=\frac{1}{\lambda^2}$, we combine \eqref{boundparams}, \eqref{eqH1bound} and \eqref{smalleepsicond2} to obtain
\begin{equation}\label{newvesteq1}
\begin{aligned}
  \Big|\frac{1}{\lambda}\frac{d \lambda}{dt}\Big|+ \Big|\frac{1}{\lambda}\frac{d z_1}{dt}-\frac{1}{\lambda^2}\Big|=&\frac{1}{\lambda^2}\Big|\frac{1}{\lambda}\frac{d \lambda}{ds}\Big|+ \frac{1}{\lambda^2}\Big|\frac{1}{\lambda}\frac{d z_1}{ds}-1\Big|\\
  \lesssim &  \Big(\|\epsilon_0\|_{H^{\frac{1}{2}}}^2+\|\epsilon_0\|_{H^{\frac{1}{2}}}^3+ \omega\Big|\int_{\mathbb{R}^2} \epsilon_0 Q\, \mathrm{d}x\Big|\Big)^{\frac{1}{2}},
\end{aligned}
\end{equation}
and thus $\Big|\frac{1}{\lambda}\frac{d \lambda}{ds}\Big|$, $\Big|\frac{1}{\lambda}\frac{d z_1}{ds}-1\Big|$ also satisfy the same upper bound above.
\end{itemize}
\subsection{Monotonicity formulas}

Let us deduce some key monotonicity formulas. Consider
\begin{equation}\label{defieta}
 \eta(x_1,x_2,t)=   \lambda^{-1}\epsilon\Big(\frac{x_1-z_1}{\lambda},\frac{x_2}{\lambda},t\Big),
\end{equation}
that is
\begin{equation}\label{eqeta}
    \eta(x_1,x_2,t)=u(x_1,x_2,t)-Q_{\lambda,z_1}(x_1,x_2),
\end{equation}
where we use the notation
\begin{equation*}
   f_{\lambda,z_1}(x_1,x_2)=\lambda^{-1}f\Big(\frac{x_1-z_1}{\lambda},\frac{x_2}{\lambda}\Big). 
\end{equation*}
Rescaling time $t\mapsto s$ by $\frac{ds}{dt}=\frac{1}{\lambda^2}$, changing variables, and using \eqref{eqH1bound}, we deduce
\begin{equation}\label{eqH1boundeta}
 \begin{aligned}
  \|\eta(t) \|_{H^{\frac{1}{2}}}^2\lesssim  \Big(\|\epsilon_0\|_{H^{\frac{1}{2}}}^2+\|\epsilon_0\|_{H^{\frac{1}{2}}}^3+\omega \Big|\int_{\mathbb{R}^2} \epsilon_0 Q\, \mathrm{d}x\Big|\Big),   
 \end{aligned}   
\end{equation}
for all $t\geq 0$.

Next, we deduce the equation for $\eta$.  Using \eqref{eqeta}, the fact that $u$ solves \eqref{E:HBO} and that $Q$ solves \eqref{EQ:groundState} with $\alpha=1$ and $n=2$, it follows
\begin{equation}\label{eqforetatime}
\begin{aligned}
  \partial_t \eta=&\partial_{x_1}|\nabla_x|\eta+ \frac{1}{\lambda} \frac{d\lambda}{dt}(\Lambda Q)_{\lambda,z_1}+\Big(\frac{1}{\lambda}\frac{dz_1}{dt}-\frac{1}{\lambda^2}\Big)(\partial_{x_1}Q)_{\lambda,z_1} -\frac{1}{2}\partial_{x_1}\big(\eta^2+2\eta Q_{\lambda,z_1}\big).
\end{aligned}  
\end{equation}
Therefore, we note that equation \eqref{eqforetatime} in terms of the parameter $s\geq 0$ is given by
\begin{equation}\label{eqforeta}
\begin{aligned}
  \partial_s \eta =&\lambda^{2}\partial_{x_1}|\nabla_x|\eta+ \frac{1}{\lambda} \frac{d\lambda}{ds}(\Lambda Q)_{\lambda,z_1}+\Big(\frac{1}{\lambda}\frac{dz_1}{ds}-1\Big)(\partial_{x_1}Q)_{\lambda,z_1} -\frac{ \lambda^2}{2}\partial_{x_1}\big(\eta^2+2\eta Q_{\lambda,z_1}\big).
\end{aligned}  
\end{equation}
It is worth mentioning that the equations above can be justified in negative Sobolev spaces as $H^{-4}(\mathbb{R}^2)$. 
 
Let $M>0$ to be chosen later, $\gamma\in (\frac{1}{2}, \frac{3}{2}]$. We will work with the function \eqref{p2} with $\sigma=(1,0)$, and variable $\frac{x}{M}\in \mathbb{R}^2$. For clarity and simplicity in notation, and given that the context is now well established, we will adopt the following notation in this section
\begin{equation}\label{psidef}
   \varphi(x_1,x_2)=\int_{-\infty}^{\frac{x_1}{M}}\frac{dr}{\langle r  \rangle^{2\gamma}}, \quad (x_1,x_2)\in \mathbb{R}^2,
\end{equation}
and
\begin{equation*}
\phi(x_1,x_2)=\partial_{x_1}\varphi(x_1,x_2)=\frac{1}{M} \Big\langle \frac{x_1}{M}\Big\rangle^{-2\gamma}.  
\end{equation*}
Note that that $\phi$ above is related to the square of \eqref{p1}.

For $x_0, t_0>0$, and $t\in [0,t_0]$, we also consider
\begin{equation}
  \mathcal{J}_{x_0,t_0}(t):=\int_{\mathbb{R}^2} \eta^2(x_1,x_2,t)\varphi(x_1-z_1(t_0)+\frac{1}{2}(t_0-t)-x_0,x_2)\, \mathrm{d}x. 
\end{equation}
The functional $ \mathcal{J}_{x_0,t_0}(t)$ is useful to study the decay of the mass of the function $\eta$ to the right of the soliton, we refer to \cite{KenigMartel2009} and references therein.

Next, we state our first monotonicity result for the function $\eta$.  
\begin{lem}[Almost monotonicity]\label{almostmontlemmaeta}
Let $1<\gamma\leq \frac{3}{2}$. Consider $\omega>0$  be small such that \eqref{smalleepsicond}-\eqref{newvesteq1} hold for the functions  $\epsilon(t)$, $z_1(t)$, and $\lambda(t)$, and such that the conclusion of Lemma \ref{boundenessL} hold. Then, there exists $M>0$, sufficiently large and independent of $\omega$, such that, by taking $\omega>0$ smaller if necessary, there exists a constant $C=C(M)>0$ such that for all $x_0>0$ and $t_0,t\geq 0$ with $t\in [0,t_0]$, it follows
 \begin{equation*}
     \mathcal{J}_{x_0,t_0}(t_0)-\mathcal{J}_{x_0,t_0}(t)\leq 
     C \Big(\|\epsilon_0\|_{H^{\frac{1}{2}}}^2+\|\epsilon_0\|_{H^{\frac{1}{2}}}^3+\omega \Big|\int_{\mathbb{R}^2} \epsilon_0 Q\, \mathrm{d}x\Big|\Big)\bigg(\frac{1}{|x_0|^{2\gamma-2}}+\frac{1}{|x_0|^{1^{-}}}\bigg).
 \end{equation*}
\end{lem}
\begin{proof}
By equation \eqref{eqforetatime}, we deduce
\begin{equation}\label{energyestimJ}
 \begin{aligned}
  \frac{d}{dt}\mathcal{J}_{x_0,t_0}(t)=& 2\int_{\mathbb{R}^2} \partial_t \eta \eta \, \varphi\, \mathrm{d}x -\frac{1}{2}\int_{\mathbb{R}^2} \eta^2 \phi\, \mathrm{d}x\\ 
  =&\underbrace{2 \int_{\mathbb{R}^2} \big(\partial_{x_1}|\nabla_x|\eta\big) \, \eta \, \varphi\, \mathrm{d}x -\frac{1}{2}\int_{\mathbb{R}^2} \eta^2 \phi\, \mathrm{d}x}_{\mathcal{A}_1} +\underbrace{\frac{2}{\lambda}\frac{d\lambda}{dt}\int_{\mathbb{R}^2} (\Lambda Q)_{\lambda,z_1} \, \eta \, \varphi \, \mathrm{d}x}_{\mathcal{A}_2}\\
  &+\underbrace{2\Big(\frac{1}{\lambda}\frac{d z_1}{dt}-\frac{1}{\lambda^2}\Big)\int_{\mathbb{R}^2} (\partial_{x_1}Q)_{\lambda,z_1}\, \eta\, \varphi \, \mathrm{d}x}_{\mathcal{A}_3} -\underbrace{\int_{\mathbb{R}^2} \partial_{x_1}\big(\eta^2+2\eta Q_{\lambda,z_1}\big)\eta \varphi\, \mathrm{d}x}_{\mathcal{A}_4}.
 \end{aligned}   
\end{equation}
Let us estimate the terms $\mathcal{A}_j$, $j=1,2,3,4$ above. Using the commutator estimate in  Corollary \ref{Maincorollary} (recall that $\phi=\partial_{x_1}\varphi$), there exists a universal constant $c>0$ such that
\begin{equation*}
\begin{aligned}
2\int_{\mathbb{R}^2} \big(\partial_{x_1}|\nabla_x|\eta \big) \, \eta \, \varphi\,  & \mathrm{d}x-\frac{1}{2}\int_{\mathbb{R}^2} \eta^2 \phi\, \mathrm{d}x\\
\leq & -\int_{\mathbb{R}^2} \Big(|\nabla_x|^{\frac{1}{2}}\big(\eta \sqrt{\phi}\big)\Big)^2\, \mathrm{d}x -\Big(\frac{1}{2}-\frac{c}{M}\Big)\int_{\mathbb{R}^2} \eta^2 \phi \, \mathrm{d}x.
\end{aligned}    
\end{equation*}
Hence, taking $M>4c$, we have
\begin{equation*}
\begin{aligned}
\mathcal{A}_1 \leq  -\int_{\mathbb{R}^2} \bigg( \Big(|\nabla_x|^{\frac{1}{2}}\big(\eta \sqrt{\phi}\big)\Big)^2  +\frac{1}{4} \eta^2 \phi\bigg) \, \mathrm{d}x.
\end{aligned}    
\end{equation*}
To estimate $\mathcal{A}_2$, we decompose $\mathbb{R}^2$ into the following sets
\begin{equation*}
\begin{aligned}
&\Omega_1:=\Big\{(x_1,x_2)\in \mathbb{R}^2: x_1-z_1(t)\leq \frac{1}{8}(t_0-t)+\frac{x_0}{2}\Big\},\\
&\Omega_2:=\Big\{(x_1,x_2)\in \mathbb{R}^2: x_1-z_1(t)\geq \frac{1}{8}(t_0-t)+\frac{x_0}{2}\Big\}.\\
\end{aligned}    
\end{equation*}
From the mean value inequality and \eqref{smalleepsicond3}, we have for $(x_1,x_2)\in \Omega_1$ that
\begin{equation*}
\begin{aligned}
x_1-z_1(t_0)+\frac{1}{2}(t_0-t)-x_0=&(x_1-z_1(t))-(z_1(t_0)-z(t))+\frac{1}{2}(t_0-t)-x_0\\
\leq & (x_1-z_1(t))-\frac{1}{4}(t_0-t)-x_0\\
\leq &-\frac{1}{8}(t_0-t)-\frac{x_0}{2}<0.
\end{aligned}    
\end{equation*}
Thus, given $(x_1,x_2)\in \Omega_1$, it follows
\begin{equation*}
\begin{aligned}
|\varphi(x_1-z_1(t_0)+\frac{1}{2}(t_0-t)-x_0,x_2)| \lesssim & \big| x_1-z_1(t_0)+\frac{1}{2}(t_0-t)-x_0 \big|^{-(2\gamma-1)}  \\
\lesssim & \Big| \frac{1}{8}(t_0-t)+\frac{x_0}{2}\Big|^{-(2\gamma-1)}.
\end{aligned}    
\end{equation*}
Consequently, from the previous estimates, \eqref{newvesteq1} and \eqref{eqH1bound}, we arrived at
\begin{equation}\label{A2estm1}
\begin{aligned}
\Big|\frac{2}{\lambda}\frac{d\lambda}{dt}&\int_{\Omega_1} (\Lambda Q)_{\lambda,z_1} \, \eta \, \varphi \, \mathrm{d}x\Big|\\
\lesssim & \Big| \frac{1}{\lambda}\frac{d\lambda}{dt}\Big| \Big|\frac{1}{8}(t_0-t)-\frac{x_0}{2}\Big|^{-(2\gamma-1)} \|(\Lambda Q)_{\lambda,z_1}\|_{L^2}\|\eta\|_{L^2}\\
\lesssim & \Big(\|\epsilon_0\|_{H^{\frac{1}{2}}}^2+\|\epsilon_0\|_{H^{\frac{1}{2}}}^3+\omega \Big|\int_{\mathbb{R}^2} \epsilon_0 Q\, \mathrm{d}x\Big|\Big)\Big|\frac{1}{8}(t_0-t)+\frac{x_0}{2}\Big|^{-(2\gamma-1)},
\end{aligned}    
\end{equation}
 where we have used that $\|(\Lambda Q)_{\lambda,z_1}\|_{L^2}\lesssim 1$ provided the decay properties of $Q$ and its derivatives (see \eqref{Qdecay} with $n=2$, $\alpha=1$). Now, when $(x_1,x_2)\in \Omega_2$, the properties of $Q$ imply
\begin{equation}\label{decayestimate}
\begin{aligned}
|(\Lambda Q)_{\lambda,z_1}(x_1,x_2)|\lesssim & \Big\langle \Big( \frac{x_1-z_1(t)}{\lambda},\frac{x_2}{\lambda}\Big) \Big\rangle^{-3}\\
\lesssim & \Big\langle\frac{1}{8}(t_0-t)+\frac{x_0}{2}\Big\rangle^{-2^{-}}\Big\langle \Big( \frac{x_1-z_1(t)}{\lambda},\frac{x_2}{\lambda}\Big) \Big\rangle^{-1^{+}}. 
\end{aligned}    
\end{equation}
Hence, using that $\langle (x_1,x_2) \rangle^{-1^{+}}\in L^{2}(\mathbb{R}^2)$, we get
\begin{equation}\label{A2estm2}
\begin{aligned}
\Big|\frac{2}{\lambda}\frac{d\lambda}{dt}&\int_{\Omega_2} (\Lambda Q)_{\lambda,z_1} \, \eta \, \varphi \, \mathrm{d}x\Big|\\
\lesssim & \Big| \frac{1}{\lambda}\frac{d\lambda}{dt}\Big| \Big \langle\frac{1}{8}(t_0-t)-\frac{x_0}{2}\Big\rangle^{-2^{-}} \|\eta\|_{L^2}\\
\lesssim & \Big(\|\epsilon_0\|_{H^{\frac{1}{2}}}^2+\|\epsilon_0\|_{H^{\frac{1}{2}}}^3+\omega \Big|\int_{\mathbb{R}^2} \epsilon_0 Q\, \mathrm{d}x\Big|\Big)\Big\langle\frac{1}{8}(t_0-t)+\frac{x_0}{2}\Big\rangle^{-2^{-}}.
\end{aligned}    
\end{equation}
The estimate for $\mathcal{A}_2$ is now a consequence of \eqref{A2estm1} and \eqref{A2estm2}.

On the other hand, given that $\partial_{x_1}Q$ decays faster than $\Lambda Q$, we have that the same arguments in the estimate for $\mathcal{A}_2$ above, and the bound of $\Big|\frac{1}{\lambda}\frac{d z_1}{dt}-1\Big|$ in \eqref{newvesteq1} yield
\begin{equation*}
 \begin{aligned}
|\mathcal{A}_3|\lesssim & \Big(\|\epsilon_0\|_{H^{\frac{1}{2}}}^2+\|\epsilon_0\|_{H^{\frac{1}{2}}}^3+\omega \Big|\int_{\mathbb{R}^2} \epsilon_0 Q\, \mathrm{d}x\Big|\Big)\\
 & \times \Big(\Big|\frac{1}{8}(t_0-t)+\frac{x_0}{2}\Big|^{-(2\gamma-1)}+\Big\langle\frac{1}{8}(t_0-t)+\frac{x_0}{2}\Big\rangle^{-2^{-}}\Big).
 \end{aligned}   
\end{equation*}
Next, we study $\mathcal{A}_4$. Using integration by parts and that $\partial_{x_1}\varphi=\phi$, we find
\begin{equation*}
\begin{aligned}
\mathcal{A}_4=&\int_{\mathbb{R}^2} \eta^2\partial_{x_1} \eta \, \varphi\, \mathrm{d}x  +2\int_{\mathbb{R}^2} Q_{\lambda,z_1} \eta \partial_{x_1} \eta \, \varphi\, \mathrm{d}x+\int_{\mathbb{R}^2} \eta^3 \phi\, \mathrm{d}x+2\int_{\mathbb{R}^2} Q_{\lambda,z_1}\eta^2 \phi\, \mathrm{d}x\\
=&\frac{2}{3}\int_{\mathbb{R}^2} \eta^3 \, \phi \, \mathrm{d}x+\int_{\mathbb{R}^2} Q_{\lambda,z_1}\eta^2 \phi\, \mathrm{d}x -\int_{\mathbb{R}^2} \partial_{x_1} \big( Q_{\lambda,z_1}\big) \eta^2 \varphi\, \mathrm{d}x\\
=:& \mathcal{A}_{4,1}+\mathcal{A}_{4,2}+\mathcal{A}_{4,3}.
\end{aligned}    
\end{equation*}
Writing $\eta^3 \phi=\eta(\eta \sqrt{\phi})^{2}$, applying H\"older inequality and Sobolev embedding $H^{\frac{1}{2}}(\mathbb{R}^2)\hookrightarrow L^4(\mathbb{R}^2)$, we deduce that there exists a constant $c_1>0$ such that
\begin{equation}\label{A41Estimate}
\begin{aligned}
|\mathcal{A}_{4,1}|\lesssim & \|\eta\|_{L^2}\|\eta \sqrt{\phi}\|_{L^4}^2\\
\lesssim & \Big(\|\epsilon_0\|_{H^{\frac{1}{2}}}^2+\|\epsilon_0\|_{H^{\frac{1}{2}}}^3+\omega \Big|\int_{\mathbb{R}^2} \epsilon_0 Q\, \mathrm{d}x\Big|\Big)^{\frac{1}{2}}\|\eta \sqrt{\phi}\|_{H^{\frac{1}{2}}}^2\\
\leq & c_1\Big(\|\epsilon_0\|_{H^{\frac{1}{2}}}^2+\|\epsilon_0\|_{H^{\frac{1}{2}}}^3+\omega \Big|\int_{\mathbb{R}^2} \epsilon_0 Q\, \mathrm{d}x\Big|\Big)^{\frac{1}{2}}\Big(\||\nabla_x|^{\frac{1}{2}}\big(\eta \sqrt{\phi}\big)\|_{L^2}^2+\|\eta \sqrt{\phi}\|_{L^2}^2\Big),
\end{aligned}    
\end{equation}
where we have applied \eqref{eqH1boundeta}. Now, the idea is to take $\|\epsilon_0\|_{H^{\frac{1}{2}}}$, and $\omega>0$ small such that one can absorb the estimate for $\mathcal{A}_{4,1}$ for that of $\mathcal{A}_1$. Taking $\omega>0$ small (thus $n$ large in \eqref{inicondtins}), we can assume that 
\begin{equation*}
 c_1\Big(\|\epsilon_0\|_{H^{\frac{1}{2}}}^2+\|\epsilon_0\|_{H^{\frac{1}{2}}}^3+\omega \Big|\int_{\mathbb{R}^2} \epsilon_0 Q\, \mathrm{d}x\Big|\Big)^{\frac{1}{2}}\leq \frac{1}{8}.   
\end{equation*}
Then we get
\begin{equation*}
\begin{aligned}
\mathcal{A}_1+\mathcal{A}_{4,1} \leq & -\int_{\mathbb{R}^2} \bigg( \frac{7}{8}\Big(|\nabla_x|^{\frac{1}{2}}\big(\eta \sqrt{\phi}\big)\Big)^2 +\frac{1}{8} \eta^2 \phi\bigg) \, \mathrm{d}x\leq  0.    
\end{aligned}
\end{equation*}
Next, we deal with $\mathcal{A}_{4,2}$, and $\mathcal{A}_{4,3}$. Decomposing $\mathbb{R}^2=\Omega_1\cup \Omega_2$ as above, and following previous arguments, one has that for $(x_1,x_2)\in \Omega_1$,
\begin{equation*}
|\phi(x_1,x_2)| \lesssim \Big\langle\frac{1}{8}(t_0-t)+\frac{x_0}{2}\Big\rangle^{-2\gamma},
\end{equation*}
and for $(x_1,x_2)\in \Omega_2$, it follows
\begin{equation*}
\begin{aligned}
|Q_{\lambda,z_1}(x_1,x_2)|
\lesssim & \Big\langle\frac{1}{8}(t_0-t)+\frac{x_0}{2}\Big\rangle^{-3}.
\end{aligned}    
\end{equation*}
Thus, similar ideas in \eqref{A2estm1} and \eqref{A2estm2} allow us to conclude
\begin{equation*}
 \begin{aligned}
|\mathcal{A}_{4,2}|\lesssim & \|\eta\|_{L^2}^2 \Big(\Big\langle\frac{1}{8}(t_0-t)+\frac{x_0}{2}\Big\rangle^{-2\gamma}+\Big\langle\frac{1}{8}(t_0-t)+\frac{x_0}{2}\Big\rangle^{-3}\Big)\\
\lesssim & \Big(\|\epsilon_0\|_{H^{\frac{1}{2}}}^2+\|\epsilon_0\|_{H^{\frac{1}{2}}}^3+\omega \Big|\int_{\mathbb{R}^2} \epsilon_0 Q\, \mathrm{d}x\Big|\Big)\\
 & \times \Big(\Big\langle\frac{1}{8}(t_0-t)+\frac{x_0}{2}\Big\rangle^{-2\gamma}+\Big\langle\frac{1}{8}(t_0-t)+\frac{x_0}{2}\Big\rangle^{-3}\Big).
 \end{aligned}   
\end{equation*}
Similarly, we find
\begin{equation*}
 \begin{aligned}
|\mathcal{A}_{4,3}|\lesssim & \Big(\|\epsilon_0\|_{H^{\frac{1}{2}}}^2+\|\epsilon_0\|_{H^{\frac{1}{2}}}^3+\omega \Big|\int_{\mathbb{R}^2} \epsilon_0 Q\, \mathrm{d}x\Big|\Big)\\
 & \times \Big(\Big|\frac{1}{8}(t_0-t)+\frac{x_0}{2}\Big|^{-(2\gamma-1)}+\Big\langle\frac{1}{8}(t_0-t)+\frac{x_0}{2}\Big\rangle^{-3}\Big).
 \end{aligned}   
\end{equation*}
Collecting the previous estimates, we deduce
\begin{equation}\label{Fexpression}
 \begin{aligned}
  \frac{d}{dt}\mathcal{J}_{x_0,t_0}(t)\lesssim &\Big(\|\epsilon_0\|_{H^{\frac{1}{2}}}^2+\|\epsilon_0\|_{H^{\frac{1}{2}}}^3+\omega \Big|\int_{\mathbb{R}^2} \epsilon_0 Q\, \mathrm{d}x\Big|\Big)\\
 & \times \Big(\Big|\frac{1}{8}(t_0-t)+\frac{x_0}{2}\Big|^{-(2\gamma-1)}+\Big\langle\frac{1}{8}(t_0-t)+\frac{x_0}{2}\Big\rangle^{-2^{-}}\Big).
 \end{aligned}   
\end{equation}
Now, given $t\in (0,t_0)$, a change of variable reveals
\begin{equation*}
\begin{aligned}
 \int_t^{t_0} \Big|\frac{1}{8}(t_0-\tau)+\frac{x_0}{2}\Big|^{-(2\gamma-1)}\, d\tau \lesssim & \frac{1}{|x_0|^{2\gamma-2}},   
\end{aligned}
\end{equation*}
and
\begin{equation*}
\begin{aligned}
 \int_t^{t_0} \Big\langle\frac{1}{8}(t_0-\tau)+\frac{x_0}{2}\Big\rangle^{-2^{-}} \, d\tau \lesssim & \frac{1}{|x_0|^{1^{-}}}.   
\end{aligned}
\end{equation*}
Consequently, integrating \eqref{Fexpression} between $[t,t_0]$, we obtain the desired inequality. 
\end{proof}
\begin{cor}\label{propdecayeta}
Let $\omega>0$, $1<\gamma\leq \frac{3}{2}$ and $M>0$ be such that the hypothesis of Lemma \ref{almostmontlemmaeta} hold true. Then there exists $C=C(M)>0$ such that for every $t\geq 0$ and $x_0>0$
    \begin{equation}\label{eqineta}
     \begin{aligned}
       \int_{\mathbb{R}}\int_{x_1>x_0} & \eta^2(x_1+z_1(t),x_2,t)\, \mathrm{d}x_1  \mathrm{d}x_2\\
       \leq&  \frac{C}{n^2|x_0|^{2\gamma-1}} +C \Big(\|\epsilon_0\|_{H^{\frac{1}{2}}}^2+\|\epsilon_0\|_{H^{\frac{1}{2}}}^3+\omega \Big|\int_{\mathbb{R}^2} \epsilon_0 Q\, \mathrm{d}x\Big|\Big)\bigg(\frac{1}{|x_0|^{2\gamma-2}}+\frac{1}{|x_0|^{1^{-}}}\bigg).
     \end{aligned}   
    \end{equation}
\end{cor}
\begin{proof}
Applying Lemma \ref{almostmontlemmaeta} with $ \mathcal{J}_{x_0,t}(t)$ and $\mathcal{J}_{x_0,t}(0)$, and that $\lambda(0)=1$, and $z_1(0)=0$, we find that there exists a constant $C>0$ such that
\begin{equation}\label{almosteq1}
 \begin{aligned}
  \int_{\mathbb{R}^2} \eta^2(x_1,x_2,t)&\varphi(x_1-z_1(t)-x_0,x_2)\, \mathrm{d}x\\
  \leq &  \int_{\mathbb{R}^2} \epsilon^2_0(x_1,x_2)\varphi(x_1-z_1(t)+\frac{1}{2}t-x_0,x_2)\, \mathrm{d}x\\
  &+C \Big(\|\epsilon_0\|_{H^{\frac{1}{2}}}^2+\|\epsilon_0\|_{H^{\frac{1}{2}}}^3+\omega \Big|\int_{\mathbb{R}^2} \epsilon_0 Q\, \mathrm{d}x\Big|\Big)\bigg(\frac{1}{|x_0|^{2\gamma-2}}+\frac{1}{|x_0|^{1^{-}}}\bigg).
 \end{aligned}   
\end{equation}
Setting $c(\gamma):=\int_{-\infty}^0 \frac{dr}{\langle r\rangle^{2\gamma}}$,  note that $\varphi(x_1,x_2)\geq c(\gamma)$ for all $x_1\geq 0$, $x_2\in \mathbb{R}$. Thus, changing variables, and using the definition of $\eta$, we deduce
\begin{equation}\label{almosteq2}
 \begin{aligned}
\int_{\mathbb{R}^2} \eta^2(x_1,x_2,t)\varphi(x_1-z_1(t)-x_0,z_2)\, \mathrm{d}x=&\int_{\mathbb{R}^2} \eta^2(x_1+z_1(t),x_2,t)\varphi(x_1-x_0,x_2)\, \mathrm{d}x_1  \mathrm{d}x_2 \\
\geq &c(\gamma)\int_{\mathbb{R}}\int_{x_1>x_0}\eta^2(x_1+z_1(t),x_2,t)\, \mathrm{d}x_1  \mathrm{d}x_2.
 \end{aligned}   
\end{equation}
On the other hand, since $z_1(t)-\frac{1}{2}t\geq 0$, and $x_1\mapsto \varphi(x_1,x_2)$ is an increasing function, it follows
\begin{equation*}
\begin{aligned}
 \int_{\mathbb{R}^2} \epsilon^2_0(x_1,x_2) &\varphi(x_1-z_1(t)+\frac{1}{2}t-x_0,x_2)\, \mathrm{d}x \\
 \leq &\int_{\mathbb{R}^2} \epsilon^2_0(x_1,x_2)\varphi(x_1-x_0,x_2)\, \mathrm{d}x\\
 =& \int_{\{x_1\leq \frac{x_0}{2}\}} \epsilon^2_0(x_1,x_2)\varphi(x_1-x_0,x_2)\, \mathrm{d}x+\int_{\{x_1\geq \frac{x_0}{2}\}} \epsilon^2_0(x_1,x_2)\varphi(x_1-x_0,x_2)\, \mathrm{d}x.
\end{aligned}    
\end{equation*}
When $x_1\leq \frac{x_0}{2}$, we have $x_1-x_0\leq -\frac{x_0}{2}<0$ (i.e., $|x_1-x_0|\geq \frac{x_0}{2}$), then from the definition of $\varphi$
\begin{equation*}
|\varphi(x_1-x_0,x_2)|\lesssim    |x_1-x_0|^{-(2\gamma-1)}\lesssim   |x_0|^{-(2\gamma-1)}.  
\end{equation*}
Hence, recalling $\epsilon_0$ as in \eqref{inicondtins}, we get
\begin{equation*}
\begin{aligned}
\int_{\{x_1\leq \frac{x_0}{2}\}} \epsilon^2_0(x_1,x_2)\varphi(x_1-x_0,x_2)\, \mathrm{d}x\lesssim \frac{1}{n^2} | x_0|^{-(2\gamma-1)}\big(\|Q\|_{L^2}^2+\|\psi_0\|_{L^2}^2\big).   
\end{aligned}    
\end{equation*}
On the other hand, when $x_1\geq \frac{x_0}{2}$, the spatial decay of $Q$ and $\psi_0$ establish
\begin{equation}\label{decayE0}
| \epsilon_0(x_1,x_2)| \lesssim \frac{1}{n}\langle (x_1,x_2)\rangle^{-3}\lesssim \frac{1}{n} \langle x_0 \rangle^{-2^{-}} \langle (x_1,x_2) \rangle^{-1^{+}}.  
\end{equation}
Hence, we get
\begin{equation*}
\begin{aligned}
\int_{\{x_1\geq \frac{x_0}{2}\}} \epsilon^2_0(x_1,x_2)\varphi(x_1-x_0,x_2)\, \mathrm{d}x\lesssim \frac{1}{n^2}  \langle x_0 \rangle^{-4^{-}}.   
\end{aligned}    
\end{equation*}
In conclusion,
\begin{equation}\label{almosteq3}
 \int_{\mathbb{R}^2} \epsilon^2_0(x_1,x_2) \varphi(x_1-z_1(t)+\frac{1}{2}t-x_0,x_2)\, \mathrm{d}x \lesssim \frac{1}{n^2} |x_0|^{-(2\gamma-1)}.    
\end{equation}
Plugging \eqref{almosteq2} and \eqref{almosteq3} into \eqref{almosteq1} completes the proof of \eqref{eqineta}. 
\end{proof}
Inspired by the results in \cite[Proposition 2]{KenigMartel2009}, we introduce a second monotonicity formula for $\eta$, which aims to improve the decay to the right in \eqref{eqineta}. Nevertheless, the initial decay obtained \eqref{eqineta}  plays an important part in deriving such improvement.

Given $x_0, t_0>0$, and $t\in [0,t_0]$, we consider 
\begin{equation}
  \int_{\mathbb{R}^2} \eta^2(x_1,x_2,t)\rho(x_1,x_2,t)\, \mathrm{d}x, 
\end{equation}
where for fixed $0<\nu <\frac{3}{8}$, setting $\widetilde{x}=x_1-z_1(t)-\nu(t_0-t)-x_0$, $x^{\ast}=-\nu(t_0-t)-x_0$, we define
\begin{equation*}
 \begin{aligned}
\rho(x_1,x_2,t):=\varphi(\widetilde{x},x_2)-\varphi(x^{\ast},x_2).    
 \end{aligned}   
\end{equation*}
Formally, $\rho$ is useful to establish better decay estimates for $\eta$ in the region $x_1\gtrsim x_0$ in contrast with \eqref{eqineta}.

\begin{lem}\label{Monforeta}
Let $0<\nu <\frac{3}{8}$ fixed. Consider $\omega>0$ and $M>0$ be such that Corollary \ref{propdecayeta} holds with $\gamma=\frac{3}{2}$. Then, by taking $M>0$ larger (independent of $\omega$), if necessary, there exists $C=C(M)>0$ such that for all $x_0\geq M$, $t_0,t\geq 0$ with $t\in [0,t_0]$
 \begin{equation}\label{Newmont1}
 \begin{aligned}
    \int_{\mathbb{R}^2}\eta^2(x_1,x_2,t_0) \rho(x_1,x_2,t_0) \, \mathrm{d}x\leq & \int_{\mathbb{R}^2}\eta^2(x_1,x_2,t) \rho(x_1,x_2,t) \, \mathrm{d}x \\
    &+C\Big(\frac{1}{n^2}+\|\epsilon_0\|_{H^{\frac{1}{2}}}^2+\|\epsilon_0\|_{H^{\frac{1}{2}}}^3+\omega \Big|\int_{\mathbb{R}^2} \epsilon_0 Q\, \mathrm{d}x\Big|\Big)\frac{1}{|x_0|^{\frac{3}{2}^{-}}}.
\end{aligned}
\end{equation}
\end{lem}
\begin{proof}
In the following, we consider the function $\varphi$ in \eqref{psidef} with $\gamma=\frac{3}{2}$.  Using the equation \eqref{eqforetatime}, that $\partial_{x_1}\varphi=\phi$, and $\partial_{x_1} \rho(x,t)=\phi(\widetilde{x},x_2)$, one gets
\begin{equation*}
   \begin{aligned}
\frac{d}{dt} \int_{\mathbb{R}^2} \eta^2 \rho \, \mathrm{d}x=&2\int_{\mathbb{R}^2} \partial_t \eta  \eta\,  \rho\, \mathrm{d}x-\Big(\frac{d z_1}{dt}-\nu\Big)\int_{\mathbb{R}^2} \eta^2 \phi(\widetilde{x})\, \mathrm{d}x-\nu \int_{\mathbb{R}^2} \eta^2 \phi(x^{\ast})\, \mathrm{d}x\\
=:&\mathcal{B}_1+\mathcal{B}_2+\mathcal{B}_3+\mathcal{B}_4,
\end{aligned}   
\end{equation*}
where
\begin{equation*}
\begin{aligned}
\mathcal{B}_1=&2 \int_{\mathbb{R}^2} \big(\partial_{x_1}|\nabla_x|\eta \big)\, \eta \, \rho \, \mathrm{d}x -\Big(\frac{d z_1}{dt}-\nu\Big)\int_{\mathbb{R}^2} \eta^2 \phi(\widetilde{x},x_2)\, \mathrm{d}x-\nu \int_{\mathbb{R}^2} \eta^2 \phi(x^{\ast},x_2)\, \mathrm{d}x, \\
\mathcal{B}_2=&\frac{2}{\lambda}\frac{d\lambda}{dt}\int_{\mathbb{R}^2} (\Lambda Q)_{\lambda,z_1} \, \eta \, \rho \, \mathrm{d}x,\\
\mathcal{B}_3=&2\Big(\frac{1}{\lambda}\frac{d z_1}{dt}-\frac{1}{\lambda^2}\Big)\int_{\mathbb{R}^2} (\partial_{x_1}Q)_{\lambda,z_1}\, \eta\, \rho \, \mathrm{d}x,\\
\mathcal{B}_4=&\frac{2}{3}\int_{\mathbb{R}^2} \eta^3 \phi (\widetilde{x},x_2)\, \mathrm{d}x+\int_{\mathbb{R}^2} Q_{\lambda,z_1} \eta^2 \phi(\widetilde{x},x_2)\, \mathrm{d}x-\int_{\mathbb{R}^2} \partial_{x_1}(Q_{\lambda,z_1})\eta^2 \rho\, \mathrm{d}x.
\end{aligned}    
\end{equation*}
The commutator estimate in Corollary \ref{Maincorollary} and \eqref{Smallnessdiff} yield
\begin{equation*}
\begin{aligned}
\mathcal{B}_1=&2 \int_{\mathbb{R}^2} \big(\partial_{x_1}|\nabla_x|\eta \big)\, \eta \, \varphi(\widetilde{x},x_2) \, \mathrm{d}x -\Big(\frac{d z_1}{dt}-\nu\Big)\int_{\mathbb{R}^2} \eta^2 \phi(\widetilde{x},x_2)\, \mathrm{d}x\\
&-\nu \int_{\mathbb{R}^2} \eta^2 \phi(x^{\ast},x_2)\, \mathrm{d}x  \\
\leq & -\int_{\mathbb{R}^2} \Big(|\nabla_x|^{\frac{1}{2}}\big(\eta \sqrt{\phi(\widetilde{x},x_2)}\big)\Big)^2\, \mathrm{d}x -\Big(\frac{3}{4}-\nu-\frac{c}{M}\Big)\int_{\mathbb{R}^2} \eta^2 \phi(\widetilde{x},x_2) \, \mathrm{d}x\\
&-\nu \int_{\mathbb{R}^2} \eta^2 \phi(x^{\ast},x_2)\, \mathrm{d}x.
\end{aligned}    
\end{equation*}
Thus, when $0<\nu<\frac{3}{8}$, taking $M>0$ large such that $\frac{c}{M}<\nu$, one deduces
\begin{equation*}
\begin{aligned}
\mathcal{B}_1\leq & -\int_{\mathbb{R}^2} \Big(|\nabla_x|^{\frac{1}{2}}\big(\eta \sqrt{\phi(\widetilde{x},x_2)}\big)\Big)^2\, \mathrm{d}x -\Big(\frac{3}{4}-2\nu\Big)\int_{\mathbb{R}^2} \eta^2 \phi(\widetilde{x},x_2) \, \mathrm{d}x\\
&-\nu \int_{\mathbb{R}^2} \eta^2 \phi(x^{\ast},x_2)\, \mathrm{d}x.
\end{aligned}    
\end{equation*}
To estimate $\mathcal{B}_2$, we decompose $\mathbb{R}^2$ into the following sets
\begin{equation*}
\begin{aligned}
&\Omega_1:=\Big\{(x_1,x_2)\in \mathbb{R}^2: |x_1-z_1(t)|\leq \frac{\nu}{2}(t_0-t)+\frac{x_0}{2}\Big\},\\
&\Omega_2:=\Big\{(x_1,x_2)\in \mathbb{R}^2: |x_1-z_1(t)|\geq \frac{\nu}{2}(t_0-t)+\frac{x_0}{2}\Big\}.\\
\end{aligned}    
\end{equation*}
We also write $\Omega_2=\Omega_2^{+}+\Omega_2^{-}$ where 
\begin{equation*}
\begin{aligned}
&\Omega_2^{+}=\left\{(x_1,x_2)\in \mathbb{R}^2: x_1-z_1(t)\geq \frac{\nu}{2}(t_0-t)+\frac{x_0}{2}\right\},\\    
&\Omega_2^{-}=\left\{(x_1,x_2)\in \mathbb{R}^2: x_1-z_1(t)\leq -\frac{\nu}{2}(t_0-t)-\frac{x_0}{2}\right\}.
\end{aligned}
\end{equation*}
We observe that for $(x_1,x_2)\in \Omega_2$, the decay properties of $Q$ imply that \eqref{decayestimate} also holds in this case. Moreover, when $(x_1,x_2)\in\Omega_2^{-}$, we have $\widetilde{x}\leq  -\frac{3\nu}{2}(t_0-t)-\frac{3x_0}{2}<0$, then 
\begin{equation*}
\begin{aligned}
|\varphi(\widetilde{x},x_2)| \lesssim |\nu(t_0-t)+x_0|^{-2},
\end{aligned}    
\end{equation*}
and since $\varphi(x^{\ast})$ satisfies the same estimate above, we conclude that for $(x_1,x_2)\in \Omega_2^{-}$ it follows
\begin{equation*}
  |\rho(x_1,x_2,t)|\lesssim |\nu(t_0-t)+x_0|^{-2}.
\end{equation*}
Hence, using that $\langle (x_1,x_2) \rangle^{-1^{+}}\in L^{2}(\mathbb{R}^2)$, we infer from Cauchy–Schwarz inequality that
\begin{equation*}
\begin{aligned}
\Big|\frac{2}{\lambda}\frac{d\lambda}{dt}\int_{\Omega_2^{-}} (\Lambda Q)_{\lambda,z_1} \, \eta \, \rho \, \mathrm{d}x\Big|\lesssim & \Big| \frac{1}{\lambda}\frac{d\lambda}{dt}\Big|\|\eta\|_{L^2} \Big \langle\nu(t_0-t)-x_0\Big\rangle^{-2^{-}} |\nu(t_0-t)+x_0|^{-2} \\
\lesssim & \Big(\|\epsilon_0\|_{H^{\frac{1}{2}}}^2+\|\epsilon_0\|_{H^{\frac{1}{2}}}^3+\omega \Big|\int_{\mathbb{R}^2} \epsilon_0 Q\, \mathrm{d}x\Big|\Big)|\nu(t_0-t)+x_0|^{-4^{-}}.
\end{aligned}    
\end{equation*}
Now, since $\nu(t_0-t)+x_0\geq x_0\geq M$, using again Cauchy–Schwarz inequality and  \eqref{eqineta} with $\gamma=\frac{3}{2}$ it follows 
\begin{equation*}
\begin{aligned}
\Big|\frac{2}{\lambda}\frac{d\lambda}{dt}&\int_{\Omega_2^{+}} (\Lambda Q)_{\lambda,z_1} \, \eta \, \rho \, \mathrm{d}x  \Big|\\
\lesssim &\Big|\frac{1}{\lambda}\frac{d\lambda}{dt}\Big|\Big(\int_{x_1-z_1(t)\geq \frac{\nu}{2}(t_0-t)+\frac{x_0}{2}} (\Lambda Q)_{\lambda,z_1}^2\, \mathrm{d}x\Big)^{\frac{1}{2}}\Big(\int_{x_1-z_1(t)\geq \frac{\nu}{2}(t_0-t)+\frac{x_0}{2}} \eta^2(x_1,x_2)\, \mathrm{d}x\Big)^{\frac{1}{2}}\\
\lesssim & \Big(\frac{1}{n^2}+\|\epsilon_0\|_{H^{\frac{1}{2}}}^2+\|\epsilon_0\|_{H^{\frac{1}{2}}}^3+\omega \Big|\int_{\mathbb{R}^2} \epsilon_0 Q\, \mathrm{d}x\Big|\Big)\langle \nu (t_0-t)-x_0 \rangle^{-2^{-}}|\nu(t_0-t)+x_0|^{-\frac{1}{2}^{-}}\\
\lesssim & \Big(\frac{1}{n^2}+\|\epsilon_0\|_{H^{\frac{1}{2}}}^2+\|\epsilon_0\|_{H^{\frac{1}{2}}}^3+\omega \Big|\int_{\mathbb{R}^2} \epsilon_0 Q\, \mathrm{d}x\Big|\Big)|\nu(t_0-t)+x_0|^{-\frac{5}{2}^{-}}.
\end{aligned}    
\end{equation*}
Gathering the previous estimates over the sets $\Omega_2^{+}$ and $\Omega_2^{-}$, we arrive at
\begin{equation*}
\begin{aligned}
\Big|\frac{2}{\lambda}\frac{d\lambda}{dt}\int_{\Omega_2} (\Lambda Q)_{\lambda,z_1} \, \eta \, \rho \, \mathrm{d}x  \Big|\lesssim  \Big(\frac{1}{n^2}+\|\epsilon_0\|_{H^{\frac{1}{2}}}^2+\|\epsilon_0\|_{H^{\frac{1}{2}}}^3+\omega \Big|\int_{\mathbb{R}^2} \epsilon_0 Q\, \mathrm{d}x\Big|\Big)|\nu(t_0-t)+x_0|^{-\frac{5}{2}^{-}}.
\end{aligned}    
\end{equation*}
On the other hand, when $(x_1,x_2)\in \Omega_1$, one has
\begin{equation*}
|\widetilde{x}-x^{\ast}|=|x_1-z_1(t)|,
\end{equation*}
the inequality
\begin{equation*}
 \begin{aligned}
\widetilde{x}\leq -\frac{\nu}{2}(t_0-t)-\frac{x_0}{2}<0,
 \end{aligned}
\end{equation*}
and $|\widetilde{x}|\sim (\nu(t_0-t)+x_0)$. Then, when $(x_1,x_2)\in \Omega_1$, the mean value inequality yields
\begin{equation*}
\begin{aligned}
 |\rho(x_1,x_2,t)| \lesssim & |x_1-z_1(t)|\sup_{s\sim (\nu(t_0-t)+x_0)} |\phi(s,x_2)|  \\
\lesssim & |x_1-z_1(t)|\langle \nu(t_0-t)+x_0\rangle^{-3}.
\end{aligned}    
\end{equation*}
Consequently, Cauchy–Schwarz inequality allows us to conclude
\begin{equation*}
\begin{aligned}
\Big|\frac{2}{\lambda}\frac{d\lambda}{dt}\int_{\Omega_1} (\Lambda Q)_{\lambda,z_1} \, \eta \, \varphi \, \mathrm{d}x\Big|\lesssim &  \Big| \frac{1}{\lambda}\frac{d\lambda}{dt}\Big| \|(x\Lambda Q)_{\lambda,z_1}\|_{L^2} \|\eta\|_{L^2} \Big \langle\nu(t_0-t)-x_0\Big\rangle^{-3}\\
\lesssim & \Big(\|\epsilon_0\|_{H^{\frac{1}{2}}}^2+\|\epsilon_0\|_{H^{\frac{1}{2}}}^3+\omega \Big|\int_{\mathbb{R}^2} \epsilon_0 Q\, \mathrm{d}x\Big|\Big)\Big\langle\nu(t_0-t)+x_0\Big\rangle^{-3}.
\end{aligned}    
\end{equation*}
Summarizing, it follows
\begin{equation*}
 |\mathcal{B}_2| \lesssim  \Big(\frac{1}{n^2}+\|\epsilon_0\|_{H^{\frac{1}{2}}}^2+\|\epsilon_0\|_{H^{\frac{1}{2}}}^3+\omega \Big|\int_{\mathbb{R}^2} \epsilon_0 Q\, \mathrm{d}x\Big|\Big)\Big|\nu(t_0-t)+x_0\Big|^{-\frac{5}{2}^{-}}. 
\end{equation*}
Similarly, the decay properties of $(\partial_{x_1}Q)_{\lambda,z_1}$ and the same arguments developed in the estimate for $\mathcal{B}_2$ above show
\begin{equation*}
 |\mathcal{B}_3| \lesssim  \Big(\frac{1}{n^2}+\|\epsilon_0\|_{H^{\frac{1}{2}}}^2+\|\epsilon_0\|_{H^{\frac{1}{2}}}^3+\omega \Big|\int_{\mathbb{R}^2} \epsilon_0 Q\, \mathrm{d}x\Big|\Big)\Big|\nu(t_0-t)+x_0\Big|^{-\frac{5}{2}^{-}}. 
\end{equation*}
Let us deal with $\mathcal{B}_4$. Writing $\eta^3 \phi=\eta(\eta \sqrt{\phi})^{2}$, and arguing as in \eqref{A41Estimate}, there exists a constant $c_1>0$ such that
\begin{equation*}
\begin{aligned}
\Big|\frac{2}{3}\int_{\mathbb{R}^2} & \eta^3 \phi(\widetilde{x},x_2)\, \mathrm{d}x\Big|\\
\leq & c_1\Big(\|\epsilon_0\|_{H^{\frac{1}{2}}}^2+\|\epsilon_0\|_{H^{\frac{1}{2}}}^3+\omega \Big|\int_{\mathbb{R}^2} \epsilon_0 Q\, \mathrm{d}x\Big|\Big)^{\frac{1}{2}}\Big(\||\nabla_x|^{\frac{1}{2}}\big(\eta \sqrt{\phi(\widetilde{x},x_2)}\big)\|_{L^2}^2+\|\eta \sqrt{\phi(\widetilde{x},x_2)}\|_{L^2}^2\Big).
\end{aligned}    
\end{equation*}
Next, we can assume that 
\begin{equation*}
 c_1\Big(\|\epsilon_0\|_{H^{\frac{1}{2}}}^2+\|\epsilon_0\|_{H^{\frac{1}{2}}}^3+\omega \Big|\int_{\mathbb{R}^2} \epsilon_0 Q\, \mathrm{d}x\Big|\Big)^{\frac{1}{2}}\leq \frac{1}{2}\Big(\frac{3}{4}-2\nu\Big)<\frac{3}{8}.   
\end{equation*}
It follows then
\begin{equation*}
\begin{aligned}
\mathcal{B}_1+\frac{2}{3}\int_{\mathbb{R}^2} & \eta^{3}\phi(\widetilde{x},x_2)\mathrm{d}x \\
\leq & -\int_{\mathbb{R}^2} \bigg( \frac{5}{8}\Big(|\nabla_x|^{\frac{1}{2}}\big(\eta \sqrt{\phi(\widetilde{x},x_2)}\big)\Big)^2 +\frac{1}{2}\Big(\frac{3}{4}-2\nu\Big) \eta^2 \phi(\widetilde{x},x_2)\bigg) \, \mathrm{d}x\\
&-\nu \int_{\mathbb{R}^2} \eta^2 \phi(x^{\ast},x_2)\, \mathrm{d}x\\
\leq & 0.    
\end{aligned}
\end{equation*}
We proceed with the estimate of the remaining two terms that conform $\mathcal{B}_4$. Using the same decomposition of $\mathbb{R}^2$ into the sets $\Omega_1$ and $\Omega_2$ previously considered, we notice that $|\widetilde{x}
|\sim (\nu(t_0-t)+x_0)$ in $\Omega_1$ implies
\begin{equation*}
|\phi(\widetilde{x},x_2)| \sim \langle \nu(t_0-t)+x_0 \rangle^{-3}, 
\end{equation*}
whenever $(x_1,x_2)\in \Omega_1$. As a consequence, we can argue as in the study of $\mathcal{B}_2$ above to get  
\begin{equation*}
 \begin{aligned}
    \Big|\int_{\mathbb{R}^2} Q_{\lambda,z_1} \eta^2 & \phi(\widetilde{x},x_2)\, \mathrm{d}x \Big|+\Big|\int_{\mathbb{R}^2} \partial_{x_1}(Q_{\lambda,z_1})\eta^2 \rho\, \mathrm{d}x\Big|\\
\lesssim & \Big(\frac{1}{n^2}+\|\epsilon_0\|_{H^{\frac{1}{2}}}^2+\|\epsilon_0\|_{H^{\frac{1}{2}}}^3+\omega \Big|\int_{\mathbb{R}^2} \epsilon_0 Q\, \mathrm{d}x\Big|\Big)|\nu(t_0-t)+x_0|^{-\frac{5}{2}^{-}}. 
\end{aligned}
\end{equation*}
Collecting all the previous estimates, we arrived at the differential inequality
\begin{equation*}
 \begin{aligned}
    \frac{d}{dt}\int_{\mathbb{R}^2} \eta^2 \rho \, \mathrm{d}x \lesssim & \Big(\frac{1}{n^2}+\|\epsilon_0\|_{H^{\frac{1}{2}}}^2+\|\epsilon_0\|_{H^{\frac{1}{2}}}^3+\omega \Big|\int_{\mathbb{R}^2} \epsilon_0 Q\, \mathrm{d}x\Big|\Big)\Big|\nu(t_0-t)+x_0\Big|^{-\frac{5}{2}^{-}}. 
\end{aligned}
\end{equation*}
By integration over $[t,t_0]$, the previous expression give us \eqref{Newmont1}.
\end{proof}
Finally, we can establish decay estimates for $\epsilon(s)$, $s\geq 0$. 
\begin{cor}\label{propdecayIMPR}
Let Let $0<\nu <\frac{3}{8}$ fixed. Consider $\omega>0$ and $M>0$ such that Lemma \ref{Monforeta} holds. Then for every $s\geq 0$ and $x_0\geq M$
    \begin{equation}\label{eqineqepsIMPR}
     \begin{aligned}
       \int_{\mathbb{R}}\int_{x_1\geq 4x_0} \epsilon^2(x_1,x_2,s)\, \mathrm{d}x_1  \mathrm{d}x_2\lesssim   \Big(\frac{1}{n^2}+\|\epsilon_0\|_{H^{\frac{1}{2}}}^2+\|\epsilon_0\|_{H^{\frac{1}{2}}}^3+\omega \Big|\int_{\mathbb{R}^2} \epsilon_0 Q\, \mathrm{d}x\Big|\Big)\frac{1}{|x_0|^{\frac{3}{2}^{-}}}.
     \end{aligned}   
    \end{equation}
Moreover, for $1<m<\frac{3}{2}$, it follows 
\begin{equation}\label{eqineqepsIMPR2}
\begin{aligned} \int_{\mathbb{R}}\int_{0}^{\infty} &\langle x_1\rangle^{m}\epsilon^2(x_1,x_2,s)\, \mathrm{d}x_1  \mathrm{d}x_2 \lesssim & \frac{1}{n^2}+\|\epsilon_0\|_{H^{\frac{1}{2}}}^2+\|\epsilon_0\|_{H^{\frac{1}{2}}}^3+\omega \Big|\int_{\mathbb{R}^2} \epsilon_0 Q\, \mathrm{d}x\Big|.
\end{aligned}    
\end{equation}
Where the implicit constants above depend on $M$.
\end{cor}
\begin{proof}
We deduce \eqref{eqineqepsIMPR} and \eqref{eqineqepsIMPR2} for the time variable $t$ as the result for variable $s\geq 0$ follows from the rescaling relation $\frac{ds}{dt}=\frac{1}{\lambda^2}$ and that $\lambda \sim 1$.  

We first consider \eqref{eqineqepsIMPR}, and take $t_0\geq 0$, and $t=0$. Using that $\lambda(0)=1$, $z_1(0)=0$, and that $\varphi\geq 0$ with $x_1\mapsto \varphi(x_1,x_2)$ being an increasing function, we deduce from \eqref{almosteq3} that
\begin{equation}\label{Newmont2}
\begin{aligned}
 \int_{\mathbb{R}^2}\eta^2(x_1,x_2,0) \rho(x_1,x_2,0) \, \mathrm{d}x=&\int_{\mathbb{R}^2} \epsilon_0^2(x_1,x_2)\big(\varphi(x_1-\nu t_0-x_0,x_2)-\varphi(-\nu t_0-x_0,x_2)\big)\, \mathrm{d}x \\
\leq & \int_{\mathbb{R}^2} \epsilon_0^2(x_1,x_2)\varphi(x_1-x_0,x_2)\, \mathrm{d}x\\
\lesssim & \frac{1}{n^2}|x_0|^{-2}.
\end{aligned}    
\end{equation}
Changing variables and using the definition of $\eta$, we have
\begin{equation}\label{Newmont3}
\begin{aligned}
  \int_{\mathbb{R}^2}\eta^2(x_1,x_2,t_0) &\rho(x_1,x_2,t_0) \, \mathrm{d}x\\
  =& \int_{\mathbb{R}^2}\eta^2(x_1+z_1(t_0),x_2,t_0)\big(\varphi(x_1-x_0,x_2)-\varphi(-x_0,x_2)\big) \, \mathrm{d}x \\
  =& \frac{1}{\lambda^2(t_0)}\int_{\mathbb{R}^2} \epsilon^2\Big(\frac{x_1}{\lambda(t_0)},\frac{x_2}{\lambda(t_0)},t_0\Big)\big(\varphi(x_1-x_0,x_2)-\varphi(-x_0,x_2)\big) \, \mathrm{d}x.
\end{aligned}    
\end{equation}
Plugging \eqref{Newmont2} and \eqref{Newmont3} into \eqref{Newmont1}, we arrive at 
\begin{equation}\label{estimatediff1}
\begin{aligned}\frac{1}{\lambda^2(t_0)} & \int_{\mathbb{R}^2} \epsilon^2\Big(\frac{x_1}{\lambda(t_0)},\frac{x_2}{\lambda(t_0)},t_0\Big)\big(\varphi(x_1-x_0,x_2)-\varphi(-x_0,x_2)\big) \, \mathrm{d}x\\
\lesssim & \frac{1}{n^2}|x_0|^{-2}+ \Big(\frac{1}{n^2}+\|\epsilon_0\|_{H^{\frac{1}{2}}}^2+\|\epsilon_0\|_{H^{\frac{1}{2}}}^3+\omega \Big|\int_{\mathbb{R}^2} \epsilon_0 Q\, \mathrm{d}x\Big|\Big)|x_0|^{-\frac{3}{2}^{-}}.
\end{aligned}    
\end{equation}
Hence, by splitting the integral on the left-hand side above according to the regions $x_1<0$ and $x_1>0$, we obtain 
\begin{equation}\label{estimatediff2}
\begin{aligned}
\frac{1}{\lambda^2(t_0)} & \int_{\mathbb{R}}\int_{x_1>0} \epsilon^2\Big(\frac{x_1}{\lambda(t_0)},\frac{x_2}{\lambda(t_0)},t_0\Big)\big(\varphi(x_1-x_0,x_2)-\varphi(-x_0,x_2)\big) \, \mathrm{d}x_1  \mathrm{d}x_2\\
\lesssim & \frac{1}{\lambda^2(t_0)} \int_{\mathbb{R}}\int_{x_1<0} \epsilon^2\Big(\frac{x_1}{\lambda(t_0)},\frac{x_2}{\lambda(t_0)},t_0\Big)\big(\varphi(-x_0,x_2)-\varphi(x_1-x_0,x_2)\big) \, \mathrm{d}x_1  \mathrm{d}x_2\\
&+ \frac{1}{n^2}|x_0|^{-2}+ \Big(\frac{1}{n^2}+\|\epsilon_0\|_{H^{\frac{1}{2}}}^2+\|\epsilon_0\|_{H^{\frac{1}{2}}}^3+\omega \Big|\int_{\mathbb{R}^2} \epsilon_0 Q\, \mathrm{d}x\Big|\Big)|x_0|^{-\frac{3}{2}^{-}}.
\end{aligned}    
\end{equation}
Now, when $x_1<0$, the definition of $\varphi$ establishes 
\begin{equation*}
\begin{aligned}
0\leq  \varphi(-x_0,x_2)-\varphi(x_1-x_0,x_2)=\int_{\frac{x_1-x_0}{M}}^{\frac{-x_0}{M}}\frac{dr}{\langle r\rangle^{3}}\lesssim |x_0|^{-2}.
\end{aligned}    
\end{equation*}
Consequently, \eqref{H1/2estimate} and changing variables give us
\begin{equation}\label{estimatediff2.1}
\begin{aligned}
\frac{1}{\lambda^2(t_0)} &\int_{\mathbb{R}}\int_{x_1<0} \epsilon^2\Big(\frac{x_1}{\lambda(t_0)},\frac{x_2}{\lambda(t_0)},t_0\Big)\big(\varphi(-x_0,x_2)-\varphi(x_1-x_0,x_2)\big) \, \mathrm{d}x_1  \mathrm{d}x_2 \\
\lesssim& |x_0|^{-2}\|\epsilon(t_0)\|_{L^2}^2\\
\lesssim& \Big(\|\epsilon_0\|_{H^{\frac{1}{2}}}^2+\|\epsilon_0\|_{H^{\frac{1}{2}}}^3+\omega \Big|\int_{\mathbb{R}^2} \epsilon_0 Q\, \mathrm{d}x\Big|\Big)|x_0|^{-2}.
\end{aligned}    
\end{equation}
Next, when $x_0\geq M$, and $x_1\geq 2x_0$, we observe
\begin{equation*}
\begin{aligned}
\varphi(x_1-x_0,x_2)-\varphi(-x_0,x_2)=\int^{\frac{x_1-x_0}{M}}_{\frac{-x_0}{M}}\frac{dr}{\langle r\rangle^{3}}\geq & \int^{\frac{x_0}{M}}_{\frac{-x_0}{M}}\frac{dr}{\langle r\rangle^{3}}
\geq \int^{1}_{-1}\frac{dr}{\langle r\rangle^{3}}=:C_1>0.
\end{aligned}
\end{equation*}
Note also that $\varphi(x_1-x_0,x_2)-\varphi(-x_0,x_2)\geq 0$ for all $x_1\geq 0$. From the previous inequality, and changing variables, we find
\begin{equation}\label{estimatediff3}
\begin{aligned}
\frac{1}{\lambda^2(t_0)}  \int_{\mathbb{R}}\int_{x_1>0} \epsilon^2\Big(\frac{x_1}{\lambda(t_0)},\frac{x_2}{\lambda(t_0)},t_0\Big) & \big(\varphi(x_1-x_0,x_2)-\varphi(-x_0,x_2)\big) \, \mathrm{d}x_1  \mathrm{d}x_2\\
\geq & \frac{C_1}{\lambda^2(t_0)} \int_{\mathbb{R}}\int_{x_1\geq 2x_0} \epsilon^2\Big(\frac{x_1}{\lambda(t_0)},\frac{x_2}{\lambda(t_0)},t_0\Big)\, \mathrm{d}x_1  \mathrm{d}x_2\\
\geq & C_1 \int_{\mathbb{R}}\int_{x_1\geq \frac{2x_0}{\lambda(t_0)}} \epsilon^2(x_1,x_2,t_0)\, \mathrm{d}x_1  \mathrm{d}x_2\\
\geq & C_1 \int_{\mathbb{R}}\int_{x_1\geq 4x_0} \epsilon^2(x_1,x_2,t_0)\, \mathrm{d}x_1  \mathrm{d}x_2,
\end{aligned}    
\end{equation}
where we have used that $\frac{1}{\lambda(t_0)}\leq 2$. Combining \eqref{estimatediff2}, \eqref{estimatediff2.1}, and \eqref{estimatediff3}, we obtain \eqref{eqineqepsIMPR}

Next, we deal with \eqref{eqineqepsIMPR2}. Multiplying \eqref{eqineqepsIMPR} by $|x_0|^{m-1}$ for some $1<m<\frac{3}{2}$ fixed and integrating over $[M,\infty)$ the $x_0$-variable, we deduce
\begin{equation*}
\begin{aligned} \int_{M}^{\infty}\int_{\mathbb{R}}\int_{x_1\geq 4x_0} &|x_0|^{m-1}\epsilon^2(x_1,x_2,t_0)\, \mathrm{d}x_1  \mathrm{d}x_2  \mathrm{d}x_0\\
\lesssim & \Big(\frac{1}{n^2}+\|\epsilon_0\|_{H^{\frac{1}{2}}}^2+\|\epsilon_0\|_{H^{\frac{1}{2}}}^3+\omega \Big|\int_{\mathbb{R}^2} \epsilon_0 Q\, \mathrm{d}x\Big|\Big)\Big(\int_{M}^{\infty}|x_0|^{m-1-\frac{3}{2}^{-}}\, \mathrm{d}x_0\Big).
\end{aligned}    
\end{equation*}
Hence, by using Fubini's theorem and \eqref{H1/2estimate} it follows that
\begin{equation*}
\begin{aligned} \int_{\mathbb{R}}\int_{x_1\geq 4M} &|x_1|^{m}\epsilon^2(x_1,x_2,t_0)\, \mathrm{d}x_1  \mathrm{d}x_2 \lesssim & \frac{1}{n^2}+\|\epsilon_0\|_{H^{\frac{1}{2}}}^2+\|\epsilon_0\|_{H^{\frac{1}{2}}}^3+\omega \Big|\int_{\mathbb{R}^2} \epsilon_0 Q\, \mathrm{d}x\Big|,
\end{aligned}    
\end{equation*}
which together with \eqref{H1/2estimate} and the estimate 
\begin{equation*}
\begin{aligned} \int_{\mathbb{R}}\int_0^{4M} |x_1|^{m}\epsilon^2(x_1,x_2,t_0)\, \mathrm{d}x_1  \mathrm{d}x_2 \leq & (4M)^{m}\|\epsilon(t_0)\|^2_{L^2}\\
\lesssim & \Big(\|\epsilon_0\|_{H^{\frac{1}{2}}}^2+\|\epsilon_0\|_{H^{\frac{1}{2}}}^3+\omega \Big|\int_{\mathbb{R}^2} \epsilon_0 Q\, \mathrm{d}x\Big|\Big),
\end{aligned}    
\end{equation*}
give us \eqref{eqineqepsIMPR2} with $|x_1|^{m}$ instead of $\langle x_1\rangle^{m}$. Nevertheless, the estimate with weight $|x_1|^{m}$ deduce above and \eqref{H1/2estimate}, we obtain \eqref{eqineqepsIMPR2}. 
\end{proof}
We define a scaled version of the viral type functional $J_A$, $A\geq 1$ given by \eqref{Jfunctional}, 
\begin{equation*}
    K_A(s)=\lambda(s)\big(J_A(s)-\kappa\big),
\end{equation*}
where 
\begin{equation*}
\kappa:= \frac{1}{2}\int_{\mathbb{R}} \Big( \int_{\mathbb{R}} \Lambda Q(x_1,x_2)\, \mathrm{d}x_1 \Big)^2\, \mathrm{d}x_2. 
\end{equation*}
The functional $K_A$ is inspired from  \cite{MartelMerle2001}. By \eqref{Jesteq} and \eqref{smalleepsicond}, we have
\begin{equation}\label{upperboundK}
    |K_A(s)|\lesssim (1+A^{\frac{1}{2}})\|\epsilon(s)\|_{L^2}+\kappa<\infty,
\end{equation}
for all $s\geq 0$. Additionally, by Lemma \ref{eqJA}
\begin{equation}\label{eqforK}
 \begin{aligned}
   \frac{d}{ds}K_A(s)=\lambda\Big(1-\Big(\frac{1}{\lambda}\frac{dz_1}{ds}-1\Big)\Big)\int_{\mathbb{R}^2} \epsilon Q\, \mathrm{d}x+\lambda R(\epsilon,A),
 \end{aligned}   
\end{equation}
where $R(\epsilon,A)$ satisfies \eqref{remaindereq}. In the following theorem, we get key lower-bound estimates for $\frac{d}{ds}K_A(s)$.
\begin{thm}\label{theolowerbound}
There exist $\omega_0>0$ sufficiently small, $n\in \mathbb{Z}^{+}$ sufficiently large, and $A\geq 1$ sufficiently large such that
\begin{equation}\label{lowerBound0}
    \frac{d}{ds}K_A(s)\geq \frac{1}{8 n}\bigg(\|Q\|_{L^2}^2-\frac{\big(\int_{\mathbb{R}^2} Q\psi_0 \mathrm{d}x\big)^2}{\|\psi_0\|_{L^2}^2}\bigg).
\end{equation}
\end{thm}
\begin{proof}[Proof of Theorem \ref{theolowerbound}]
By mass conservation Lemma \ref{Energyconsrlemma}, we have $2\int_{\mathbb{R}^2} \epsilon Q\, \mathrm{d}x=\mathcal{M}_0-\int_{\mathbb{R}^2} \epsilon^2\, \mathrm{d}x$, which shows that \eqref{eqforK} can be rewritten as follows
\begin{equation}\label{lowerbound1}
\begin{aligned}
    \frac{d}{ds}K_A(s)=\frac{\lambda}{2}\Big(1-\Big(\frac{1}{\lambda}\frac{dz_1}{ds}-1\Big)\Big)\mathcal{M}_0-\frac{\lambda}{2}\Big(1-\Big(\frac{1}{\lambda}\frac{dz_1}{ds}-1\Big)\Big)\int_{\mathbb{R}^2} \epsilon^2\, \mathrm{d}x+\lambda R(\epsilon,A).
\end{aligned}
\end{equation}
We observe that $\mathcal{M}_0\geq 2\int_{\mathbb{R}^2} \epsilon_0 Q\, \mathrm{d}x$, then from \eqref{inicondtins} and \eqref{Smallnessdiff}, it follows that
\begin{equation}\label{lowerbound2}
 \begin{aligned}
 \frac{\lambda}{2}\Big(1-\Big(\frac{1}{\lambda}\frac{dz_1}{ds}-1\Big)\Big)\mathcal{M}_0\geq & \frac{\lambda}{2 n}\int_{\mathbb{R}^2} (Q+a\psi_0)Q\, \mathrm{d}x \\
=& \frac{1}{4 n}\bigg(\|Q\|_{L^2}^2-\frac{\big(\int_{\mathbb{R}^2} Q\psi_0 \mathrm{d}x\big)^2}{\|\psi_0\|_{L^2}^2}\bigg),
 \end{aligned}   
\end{equation}
where we have also used that $\lambda \geq \frac{1}{2}$. From \eqref{H1/2estimate} and the initial condition \eqref{inicondtins}, we deduce
\begin{equation}\label{ineqepsilondecay}
 \begin{aligned}
  \|\epsilon(s)\|_{H^{\frac{1}{2}}}^2 \lesssim &   \|\epsilon_0\|_{H^{\frac{1}{2}}}^2+\|\epsilon_0\|_{H^{\frac{1}{2}}}^3+\omega \Big|\int_{\mathbb{R}^2} \epsilon_0 Q\, \mathrm{d}x\Big| \\
  \lesssim & \frac{1}{n^2}\big(\|Q+a\psi_0\|_{H^{\frac{1}{2}}}^2+\frac{1}{n}\|Q+a\psi_0\|_{H^{\frac{1}{2}}}^3\big)+ \frac{\omega}{n}\bigg(\|Q\|_{L^2}^2-\frac{\big(\int_{\mathbb{R}^2} Q\psi_0 \mathrm{d}x\big)^2}{\|\psi_0\|_{L^2}^2}\bigg)\\
 \lesssim &\frac{1}{n^2}+\frac{\omega}{n},
 \end{aligned}   
\end{equation}
provided that $n$ is large enough. Thus, it follows
\begin{equation}\label{EpsilonLest1}
\bigg|\frac{\lambda}{2}\Big(1-\Big(\frac{1}{\lambda}\frac{dz_1}{ds}-1\Big)\Big)\int_{\mathbb{R}^2} \epsilon^2\, \mathrm{d}x \bigg|\lesssim    \frac{1}{n^2}+\frac{\omega}{n}.
\end{equation}
We recall that $R(\epsilon,A)$ satisfies \eqref{remaindereq} for some $\gamma_1\in (0,1)$ that we fix from now on. On the other hand, using \eqref{newvesteq1} we get
\begin{equation}
\begin{aligned}
|\lambda R(\epsilon,A)|\lesssim & \Big(1+\frac{1}{A}\Big)\|\epsilon\|^2_{L^2}+\Big(\frac{1}{A}+\frac{1}{A^{\frac{1}{2}}}\Big)\|\epsilon\|_{L^2} \\
&+ \mathcal{G}(\epsilon)\bigg(\frac{1}{A^{\gamma_1}}+\|\epsilon\|_{L^2} +A^{\frac{1}{2}}\|\epsilon\|_{L^2([A,\infty)\times \mathbb{R})}+\Big|\int_{\mathbb{R}^2} \epsilon x_2 \partial_{x_2}F \chi_A\, \mathrm{d}x\Big|\bigg)   \\
&+\mathcal{G}(\epsilon)\bigg(\frac{1}{A}+\Big(1+\frac{1}{A^{\frac{1}{2}}}+\frac{1}{A}\Big)\|\epsilon\|_{L^2}\bigg),
\end{aligned}   
\end{equation}
where from \eqref{ineqepsilondecay}, one has
\begin{equation*}
\begin{aligned}
\mathcal{G}(\epsilon)=&\Big(\|\epsilon_0\|_{H^{\frac{1}{2}}}^2+\|\epsilon_0\|_{H^{\frac{1}{2}}}^3+ \omega\Big|\int_{\mathbb{R}^2} \epsilon_0 Q\, \mathrm{d}x\Big|\Big)^{\frac{1}{2}} \lesssim & \Big(\frac{1}{n^2}+\frac{\omega}{n}\Big)^{\frac{1}{2}}.
\end{aligned}
\end{equation*}
Consequently, we deduce
\begin{equation}\label{LambdaResti}
\begin{aligned}
|\lambda R(\epsilon,A)|\lesssim &
 \Big(1+\frac{1}{A}\Big)\Big(\frac{1}{n^2}+\frac{\omega}{n}\Big)+ \Big(\frac{1}{A}+\frac{1}{A^{\gamma_1}}+\frac{1}{A^{\frac{1}{2}}}\Big)\Big(\frac{1}{n^2}+\frac{\omega}{n}\Big)^{\frac{1}{2}}\\
 &+\Big(\frac{1}{n^2}+\frac{\omega}{n}\Big)^{\frac{1}{2}}\bigg(A^{\frac{1}{2}}\|\epsilon\|_{L^2([A,\infty)\times \mathbb{R})}+\Big|\int_{\mathbb{R}^2} \epsilon x_2 \partial_{x_2}F \chi_A\, \mathrm{d}x\Big|\bigg). 
\end{aligned}    
\end{equation}
Let us estimate the last two terms of the above inequality. We take $A\geq M$ with $A\geq 1$, where $M>0$ is given by Corollary \ref{propdecayIMPR}. Then by \eqref{eqineqepsIMPR}, we get
\begin{equation*}
 \begin{aligned}
 A^{\frac{1}{2}}\|\epsilon\|_{L^2([A,\infty)\times \mathbb{R})}\lesssim   \frac{A^{\frac{1}{2}}}{A^{\frac{3}{4}^{-}}} \Big(\frac{1}{n^2}+\frac{\omega}{n}\Big)^{\frac{1}{2}}\lesssim \frac{1}{A^{\frac{1}{4}^{-}}} \Big(\frac{1}{n^2}+\frac{\omega}{n}\Big)^{\frac{1}{2}}.
 \end{aligned}   
\end{equation*}
Next, we write
\begin{equation*}
\begin{aligned}
\int_{\mathbb{R}^2} \epsilon x_2 \partial_{x_2}F \chi_A\, \mathrm{d}x=&\int_{\mathbb{R}}\int_{-\infty}^0 \epsilon x_2 \partial_{x_2}F \chi_A\, \mathrm{d}x_1  \mathrm{d}x_2+\int_{\mathbb{R}}\int_{0}^{\infty} \epsilon x_2 \partial_{x_2}F \chi_A\, \mathrm{d}x_1  \mathrm{d}x_2.   
\end{aligned}    
\end{equation*}
By using \eqref{decayFinyvar} with $l=1$, $\alpha_1=\alpha_2=\frac{5}{4}$, $\alpha_3=\frac{3}{2}$, Cauchy–Schwarz inequality shows
\begin{equation*}
\begin{aligned}
\Big|\int_{\mathbb{R}}\int_{-\infty}^0 \epsilon x_2 \partial_{x_2}F \chi_A\, \mathrm{d}x_1  \mathrm{d}x_2\Big| \lesssim &  \|\chi_A\|_{L^{\infty}}\int_{\mathbb{R}}\int_{-\infty}^0  |\epsilon| \langle x_1\rangle^{-\frac{5}{4}}\langle x_2\rangle^{-\frac{5}{4}}\, \mathrm{d}x_1  \mathrm{d}x_2\\
\lesssim & \|\epsilon\|_{L^2}\\
\lesssim & \Big(\frac{1}{n^2}+\frac{\omega}{n}\Big)^{\frac{1}{2}}.
\end{aligned}    
\end{equation*}
Using \eqref{decayFinyvar0} with $l=1$, $\alpha_1=\alpha_2=2$, we have $\sup_{x_1\in \mathbb{R}}|\partial_{x_2} F(x_1,x_2)|\lesssim \langle x_2\rangle^{-2}$. Thus, by taking $\frac{1}{2}<m_1<\frac{3}{4}$ fixed, we apply H\"older's inequality and \eqref{eqineqepsIMPR2} with $m=2m_1$ to obtain 
\begin{equation*}
\begin{aligned}
\Big|\int_{\mathbb{R}}&\int_{0}^{\infty} \epsilon x_2 \partial_{x_2}F \chi_A\, \mathrm{d}x_1  \mathrm{d}x_2\Big|\\
\lesssim & \bigg(\int_{\mathbb{R}}\int_0^{\infty} \langle x_1\rangle^{2m_1}\epsilon^2(x_1,x_2,s)\, \mathrm{d}x_1  \mathrm{d}x_2\bigg)^{\frac{1}{2}}\|\langle x_1\rangle^{-m_1}|x_2|\sup_{x_1\in \mathbb{R}}|\partial_{x_2} F(x_1,x_2)|\| _{L^2(\mathbb{R}^2)}\\
\lesssim& \Big(\frac{1}{n^2}+\frac{\omega}{n}\Big)^{\frac{1}{2}}\|\langle x_1\rangle^{-m_1}\langle x_2 \rangle^{-1}\| _{L^2(\mathbb{R}^2)}\\
\lesssim& \Big(\frac{1}{n^2}+\frac{\omega}{n}\Big)^{\frac{1}{2}}.
\end{aligned}    
\end{equation*}
Therefore, inserting the previous inequalities in \eqref{LambdaResti}, we obtain  
\begin{equation}\label{LambdaResti2}
\begin{aligned}
|\lambda R(\epsilon,A)|\lesssim &
 \Big(1+\frac{1}{A}\Big)\Big(\frac{1}{n^2}+\frac{\omega}{n}\Big)+ \Big(\frac{1}{A}+\frac{1}{A^{\gamma_1}}+\frac{1}{A^{\frac{1}{2}}}\Big)\Big(\frac{1}{n^2}+\frac{\omega}{n}\Big)^{\frac{1}{2}}\\
 &+\Big(\frac{1}{n^2}+\frac{\omega}{n}\Big)\Big(1+\frac{1}{A^{\frac{1}{4}^{-}}}\Big)\\
 =:&\widetilde{R}(n,\omega,A).
\end{aligned}    
\end{equation}
Note that $\widetilde{R}(n,w,A)=o(\frac{1}{n})$ as $\omega \to 0^{+}$, and $A\to \infty$. Thus, using also \eqref{EpsilonLest1}, we can take $\omega>0$ small, $n\in \mathbb{Z}^{+}$ and $A\geq 1$ large such that 
\begin{equation}\label{lowerbound3}
\begin{aligned}
 \bigg|\frac{\lambda}{2}\Big(1-\Big(\frac{1}{\lambda}\frac{dz_1}{ds}-1\Big)\Big)\int_{\mathbb{R}^2} \epsilon^2\, \mathrm{d}x \bigg|+|\lambda R(\epsilon,A)|\leq \frac{1}{8 n}\bigg(\|Q\|_{L^2}^2-\frac{\big(\int_{\mathbb{R}^2} Q\psi_0 \mathrm{d}x\big)^2}{\|\psi_0\|_{L^2}^2}\bigg).   
\end{aligned}
\end{equation}
Combining \eqref{lowerbound1}, \eqref{lowerbound2}, and \eqref{lowerbound3} completes the desired estimate.
\end{proof}
\begin{proof}[Conclusion proof of Theorem \ref{MainInsta}]
Integrating \eqref{lowerBound0} between $0$ and $s$ shows
\begin{equation*}
\begin{aligned}
    K_A(s)\geq K_A(0)+\frac{s}{8 n}\bigg(\|Q\|_{L^2}^2-\frac{\big(\int_{\mathbb{R}^2} Q\psi_0 \mathrm{d}x\big)^2}{\|\psi_0\|_{L^2}^2}\bigg).
\end{aligned}    
\end{equation*}
Since $|K_A(0)|\lesssim 1$, the inequality above yields $K_A(s)\to \infty$ as $s\to \infty$. However, this contradicts \eqref{upperboundK}. This  contradiction arise from assuming that $Q$ is stable, thus the instability result in Theorem \ref{MainInsta} must follow. 
\end{proof}
 \section*{Acknowledgment}
The authors thank the Instituto Nacional de Matem\'atica Pura e Aplicada (IMPA) in Rio de Janeiro, Brazil, for their hospitality. They also wish to express their gratitude to the Organizing Committee of the Workshop in Nonlinear PDEs, held at IMPA, Rio de Janeiro, from August 19–23, 2024, where this research was initiated.
{\footnotesize
\bibliographystyle{acm}
\bibliography{Ref-monotonicity}
}
\end{document}